\documentclass[11pt]{article}
\usepackage{amssymb}
\usepackage{amsbsy}
\usepackage[latin1]{inputenc}
\usepackage{amsthm}
\usepackage{graphicx} 
\usepackage{subfigure}
\usepackage{pst-eucl}
\usepackage[latin1]{inputenc}
\usepackage[english]{babel}
\usepackage{amsmath,amssymb,graphics,mathrsfs}
\usepackage{amsmath,amssymb,latexsym}
\usepackage{float}
\usepackage{graphicx,color}
\usepackage[T1]{fontenc}
\usepackage[active]{srcltx}
\usepackage{multicol}
\usepackage[latin1]{inputenc}
\usepackage{pst-all}
\usepackage{enumerate}
\usepackage{pstricks}
\usepackage{pstricks-add}
\usepackage{setspace}
\usepackage{soul}
\usepackage{cancel}
\usepackage{epsfig}
\usepackage{nonfloat}
\usepackage{caption}[=v1]

\usepackage[left=3cm,top=3cm,right=2.4cm,bottom=3.2cm]{geometry}
\usepackage[colorlinks=true,citecolor=red,linkcolor=blue,urlcolor=RubineRed,pdfpagetransition=Blinds,pdftoolbar=false,pdfmenubar=false]{hyperref}

\newcommand{\R}{\hbox{\rm I \kern -5pt R}}     
\newcommand{\p} {\hbox{\rm I \kern -5pt P}}
\def\x  {\boldsymbol x}

\def\Om       {\Omega}

\def\n        {{\boldsymbol n}}

\def\eps        {\varepsilon}

\def \hueco{\noalign{\medskip}}
\def \pato{\forall\,}
\def \beq{\begin{equation}}
\def \eeq{\end{equation}}
\def \ba{\begin{array}}
\def \ea{\end{array}}
\def \dis{\displaystyle}

\newtheorem{obs}{Remark}
\newtheorem{lem}{Lemma}
\newtheorem{coro}{Corollary}

\DeclareMathOperator*{\esssup}{ess\,sup}

\begin{document}

\title{Energy-stable and boundedness preserving numerical schemes for the Cahn-Hilliard equation with degenerate mobility} 

\author{F.~Guill\'en-Gonz\'alez\thanks{Dpto. Ecuaciones Diferenciales y An\'alisis Num\'erico and IMUS, 
Universidad de Sevilla, Facultad de Matem\'aticas, C/ Tarfia, S/N, 41012 Sevilla (SPAIN). Email: guillen@us.es},
~and
 G.~Tierra\thanks{Department of Mathematics, University of North Texas, Denton TX (USA). Email:  gtierra@unt.edu}}

\maketitle


\begin{abstract}
Two new numerical schemes to approximate the Cahn-Hilliard equation with degenerate mobility  (between stable values $0$ and $1$) are presented, by using two different non-centered approximation of the mobility. We prove that both schemes are energy stable and preserve the maximum principle approximately, i.e.~the amount of the solution being outside of the interval $[0,1]$ goes to zero in terms of a truncation parameter. Additionally, we present several numerical results in order to show the accuracy and the well behavior of the proposed schemes, comparing both schemes and the corresponding centered scheme.
\end{abstract}
%
%
\section{Introduction}

The study of interfacial dynamics has become a key component to understand the behavior of a great variety of systems, in scientific, engineering and industrial applications.
A very effective 
approach for representing interface problems is the \textit{diffuse interface/phase field} approach, which describes the interfaces by layers of small thickness and whose structure is determined by a balance of molecular forces, in such a way that the tendencies for mixing and de-mixing compete through a non-local \textit{mixing} energy. In fact, this idea can be traced to van der Waals \cite{Waals}, and can be viewed as the foundation for the phase-field theory for phase transition and critical phenomena.  This approach uses a phase-field function $\phi= \phi(\x,t)$ that takes distinct values in the pure phases (in our case $\phi=0$ and $\phi=1$) and varies smoothly in the interfacial regions.

The Cahn-Hilliard equation was originally introduced in \cite{Cahn1958} to model the thermodynamic forces driving phase separation, arriving to a system with a gradient flow structure, that is, when there are no external forces applied to the system, the total free energy of the mixture is not increasing in time. The Cahn-Hilliard equation can be defined as a mass balance law with a phase flux $\mathcal{J}$ and $M(\phi)$ representing a mobility function:
$$
\phi_t + \nabla\cdot\mathcal{J}\,=\,0
\quad\quad
\mbox{ with }
\quad\quad
\mathcal{J}\,=\,-M(\phi)\nabla\left(\frac{\delta E(\phi)}{\delta \phi}\right)\,.
$$
For a detailed derivation and discussion of the physics of the Cahn-Hilliard equation we refer the reader to \cite{Bates93,Novick84}. In this work we focus on the case where the mobility $M(\phi)$ is a non-linear degenerate function, meaning that the flux $\mathcal{J}$ only acts away of the pure phases $\phi=0$ and $\phi=1$. The existence and regularity of weak solutions of the Cahn-Hilliard with a degenerate mobility was stablished in \cite{ElliotGarcke1996}, as well as a maximum principle for $\phi$, of type $0\le \phi \le 1$. More recently, well-posedness for the case of a strictly non-negative mobility term was stablished in \cite{DaiDu}. Traditionally, researchers have focused in designing numerical schemes for the constant mobility case (check \cite{Tierra} for an overview and comparison of different approaches), although the case of non degenerate mobility has also received attention. Mixed finite element approximations using logarithmic potentials and degenerate mobilities were studied in \cite{Barrett98,Copetti92}. More recently, several discontinuous Galerkin methods have been considered for the problem with and without convection \cite{DKR22,Kay2009,LiuFrankRiviere,XiaXuShu}. There has been also a great amount of effort in the field to get a better understanding of the mechanisms behind the interface motion, especially in the limit case of taking the interface width to zero. We refer the reader to \cite{DaiDu12,DaiDu14,DaiDu16,Leeetal1,Leeetal2,Leeetal3,Voigt} and the references therein to get an idea of some of the most relevant analytical and computational studies in this direction.

In this work we derive accurate numerical schemes using Finite Differences (FD) in time and Finite Elements (FE) in space to approximate the Cahn-Hilliard equation with degenerate mobility. The proposed schemes are energy stable and at the same time they satisfy an approximate maximum principle, i.e., the amount of the solution being outside of the interval $[0,1]$ is bounded by the small parameter $\varepsilon$ used to truncate the mobility term $M(\phi)$. Our approach is based on using non-centered approximations of the mobility term that lead to estimates involving non-singular functionals, which end up being the key point to derive the approximate maximum principles. In \cite{KMAD19, KMAD20, KMAD21, KMAD22, GuillenTierra} schemes based in similar ideas have proved useful themselves in the context of some chemotaxis models.

This work is organized as follows: In Section~\ref{sec:model} we present the Cahn-Hilliard model with degenerate mobility and we introduce the truncated-degenerated model together with the truncated-non-degenerated model and with the regularization of two singular potentials, so-called G$_\varepsilon(\phi)$ and J$_\varepsilon(\phi)$, detailing the estimates that they satisfy. Two new numerical schemes are presented in Section~\ref{sec:schemes}, together with the 
properties that they satisfy, including conservation of volume, energy-stability and approximated maximum principles. In Section~\ref{sec:simulations} we present several numerical results illustrating the accuracy of the proposed numerical schemes and their applicability to simulate complex situations. Finally the conclusions of our work are stated in Section~\ref{sec:conclusions}.

\section{Model}\label{sec:model}
Let $\Omega\subset\mathbb{R}^d$ (with $d=1,2,3$) be a bounded spatial domain and $[0,T]$ a finite time interval.
The Cahn-Hilliard equation is a fourth order PDE where the phase function $\phi=\phi(\x,t)$ satisfy the PDE 
\beq\label{eq:CH}
\phi_t - \nabla \cdot\left[M(\phi)\nabla\left(\frac{\delta E}{\delta \phi}\right)\right] \,=\,0\, ,
\quad\quad\quad \mbox{ for } (\x,t)\in\Omega\times(0,T)\,,
\eeq
subject to the following boundary and initial conditions:
\beq\label{eq:CHbc}
\partial_\n \phi|_{\partial\Omega}=\partial_\n \left(\frac{\delta E}{\delta \phi}\right)\Big|_{\partial\Omega}\,=\,0
\quad
\mbox{ and }
\quad \phi(\x,0)=\phi_0(\x)
\quad
\mbox{ with }
\quad
0\leq\phi_0(x)\leq 1\,.
\eeq
Here, $M(\phi)\ge 0$ is a mobility function and $E(\phi)$ denotes the free energy of the system
\beq\label{eq:CHenergy}
E(\phi)
\,:=\,
E_{philic}(\phi)+E_{phobic}(\phi)
\,:=\,
\int_\Omega \left(\frac12|\nabla\phi|^2 + F(\phi)\right)d\x\,,
\eeq
with $\frac{\delta E}{\delta \phi}$ denoting the Riesz identification in $L^2(\Omega)$  of the variational derivative of the functional $E(\phi)$ with respect to $\phi$, that is,
\beq\label{eq:varderE}
\frac{\delta E}{\delta \phi} 
\,=\,
-\Delta\phi  + F'(\phi)
\,.
\eeq 

The Cahn-Hilliard equation \eqref{eq:CH} can also be written as a system of two second order PDEs introducing as a new unknown the chemical potential $\mu:=\frac{\delta E}{\delta \phi}$ such that:
\beq\label{eq:CHsystem}
\left\{\ba{rcl}
\phi_t - \nabla\cdot(M(\phi)\nabla\mu) 
&=&0\,,
\\ \hueco
-\Delta\phi + F'(\phi)
&=&
\mu\,,
\ea\right.
\eeq
endowed with the boundary and initial conditions 
\beq\label{eq:CHbc-bis}
\partial_\n \phi |_{\partial\Omega}=\partial_\n\mu |_{\partial\Omega}\,=\,0
\quad
\mbox{ and }
\quad \phi(\x,0)=\phi_0(\x)\,.
\eeq

The system depends on a double-well potential $F(\phi)$ taking two minimum (stable) values. The original potential introduced by Cahn and Hilliard in \cite{Cahn1958} it is known as the \textit{Flory-Huggins} potential:
$$
F_{FH}(\phi)\,:=\, \frac{\theta}2[\phi\ln(\phi) + (1-\phi)\ln(1-\phi)]
+ \frac{\theta_c}2\phi(1-\phi)\,,
$$
with $\theta, \theta_c > 0$ denoting the absolute and critical temperature of the system, respectively.
The fact that $F_{FH}(\phi)$ involves logarithmic terms (making the derivative of the  potential  singular) introduces great difficulties to study the system analytically as well as numerically. Then, several approximations of the potential have been considered, being the most widely used  the so-called \textit{Ginzburg-Landau} one:
\beq\label{eq:potentialGL}
F_{GL}(\phi)\,:=\,\frac1{4\eta^2}\phi^2(\phi - 1)^2\,,
\eeq
with $\eta>0$ being a small parameter (related with the interfacial width). 

Some truncated versions of the potential \eqref{eq:potentialGL} have been  proved useful to design numerical schemes in different settings related with phase field models \cite{Nochetto14,Salgado12,SHEN10b,WuZwietenZee}. Moreover, using a  truncated potential and a constant mobility term ($M(\phi)=C$), Caffarelli and Muller \cite{Caffarelli} were able to show $L^\infty$-bounds on the solutions of the Cahn-Hilliard equations when the considered domain is the whole space $\Omega=\mathbb{R}^d$ .

In this work we focus on the \textit{Ginzburg-Landau} potential, so from now on we denote $F_{GL}(\phi)$ by $F(\phi)$, omitting the subscripts for the sake of simplicity of notation. In general, potentials can be splitted into two parts, a convex (or contractive) part ($F_c''(\phi)\geq0$) and a concave (or expansive) part ($F_e''(\phi)\leq0$), such that
$$
F(\phi)\,=\,F_c(\phi) + F_e(\phi)\,.
$$
In particular, in this work we consider the following splitting (it will be used in the proof of Lemma~\ref{lem:Gestimatemod})
\beq\label{eq:convexsplitting1}
F_c(\phi)
\,:=\,
\frac{1}{4\eta^2}\left(\phi^4 - 2 \phi^3 + \frac32\phi^2\right)
\quad
\mbox{ and }
\quad
F_e(\phi)=-\frac1{8\eta^2}\phi^2\,,
\eeq
which satisfies
\beq\label{eq:convexsplitting2}
F_c''(\phi)=\frac{3}{\eta^2}\left(\phi - \frac12\right)^2 \geq 0
\quad
\mbox{ and }
\quad
F_e''(\phi)=-\frac{1}{4\eta^2} \leq 0\,.
\eeq
%
%

\subsection{Truncated-Degenerated model}
We consider the degenerate and truncated by zero mobility term:
\beq\label{def:mobility0}
M_0(\phi)
\,=\,
\left\{\ba{lr}
\phi(1 - \phi), & \phi\in[0,1]\,,
\\ [2ex]
0 & \mbox{otherwise}\,.
\ea\right.
\eeq
Therefore the truncated-degenerated Cahn-Hilliard system can be rewritten as: Find $(\phi,\mu)$ such that
\beq\label{eq:CHsystem0}
\left\{\ba{rcl}
\phi_t - \nabla\cdot(M_0(\phi)\nabla\mu) 
&=&0\,,
\\ \hueco
-\Delta\phi + F'(\phi)
&=&
\mu\,.
\ea\right.
\eeq

An existence result for the problem with degenerate mobility \eqref{eq:CHsystem0} is presented in \cite{ElliotGarcke1996},
as well as a result about the boundedness of the magnitude of the solution, that is,
\beq\label{eq:boundsolution}
\mbox{ if }\quad 
\phi(\x,0)\in [0,1] \mbox{ a.e. in } \Omega,
\quad\mbox{ then }\quad
\phi(\x,t)\in [0,1]
 \mbox{ a.e. in } \Omega\times(0,T)\,.
\eeq

In particular, under reasonable assumptions for the initial data $\phi_0$, the solution of problem \eqref{eq:CHsystem0} satisfy the following regularity results \cite{ElliotGarcke1996}:
\begin{itemize}
\item $
\phi\in L^2(0,T;H^2(\Omega))\cap L^\infty(0,T;H^1(\Omega))\cap C([0,T];L^2(\Omega))\,,
$
\item $\phi_t \in L^2(0,T;(H^1(\Omega))')\,,$
\item $M_0(\phi)\nabla\mu \in L^2(\Omega\times(0,T);\mathbb{R}^d)\,.$
\end{itemize}



\begin{obs}
The existence and boundedness results presented in \cite{ElliotGarcke1996} hold for  the \textit{Flory-Huggins} potential ($F_{FH}(\phi)$) and for the \textit{Ginzburg-Landau} ($F_{GL}(\phi)$) one.
\end{obs}

\subsection{Truncated-Non-Degenerated model}
We follow the arguments in \cite{ElliotGarcke1996} and we present a non-degenerated version of problem \eqref{eq:CHsystem0} by replacing the mobility term $M_0(\phi)$ defined in \eqref{def:mobility0} by $M_\varepsilon(\phi)$ depending on a small parameter $\varepsilon>0$ such that 
\beq\label{eq:defMeps}
M_\varepsilon(\phi)
\,:=\,
\left\{
\ba{ll}
M(\varepsilon) 
& \mbox{ if } \phi < \varepsilon\,,
\\ \hueco
M(\phi) 
& \mbox{ if }  \varepsilon \le \phi \le1-\varepsilon\,,
\\ \hueco
M(1-\varepsilon) 
& \mbox{ if } 1-\varepsilon < \phi\,.
\ea\right.
\eeq
In particular, $| M_\varepsilon(\phi) - M_0(\phi)|\le \varepsilon(1-\varepsilon) $  for all $\phi\in \mathbb{R}$.


\begin{figure}[H]
\begin{center}
\includegraphics[scale=0.1152]{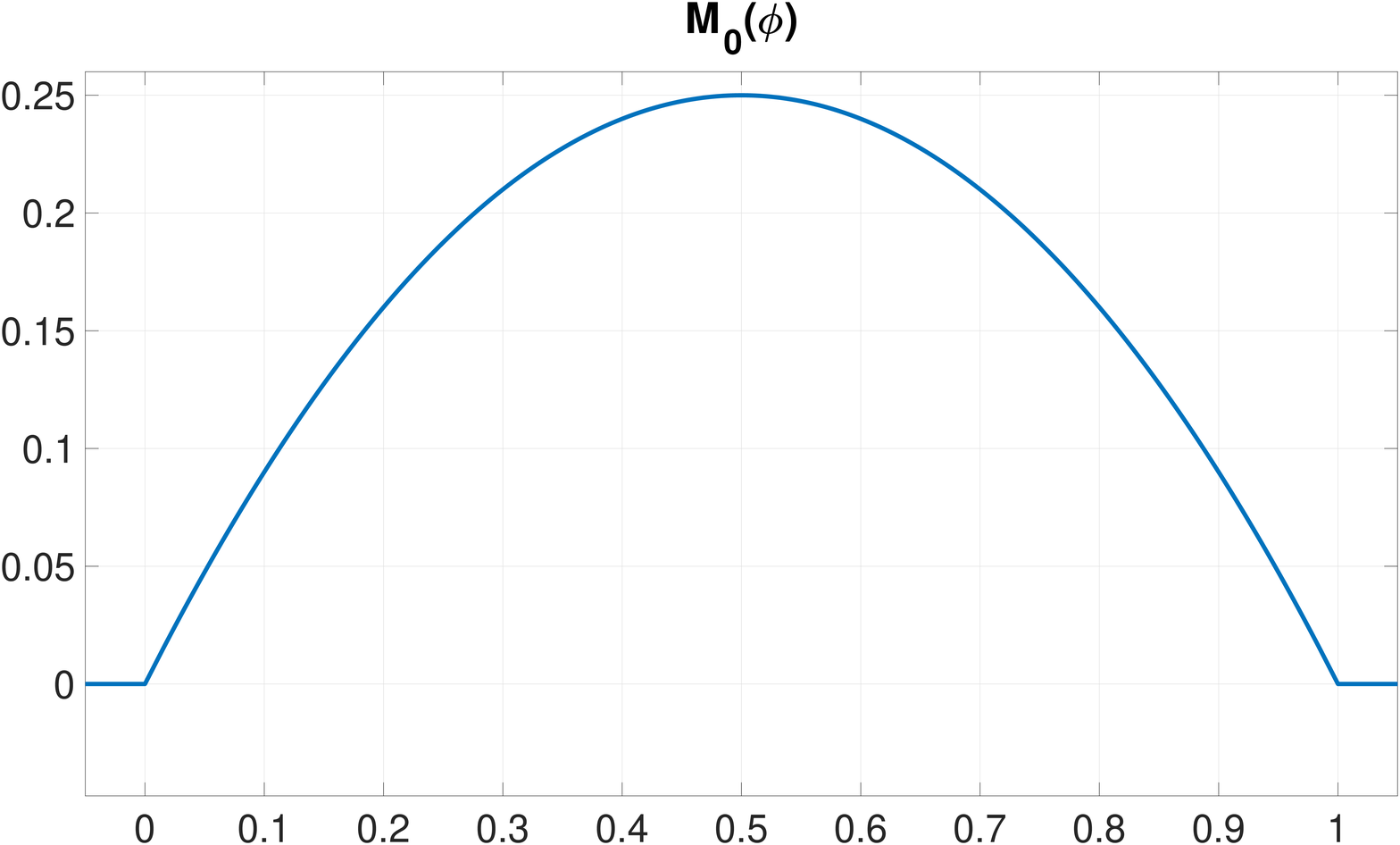}
\includegraphics[scale=0.1152]{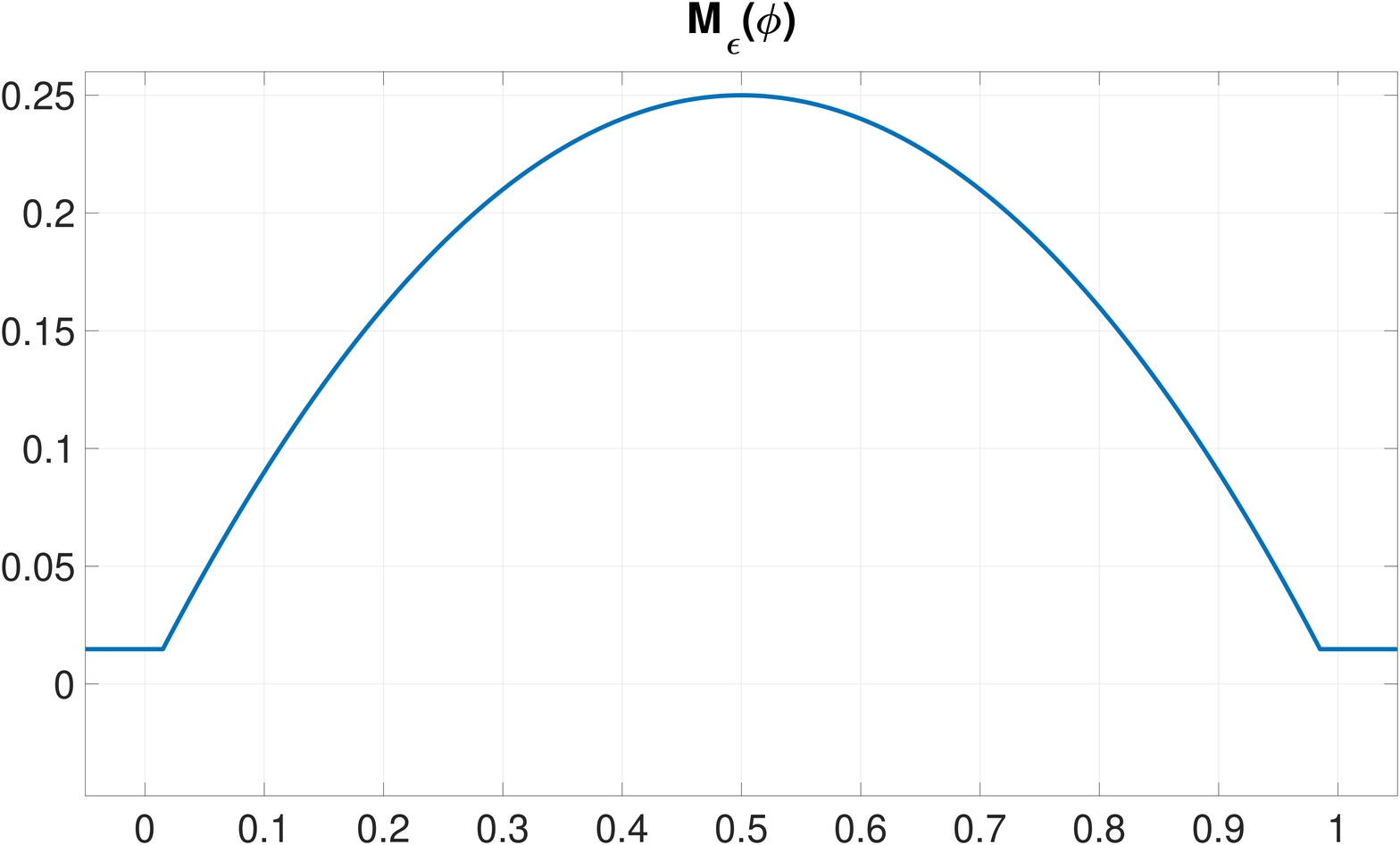}
\end{center}
\caption{Comparison between mobility term $M_0(\phi)$ presented in \eqref{def:mobility0} (left) and mobility term $M_\varepsilon(\phi)$ presented in \eqref{eq:defMeps} with $\varepsilon=0.015$ (right).}
\end{figure}

Then, using definition \eqref{eq:defMeps} we consider the approximate version of \eqref{eq:CHsystem0}:
\beq\label{eq:CHsystemmod}
\left\{\ba{rcl}
\phi_t - \nabla\cdot(M_\varepsilon(\phi)\nabla\mu) 
&=&0\,,
\\ \hueco
-\Delta\phi + F'(\phi)
&=&
\mu\,,
\ea\right.
\eeq
subject to the same boundary and initial conditions presented in \eqref{eq:CHbc-bis} and with the same energy \eqref{eq:CHenergy} considered in the original problem \eqref{eq:CHsystem}. An existence result for problem \eqref{eq:CHsystemmod} is also presented in \cite{ElliotGarcke1996}.

In particular, under reasonable assumptions for the initial data $\phi_0$, the solution of problem \eqref{eq:CHsystemmod} satisfy the following regularity results \cite{ElliotGarcke1996}:
\begin{itemize}
\item $
\phi\in L^\infty(0,T;H^1(\Omega))\cap C([0,T];L^2(\Omega))\,,
$
\item $\phi_t \in L^2(0,T;(H^1(\Omega))')\,,$
\item $\mu \in L^2(0,T;H^1(\Omega))\,.$
\end{itemize}
Moreover, the following approximate boundedness results are provided:
$$
\esssup_{0\leq t \leq T}\int_\Omega (\phi)_-^2 d\x \leq C\varepsilon^\alpha
\quad\quad\mbox{ and }\quad\quad
\esssup_{0\leq t \leq T}\int_\Omega (\phi - 1)_+^2 d\x\leq C\varepsilon^\alpha\,,
$$
with $\alpha=1$ in the case of the potential and mobility terms considered in this work.

In fact, these results are the main tools to show the existence and boundedness of solution of the degenerated problem \eqref{eq:CHsystem0} via a limit process as $\varepsilon\to 0$ (see  \cite{ElliotGarcke1996}).

In particular, for the approximated model \eqref{eq:CHsystemmod}, the following result holds:

\begin{lem}\label{lem:energylawmod}
Any solution of the problem \eqref{eq:CHsystemmod} and \eqref{eq:CHbc-bis},  satisfies the conservation property 
\beq\label{eq:conservation}
\int_\Omega \phi(t,\x) d\x
\,=\,
\int_\Omega \phi(0,\x) d\x
\quad \pato t>0
\eeq
and the following integral version of the energy law
\beq\label{eq:Energylawmod}
E (\phi(t))
+
\int_0^t \int_\Om M_\varepsilon(\phi)|\nabla\mu|^2 d\x \,d\tau
=
E (\phi_0)
\quad \mbox{ a.e. } t>0\,.
\eeq
In particular, \eqref{eq:Energylawmod} implies the following estimates
\begin{equation}\label{estimH1bL1mod}
\int_\Omega |\nabla\phi|^2  d\x \in L^\infty(0,+\infty)
\quad \mbox{ and }\quad
\int_\Om M_\varepsilon(\phi)|\nabla\mu|^2 d\x\in L^1(0,+\infty)\,.
\end{equation}
\end{lem}
\begin{proof}
We obtain \eqref{eq:conservation} by testing \eqref{eq:CHsystemmod}$_1$ by $1$. Moreover, testing equation \eqref{eq:CHsystemmod}$_1$ by $\mu$ and \eqref{eq:CHsystemmod}$_2$ by $\phi_t$ and working as in  \cite{ElliotGarcke1996} we arrive at \eqref{eq:Energylawmod}.
\end{proof}

The next step is to introduce two functionals $G_\varepsilon(\phi)$ and $J_\varepsilon(\phi)$ (defined from $M_\varepsilon(\phi)$) such that they satisfy estimates that will be helpful to show the approximate boundedness of the solution. Moreover, discrete versions of these estimates will prove useful when deriving the numerical schemes, in order to be able to show approximate maximum principle properties.
\subsubsection{Functional $G_\varepsilon(\phi)$}

Following the ideas in \cite{ElliotGarcke1996}, we define a new functional $G_\varepsilon(\phi)$ by using the truncated mobility $M_\varepsilon(\phi)$ for a fixed  truncation parameter $\varepsilon>0$:

\beq\label{eq:G2trunc}
G''_\varepsilon(\phi)
\,:=\,
\frac{1}{M_\varepsilon(\phi)}
\quad \forall\, \phi\in \mathbb{R}\,.
\eeq


\begin{obs}
Denoting 
$$
G(\phi)\,:=\,\phi\ln(\phi) + (1 - \phi)\ln(1-\phi) + 1
\quad \forall\, \phi\in [0,1]\,,
$$
with
$$
G'(\phi)\,=\,\ln\left(\frac\phi{1-\phi}\right) \quad \forall\, \phi\in (0,1),
$$
and
$$
G''(\phi)\,=\,\frac1{\phi(1-\phi)}=\frac1{M(\phi)}
\quad \forall\, \phi\in (0,1)\,,
$$
then the functionals $G''_\varepsilon(\phi)$ in \eqref{eq:G2trunc}, $G'_\varepsilon(\phi)$ and $G_\varepsilon(\phi)$ can be written as
$$
G'_\varepsilon(\phi)
\,:=\,
\left\{
\ba{ll}
\dis G'(\varepsilon) + G''(\varepsilon)(\phi -\varepsilon)
& \mbox{ if } \phi < \varepsilon\,,
\\ \hueco
\dis G'(\phi)
 & \mbox{ if }  \varepsilon \le \phi \le 1-\varepsilon\,,
\\ \hueco
\dis G'(1-\varepsilon) + G''(1-\varepsilon)(\phi -(1-\varepsilon))
& \mbox{ if } 1-\varepsilon < \phi\,,
\ea
\right.
$$
and
$$
G_\varepsilon(\phi)
\,:=\,
\left\{
\ba{ll}
\dis 
G(\varepsilon)+G'(\varepsilon)(\phi -\varepsilon)+ \frac12 G''(\varepsilon)(\phi -\varepsilon) ^2
& \mbox{ if } \phi < \varepsilon\,,
\\ \hueco G(\phi)
& \mbox{ if }  \varepsilon \le \phi \le 1-\varepsilon\,,
\\ \hueco
\dis G(1-\varepsilon)+G'(1-\varepsilon)(\phi -(1-\varepsilon))+ \frac12 G''(1-\varepsilon)(\phi -(1-\varepsilon)) ^2
& \mbox{ if } 1-\varepsilon < \phi\,,
\ea
\right.
$$

\end{obs}

\begin{obs}\label{rk:boundsG}
Functional $G_\varepsilon(\phi)\in C^2(\mathbb{R})$ with $G_\varepsilon(\phi)\ge0$, $G_\varepsilon'(\varepsilon)<0$ and $G_\varepsilon'(1 - \varepsilon)>0$, then $G_\varepsilon(\phi)$ satisfy \cite{ElliotGarcke1996}

$$
G_\varepsilon(\phi)
\geq
\dis\frac{1}{2\varepsilon(1 - \varepsilon)} (\phi - \varepsilon)^2
\geq
\dis\frac{1}{2\varepsilon(1 - \varepsilon)} \phi^2
\quad\quad
\pato \phi\leq \varepsilon
$$
and 
$$
G_\varepsilon(\phi)
\geq
\dis\frac{1}{2\varepsilon(1 - \varepsilon)} (\phi - (1-\varepsilon))^2
\geq
\dis\frac{1}{2\varepsilon(1 - \varepsilon)} (\phi - 1)^2
\quad 
\pato \phi\geq 1-\varepsilon\,.
$$

\end{obs}
\begin{lem}\label{lem:Gestimatemod}
Any solution of the problem \eqref{eq:CHsystemmod} and \eqref{eq:CHbc-bis} satisfies  the following relation a.e. $t>0$:
\beq\label{eq:Gestimatemod}
\int_\Omega G_\varepsilon(\phi(t)) \,d\x
+
\int_0^t \int_\Omega \left((\Delta\phi)^2
+
\frac3{\eta^2}\left(\phi - \frac12\right)^2 |\nabla\phi|^2 
\right) d\x\, d\tau
=
\int_\Omega G_\varepsilon(\phi_0) \,d\x
+ 
\frac1{4\eta^2}\int_0^t\int_\Omega|\nabla\phi|^2 d\x \,d\tau
\,.
\eeq
Moreover, using estimates in \eqref{estimH1bL1mod},  we can deduce 
$$
\left\| \int_\Omega G_\varepsilon(\phi(t,\cdot)) d\x\right\|_{L^\infty(0,T)}
+ \int_0^T\int_\Omega (\Delta\phi)^2 d\x dt
\leq
C\,\frac{T}{\eta^2}\,,
$$
where $C$ depends on the initial energy $E(\phi_0)$ and  $\int_\Omega G_\varepsilon(\phi_0) \,d\x$. 
\end{lem}
\begin{proof}
Testing equation \eqref{eq:CHsystemmod}$_1$ by $G_\varepsilon'(\phi)$ and using that $M_\varepsilon(\phi)=1/G_\varepsilon''(\phi)$, we obtain
$$
0\,=\,
\dis\int_\Omega \phi_tG_\varepsilon'(\phi) d\x
+ \int_\Omega \frac{1}{G_\varepsilon''(\phi)}\nabla\mu\cdot\nabla G_\varepsilon'(\phi) d\x\,.
$$
Moreover, using the expression of $\mu$ in \eqref{eq:CHsystemmod}$_2$ and the splitting presented in \eqref{eq:convexsplitting1} we get
$$
\ba{rcl}
\dis
\int_\Omega \nabla\mu\cdot \nabla\phi \,d\x
&=&\dis
\int_\Omega \nabla\left(-\Delta\phi + F'(\phi)\right)\cdot \nabla\phi \,d\x
\\ \hueco
&=&\dis
\int_\Omega (\Delta\phi)^2 d\x
+\int_\Omega  F_c''(\phi)|\nabla\phi|^2 \,d\x
+\int_\Omega  F_e''(\phi)|\nabla\phi|^2 \,d\x\,.
\ea
$$
Integrating in time, taking into account the equality \cite{ElliotGarcke1996}
$$
\int_0^t \int_\Omega \phi_tG_\varepsilon'(\phi) d\x= \int_\Omega G_\varepsilon(\phi(t)) \,d\x - \int_\Omega G_\varepsilon(\phi(0)) \,d\x
\quad\mbox{ a.e. } t>0\,
$$
and using \eqref{eq:convexsplitting2} we arrive at \eqref{eq:Gestimatemod}. 

\end{proof}

\subsubsection{Functional $J_\varepsilon(\phi)$}

We define a new functional $J_\varepsilon(\phi)$ by using the truncated mobility $M_\varepsilon(\phi)$
such that:

\beq\label{eq:J2trunc}
J''_\varepsilon(\phi)
\,:=\,
\frac{1}{\sqrt{M_\varepsilon(\phi)}}
\quad \forall\, \phi\in \mathbb{R}\,.
\eeq

\begin{obs}
Denoting
$$
J(\phi)\,:=\,(-2\phi +1)\arcsin\big({\sqrt{1-\phi}}\big) +\sqrt{(1-\phi)\phi}
+ 2\arcsin({\sqrt{1/2}})\phi
\quad \forall\, \phi\in [0,1]\,,
$$
with
$$
J'(\phi)\,=\,-2(\arcsin({\sqrt{1-\phi}}) - \arcsin({\sqrt{1/2}}))
\quad \forall\, \phi\in [0,1]
$$
and
$$
J''(\phi)\,=\,\frac1{\sqrt{\phi(1-\phi)}}=\frac1{\sqrt{M(\phi)}}
\quad \forall\, \phi\in (0,1)\,,
$$
then the functionals $J''_\varepsilon(\phi)$ in \eqref{eq:J2trunc}, $J'_\varepsilon(\phi)$ and $J_\varepsilon(\phi)$ can be written as
$$
J'_\varepsilon(\phi)
\,:=\,
\left\{
\ba{ll}
\dis J'(\varepsilon) + J''(\varepsilon)(\phi-\varepsilon)
 & \mbox{ if } \phi < \varepsilon\,,
\\ \hueco
J'(\phi)
& \mbox{ if }  \varepsilon \le \phi \le 1-\varepsilon\,,
\\ \hueco
J'(1-\varepsilon) + J''(1-\varepsilon)(\phi-(1-\varepsilon))
& \mbox{ if } 1-\varepsilon < \phi\, .
\ea\right.
$$
and
$$
J_\varepsilon(\phi)
\,:=\,
\left\{
\ba{ll}
\dis J(\varepsilon)+ J'(\varepsilon) (\phi-\varepsilon)+ \frac12 J''(\varepsilon)(\phi-\varepsilon) ^2
 & \mbox{ if } \phi < \varepsilon\,,
\\ \hueco\dis J(\phi)
& \mbox{ if }  \varepsilon \le \phi \le 1-\varepsilon\,,
\\ \hueco
\dis J(1-\varepsilon)+ J'(1-\varepsilon) (\phi-(1-\varepsilon))+ \frac12 J''(1-\varepsilon)(\phi-(1-\varepsilon)) ^2
& \mbox{ if } 1-\varepsilon < \phi\,.
\ea\right.
$$
\end{obs}
\begin{obs}\label{rk:boundsJ}

Functional $J_\varepsilon(\phi)\in C^2(\mathbb{R})$ with $J_\varepsilon(\phi)\ge0$, $J_\varepsilon'(\varepsilon)<0$ and $J_\varepsilon'(1 - \varepsilon)>0$, then $J_\varepsilon(\phi)$ satisfy \cite{ElliotGarcke1996}

$$
J_\varepsilon(\phi)
\geq
\dis\frac{1}{2\sqrt{\varepsilon(1 - \varepsilon)}} (\phi - \varepsilon)^2
\geq
\dis\frac{1}{2\sqrt{\varepsilon(1 - \varepsilon)}} \phi^2
\quad\quad
\pato \phi\leq \varepsilon
$$
and 
$$
J_\varepsilon(\phi)
\geq
\dis\frac{1}{2\sqrt{\varepsilon(1 - \varepsilon)}} (\phi - (1-\varepsilon))^2
\geq
\dis\frac{1}{2\sqrt{\varepsilon(1 - \varepsilon)}} (\phi - 1)^2
\quad 
\pato \phi\geq 1-\varepsilon\,.
$$
%
%
%
%
%

\end{obs}
\begin{lem}\label{lem:Jestimatemod}
Any solution of the problem \eqref{eq:CHsystemmod} and \eqref{eq:CHbc-bis} satisfies the following inequality a.e. $t>0$:
\beq\label{eq:Jestimatemod}
\int_\Omega J_\varepsilon(\phi(t)) \,d\x
\leq
\int_\Omega J_\varepsilon(\phi_0) \,d\x
+ 
\int_0^t\int_\Omega\left( 
M_\varepsilon(\phi)|\nabla\mu|^2
+
\int_\Omega |\nabla\phi|^2\right) d\x
\, d\tau\,.
\eeq
Moreover, using estimates in \eqref{estimH1bL1mod} we can deduce 
$$
\left(\int_\Omega J_\varepsilon(\phi(t)) \,d\x\right)
\leq
C_1+C_2 t\,,
\quad \hbox{a.e. $t>0$,} 
$$
with $C_1$ and $C_2$ are independent of $\eta$ (and depend on the initial energy $E(\phi_0)$ and $\int_\Omega J_\varepsilon(\phi_0) \,d\x$). Therefore we can also deduce
$$
\left\| \int_\Omega J_\varepsilon(\phi(t,\cdot)) d\x\right\|_{L^\infty(0,T)}
\leq
\Big(C_1 + C_2 T\Big)\,.
$$ 
\end{lem}
\begin{proof}
Testing \eqref{eq:CHsystemmod}$_1$ by $J_\varepsilon'(\phi)$ and using \eqref{eq:J2trunc} we obtain
$$
\int_\Omega J_\varepsilon'(\phi)\phi_t \,d\x
+ \int_\Omega \sqrt{M_\varepsilon(\phi)}\nabla\mu\cdot\nabla\phi \,d\x
=
0\,.
$$
Integrating in time,
 and using $ab\leq {a^2}/2 +{b^2}/2$, we can deduce \eqref{eq:Jestimatemod}.

\end{proof}
\section{Two fully discrete numerical schemes}\label{sec:schemes}
The numerical schemes proposed in this section are designed as approximations of the corresponding weak formulation of the system \eqref{eq:CHsystemmod}. Hereafter $(\cdot,\cdot)$ denotes the $L^2(\Omega)$-scalar product.
For all numerical schemes we consider a partition of the time interval $[0,T]$ into $N$ subintervals, with constant time step $\Delta t = T /N$ and we denote by $\delta_t$ the (backward) discrete time derivative
$$
\delta_t f\,:=\,\frac{f^{n+1} - f^n}{\Delta t}\,.
$$
We consider structured triangulations $\{\mathcal{T}_h\}$ of the domain $\Omega$ with its elements denoted by $I_i$, that is $\{\mathcal{T}_h\}=\bigcup_{i} I_i$ with the size of elements being bounded  by $h>0$. The unknowns $(\phi,\mu)$ are approximated by the $C^0$-Finite Elements spaces of order $k\ge1$ (denoted by $\mathbb{P}_k$):
$$
(\phi,\mu)\in\Phi_h\times W_h=\mathbb{P}_1\times\mathbb{P}_k\,.
$$ 
Moreover, we need to use \textit{mass-lumping} ideas \cite{CiarletRaviart73} to help us achieve boundedness of the unknown $\phi$. To this end we introduce the discrete semi-inner product on $C^0(\overline{\Omega})$ and its induced discrete semi-norm:
$$
(f,g)_h\,:=\,\int_\Omega I_h(f g)d\x
\quad
\mbox{ and }
\quad
|f|_h\,:=\,\sqrt{(f,f)_h}\,,
$$
with $I_h(f(\x))$ denoting the nodal $\mathbb{P}_1$-interpolation of the function $f(\x)$. We use $f_+$ or $f_-$ to denote the positive or negative part of a function $f$ ($f_+(\x)=\max\{f(\x),0\}$ and $f_-(\x)=\min\{f(\x),0\}$). 

We propose two numerical schemes,  called \textit{G$_\varepsilon$-scheme} and \textit{J$_\varepsilon$-scheme}, designed to take advantage of the discrete versions of the estimates for $G_\varepsilon(\phi)$ \eqref{eq:Gestimatemod} 
 and $J_\varepsilon(\phi)$ \eqref{eq:Jestimatemod}, respectively. Both schemes will be \textit{conservative} and \textit{energy stable}, that is, they will maintain the amount of phase  constant in time ($\int_\Omega\phi^{n+1}=\int_\Omega\phi^{n}$) and they will satisfy discrete versions of the energy law \eqref{eq:Energylawmod}, which in particular imply the energy decreasing property ($E(\phi^{n+1})\le E(\phi^{n})$). In fact, the key point to achieve the energy stability of the schemes is to take advantage of the splitting in \eqref{eq:convexsplitting1} and to consider the ideas of Eyre \cite{Eyre}, i.e., take implicitly the convex part and explicitly the concave part because the following relation holds:
\beq\label{eq:eyre}
\frac1{\Delta t}\int_\Omega \Big(F'_c(\phi^{n+1}) + F'_e(\phi^{n})\Big)(\phi^{n+1} - \phi^n) d\x
\,\geq\,\frac1{\Delta t}\int_\Omega \Big(F(\phi^{n+1}) - F(\phi^{n})\Big) d\x \,.
\eeq
\subsection{G$_\varepsilon$-scheme}
The \textit{G$_\varepsilon$-scheme} is defined as follows: Given $\phi^{n} \in \Phi_h$, find $(\phi^{n+1},\mu^{n+1}) \in \Phi_h\times M_h$ such that
\beq\label{eq:SchemeG}
\left\{\ba{rcl}\dis
\frac1{\Delta t}\big(\phi^{n+1} - \phi^n, \bar\mu\big)_h 
+\big(M^G_\varepsilon(\phi ^{n+1})\nabla\mu^{n+1},\nabla\bar\mu\big) 
&=&0\,,
\\ \hueco
(\nabla\phi^{n+1},\nabla\bar\phi)
+ \big(I_h(F_c'(\phi^{n+1})) + I_h(F_e'(\phi^{n})),\bar\phi\big)_h
&=&
(\mu^{n+1},\bar\phi)_h\,,
\ea\right.
\eeq
for all $(\bar\phi,\bar\mu)\in\Phi_h\times M_h$, 
where $M^G_\varepsilon(\phi)$ for any $\phi \in\Phi_h=\mathbb{P}_1$ will be an adequate $\mathbb{P}_0$ approximation of 
$M_\varepsilon(\phi)=1/G''_\varepsilon(\phi)$ (owing to \eqref{eq:G2trunc}), satisfying the $\mathbb{P}_0$ discrete equality
\beq\label{eq:MG}
M^G_\varepsilon(\phi)\,\nabla I_h \Big(G'_\varepsilon(\phi)\Big) = \nabla \phi
\quad \forall\, \phi \in\Phi_h \,.
\eeq
Indeed, $M^G_\varepsilon(\phi)$ for any $\phi\in \Phi_h$ will be a diagonal matrix function, denoting its diagonal  elements $M^G_{kk}$ for $k=1,\dots,d$. 
For each element $I_i$ of the triangulation $\{\mathcal{T}_h\}$ (with nodes $\x_{0}$, $\x_k$ in the $k$-th axis) we compute the $\mathbb{P}_0$ function
\beq\label{eq:defdiscreteMobilityG}
M^G_{kk}\Big|_{I_i} 
\,:=\,
\left\{\ba{cc}\dis
\frac{ \phi(\x_k) - \phi(\x_{0}) }
{ G'_\varepsilon(\phi(\x_k))  -  G'_\varepsilon(\phi(\x_{0}))}
& 
 \mbox{ if } \phi(\x_{0}) \neq \phi(\x_k)\,,
\\ \hueco
\dis\frac1{G''_{\varepsilon}(\phi(\x_{0}))} 
& 
\mbox{ if } \phi(\x_{0}) = \phi(\x_k)\,.
\ea\right.
\eeq
In fact, since $\Phi_h=\mathbb{P}_1$, the relation \eqref{eq:MG} holds as a $\mathbb{P}_0$ equality. Note that $M^G_{kk}\Big|_{I_i}\ge 0$ owing to the convexity of $G_\varepsilon(\phi)$.
\begin{obs}
$M^G_\varepsilon(\phi)$ can be understood as a non-centered modification of the mobility  $M_\varepsilon(\phi)$ in such a way that both
practically coincide when $\phi$ is located around the center $\phi=1/2$ of the interval $(0,1)$ but they differ when the values of $\phi$ are close to $0$ and $1$ (see  Figure~\ref{fig:ComparionsMgMjM}).  We observe how this  non-centered mobility  $M^G_\varepsilon(\phi)$ is going to zero at the endpoints $\phi=0$ and $\phi=1$ faster than $M_\varepsilon(\phi)$, as $\varepsilon $ goes to zero, which will help to control the boundedness of the unknown $\phi$ between $0$ and $1$ in the numerical scheme.
\begin{figure}
\begin{center}
\includegraphics[scale=0.115]{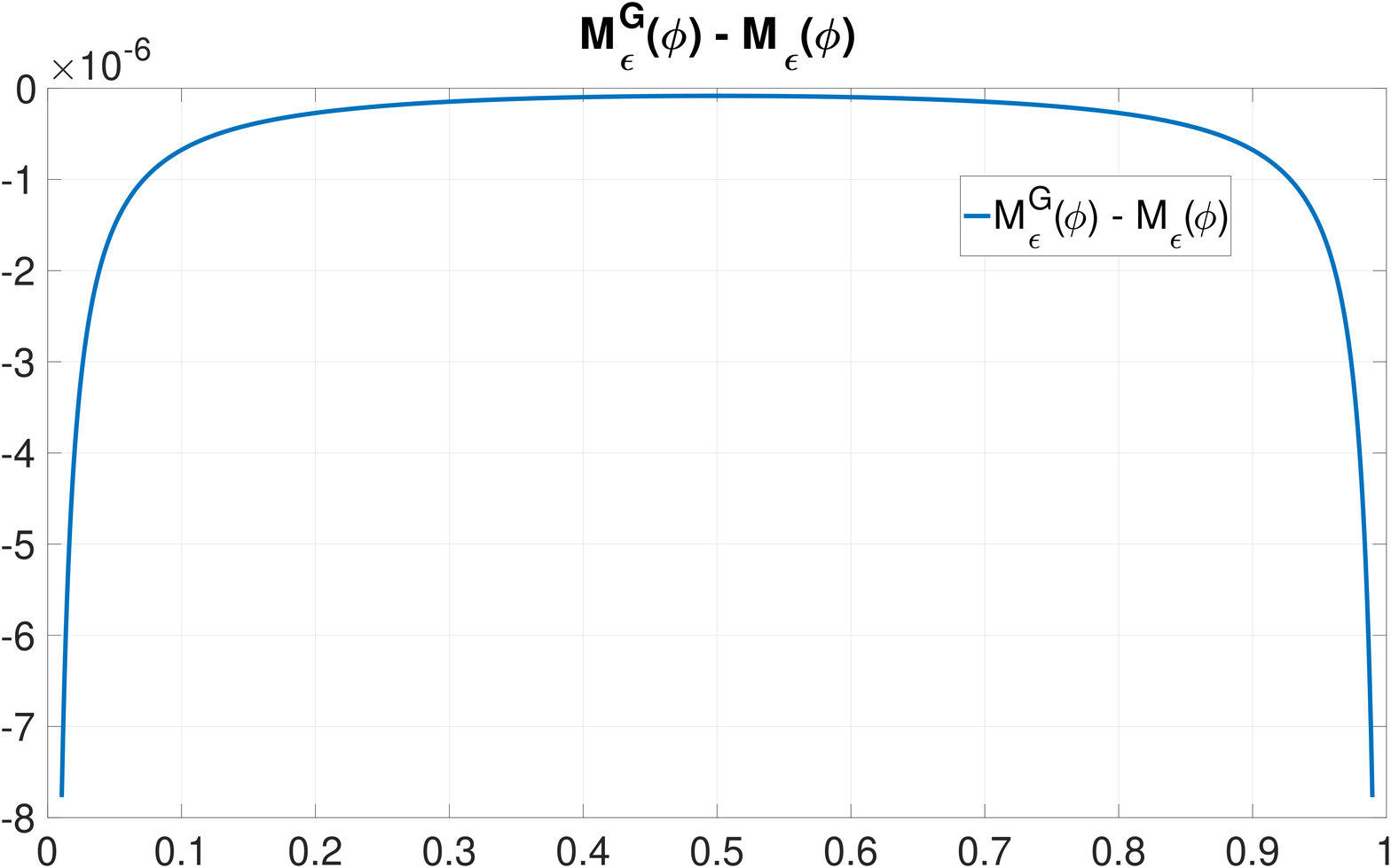}
\includegraphics[scale=0.115]{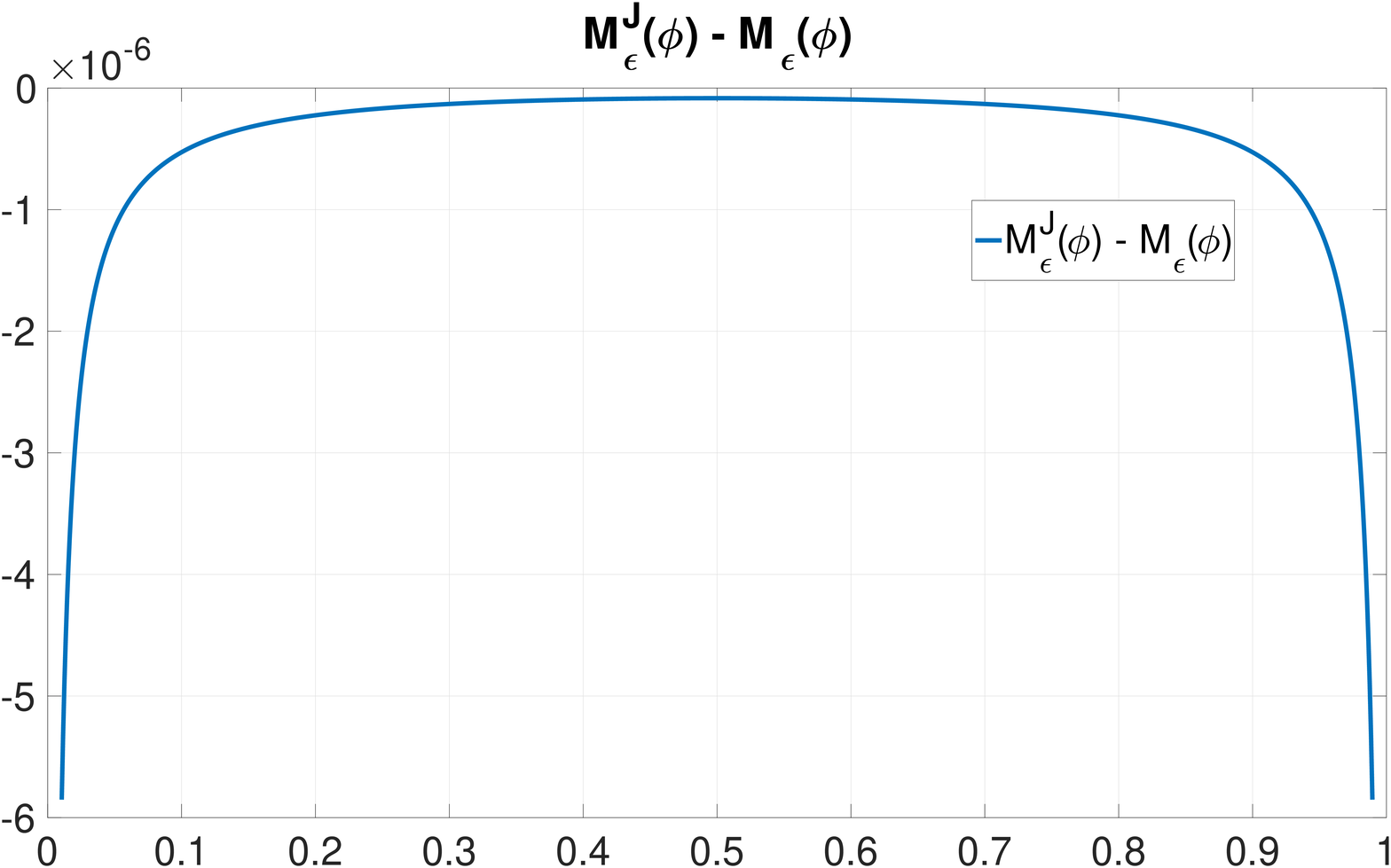}
\end{center}
\caption{Comparison of $M_\varepsilon(\phi)$ with $M^G_\varepsilon(\phi)$ (left) and $M^J_\varepsilon(\phi)$ (right) by plotting their difference in the interval $\phi\in [0.01,0.99]$.}\label{fig:ComparionsMgMjM}
\end{figure}
\end{obs}

\begin{obs}
In order to be able to construct the matrix $M^G_\varepsilon(\phi)$ we need to use the requirement of considering structured triangulations of $\Omega$.
\end{obs}

\subsubsection{Conservation of volume and energy-stability}

\begin{lem} Any solution $(\phi^{n+1},\mu^{n+1}) \in \Phi_h\times M_h$ of the 
G$_\varepsilon$-scheme \eqref{eq:SchemeG} satisfies the conservation property
\beq\label{eq:conservationG}
\int_\Omega \phi^{n+1} d\x
\,=\,
\int_\Omega \phi^{n} d\x
\eeq
and the following discrete version of the energy law \eqref{eq:Energylawmod}:
%
%
\beq\label{eq:EnergylawGscheme}
\delta_t E_h(\phi^{n+1})
+ \dis\int_\Omega \left| \sqrt{M^G_\varepsilon(\phi^{n+1})} \nabla\mu^{n+1} \right|^2 d\x
\leq 0\,,
\eeq
where 
$$
E_h(\phi)=\int_\Omega \frac12|\nabla\phi|^2 + \int_\Omega I_h(F(\phi))\,.
$$

\end{lem}

\begin{proof}
We obtain \eqref{eq:conservationG} by taking  $\bar\mu=1$ in \eqref{eq:SchemeG}$_1$. Moreover, by testing equation \eqref{eq:SchemeG}$_2$ by $\bar\phi=\frac1{\Delta t} (\phi^{n+1} - \phi^n)$, equation \eqref{eq:SchemeG}$_1$ by $\bar\mu=\mu^{n+1}$ and using \eqref{eq:eyre} we arrive at \eqref{eq:EnergylawGscheme}.
\end{proof}

\subsubsection{Approximated maximum principle}
\begin{lem}
If $M_h=\mathbb{P}_1$ (hence $\Phi_h=M_h$), then any solution $(\phi^{n+1},\mu^{n+1}) \in \Phi_h\times M_h$ of the  G$_\varepsilon$-scheme \eqref{eq:SchemeG} satisfies the following discrete version of \eqref{eq:Gestimatemod}:
\beq\label{eq:Gschemeestimate}
\ba{c}\dis
\delta_t\left(\int_\Omega I_h(G_\varepsilon(\phi^{n+1})) d\x\right)
+ \dis\int_\Omega I_h\Big((\omega^{n+1})^2\Big) d\x
+ \int_\Omega \left| \sqrt{R(\phi^{n+1})} \nabla\phi^{n+1} \right|^2 d\x
\\ \hueco
\leq\dis
\frac{1}{8\eta^2}
\left(\int_\Omega|\nabla\phi^{n}|^2 d\x
+\int_\Omega|\nabla\phi^{n+1}|^2 d\x\right)\,,
\ea
\eeq
with $\omega^{n+1}=\mu^{n+1} - \big(I_h(F_c'(\phi^{n+1})) + I_h(F_e'(\phi^{n}))\big)$ and $R(\phi)$ being a positive semidefinite diagonal matrix function (related to $F''_c(\phi))$, defined by
$$
R_{kk}\Big|_{I_i} 
\,:=\,
\left\{\ba{cc}\dis
\frac{ F'_c(\phi(\x_k))  -  F'_c(\phi(\x_{0}))}{ \phi(\x_k) - \phi(\x_{0}) }
& 
 \mbox{ if } \phi(\x_{0}) \neq \phi(\x_k)\,,
\\ \hueco
\dis{F''_c(\phi(\x_{0}))} 
& 
\mbox{ if } \phi(\x_{0}) = \phi(\x_k)\,.
\ea\right.
$$
Note that $R_{kk}\Big|_{I_i}\ge 0$ owing to the convexity of $F_c(\phi)$.

Moreover, using \eqref{eq:EnergylawGscheme} we can deduce 
\beq\label{eq:IhGbound}
\int_\Omega I_h(G_\varepsilon(\phi^{n+1})) d\x
\leq
C\,\frac{T}{\eta^2}\,,
\eeq
where $C$ depends on the initial energy $E(\phi^0)$ and $\int_\Omega I_h(G_\varepsilon(\phi^0)) d\x$. 
\end{lem}
\begin{proof}
Testing equation \eqref{eq:SchemeG}$_1$ by $\bar\mu=I_h(G'(\phi^{n+1}))$  and using  property \eqref{eq:MG}
we get: 
$$
\int_\Omega I_h\Big(\delta_t \phi^{n+1} \, G_\varepsilon'(\phi^{n+1})\Big) d\x
\,+\, ( \nabla\mu^{n+1},\nabla\phi^{n+1})
\,=\,0\, .
$$
By introducing the auxiliary variable $\omega^{n+1}=\mu^{n+1} - \big(I_h(F_c'(\phi^{n+1})) + I_h(F_e'(\phi^{n}))\big)\in \mathbb{P}_1$ (because $M_h=\mathbb{P}_1$) we can rewrite \eqref{eq:SchemeG}$_2$ as
$$
(\omega^{n+1},\bar\phi)_h
=(\nabla\phi^{n+1},\nabla\bar\phi)
\quad\pato\bar\phi\in\Phi_h\,,
$$
therefore taking first $\bar\phi=\mu^{n+1}\in \mathbb{P}_1$ and afterwards $\bar\phi=I_h(F_c'(\phi^{n+1})) + I_h(F_e'(\phi^{n}))\in \mathbb{P}_1$, one has 
$$
\ba{rcl}\dis
(\nabla\mu^{n+1},\nabla\phi^{n+1} )
&=&\dis
(\omega^{n+1},\mu^{n+1})_h
\\ \hueco 
&=&\dis
\Big( \omega^{n+1},\omega^{n+1} + I_h(F_c'(\phi^{n+1})) + I_h(F_e'(\phi^{n}))\Big)_h
\\ \hueco 
&=&\dis
\int_\Omega I_h\Big((\omega^{n+1})^2\Big) d\x
+ \int_\Omega \nabla \phi^{n+1}\cdot\nabla\big(I_h(F_c'(\phi^{n+1})) + I_h(F_e'(\phi^{n}))\big) d\x\,,
\ea
$$
hence
\beq\label{eq:proofGeps1}
\dis
\int_\Omega 
I_h\Big(\delta_t \phi^{n+1} \,G_\varepsilon'(\phi^{n+1}) \Big)
+ I_h\Big((\omega^{n+1})^2\Big)
+  \nabla\Big(I_h(F_c'(\phi^{n+1})) 
+ I_h(F_e'(\phi^{n}))\Big)\cdot\nabla \phi^{n+1} 
=0
\,.
\eeq
%
By using Taylor expansion and the convexity of $G_\varepsilon(\phi)$ ($G_\varepsilon''(\phi)\geq 0$), we obtain
\beq\label{eq:proofGeps2}
\dis\int_\Omega I_h\Big(\delta_t\phi^{n+1} \, G_\varepsilon'(\phi^{n+1})\Big) d\x
\ge \delta_t \int_\Omega I_h(G_\varepsilon(\phi^{n+1}))  d\x\,.
\eeq
Because $F'_e$ is linear $(I_h(F'_e(\phi)) = F'_e(\phi))$ we can deduce that
\beq\label{eq:proofGeps3}
-\dis\int_\Omega\nabla I_h(F_e'(\phi^{n})) \cdot\nabla\phi^{n+1} d\x
=\frac1{4\eta^2}\dis\int_\Omega\nabla \phi^{n} \cdot\nabla\phi^{n+1} d\x
\leq
\frac{1}{8\eta^2}
\left(\int_\Omega|\nabla\phi^{n}|^2 d\x
+\int_\Omega|\nabla\phi^{n+1}|^2 d\x\right)\,.
\eeq

Now, taking into account that the mesh is structured and $F''_c(\phi)\geq0$ for all $\phi\in\mathbb{R}$, we consider the $\mathbb{P}_0$ positive semidefinite diagonal matrix function
$R(\phi^{n+1})$ which satisfies
$$
\nabla I_h(F_c'(\phi^{n+1})) = R(\phi^{n+1}) \nabla\phi^{n+1}\,,
$$
which implies
\beq\label{eq:proofGeps4}
\nabla I_h(F_c'(\phi^{n+1}))\cdot \nabla\phi^{n+1}
 =
 \Big( R(\phi^{n+1}) \nabla\phi^{n+1}\Big)\cdot\nabla\phi^{n+1} 
 = \left| \sqrt{R(\phi^{n+1})} \nabla\phi^{n+1} \right|^2 .
\eeq
Therefore, using \eqref{eq:proofGeps2}, \eqref{eq:proofGeps3} and \eqref{eq:proofGeps4} in \eqref{eq:proofGeps1} we obtain \eqref{eq:Gschemeestimate}.
\end{proof}
\begin{coro}\label{lem:Gapprxbounds}
If $\Phi_h=M_h= \mathbb{P}_1$, any solution $(\phi^{n+1},\mu^{n+1})\in \Phi_h\times M_h$  of G$_\varepsilon$-scheme \eqref{eq:SchemeG} satisfies the following estimates:
\beq\label{eq:Glowerbound}
\int_\Omega \big(I_h(\phi_-^{n+1})\big)^2 d\x 
\leq
C \, \frac{\varepsilon(1 - \varepsilon)}{\eta^2}
\leq
C \, \frac{\varepsilon}{\eta^2}
\eeq
and 
\beq\label{eq:Gupperbound}
\int_\Omega \Big(I_h\big((\phi^{n+1} - 1)_+\big)\Big)^2 d\x 
\leq
C \, \frac{\varepsilon(1 - \varepsilon)}{\eta^2}
\leq
C \, \frac{\varepsilon}{\eta^2}\,,
\eeq
where $C$ depends on the initial energy $E(\phi^0)$ and $\int_\Omega I_h(G_\varepsilon(\phi^0)) d\x$. 
\end{coro}
\begin{proof} The main ideas of this proof were 
already used in \cite{KMAD19}.  From Remark~\ref{rk:boundsG} and taking into account that $G_\varepsilon(\phi)\geq0$ for all $\phi\in\mathbb{R}$ we have
$$
G_\varepsilon(\phi)
\geq
\dis\frac{1}{2\varepsilon(1 - \varepsilon)} (\phi_-)^2
\quad\mbox{ and }\quad
G_\varepsilon(\phi)
\geq
\dis\frac{1}{2\varepsilon(1 - \varepsilon)} \big((\phi - 1)_+\big)^2
\quad \pato \phi \in\mathbb{R}\,.
$$
Then, using that 
$$
\big(I_h(\phi)\big)^2\leq I_h(\phi^2)
\quad\pato\phi\in C^0(\overline\Omega)\,,
$$
we obtain
$$
\int_\Omega
I_h\big(G_\varepsilon(\phi^{n+1})\big) d\x
\geq
\dis\frac{1}{2\varepsilon(1 - \varepsilon)} 
\int_\Omega I_h\big((\phi^{n+1}_-)^2\big) d\x
\geq
\dis\frac{1}{2\varepsilon(1 - \varepsilon)} 
\int_\Omega \big(I_h(\phi^{n+1}_-)\big)^2 d\x\,
$$
and
$$
\int_\Omega
I_h\big(G_\varepsilon(\phi^{n+1})\big) d\x
\geq
\dis\frac{1}{2\varepsilon(1 - \varepsilon)} 
\int_\Omega I_h\Big( \big((\phi^{n+1} - 1)_+\big)^2\Big) d\x
\geq
\dis\frac{1}{2\varepsilon(1 - \varepsilon)} 
\int_\Omega \Big(I_h \big((\phi^{n+1} - 1)_+\big)\Big)^2 d\x\,.
$$
Finally, combining previous expressions with estimate \eqref{eq:IhGbound} we obtain \eqref{eq:Glowerbound} and \eqref{eq:Gupperbound}.
\end{proof}

\begin{obs}
G$_\varepsilon$-scheme \eqref{eq:SchemeG} satisfies an approximate maximum principle for $\phi^{n+1}$, because estimates \eqref{eq:Glowerbound} and \eqref{eq:Gupperbound} imply in particular  that
$$
I_h(\phi^{n+1}_-) \rightarrow 0
\quad\mbox{ and }\quad
I_h \big((\phi^{n+1} - 1)_+\big) \rightarrow 0
\quad \mbox{ in }  L^2(\Omega),
\quad \mbox{ as }  \quad\varepsilon\rightarrow0\,,
$$
with $\mathcal{O}\left(\sqrt{\varepsilon}/\eta\right)$ accuracy rate.
\end{obs}

\subsection{J$_\varepsilon$-scheme}
The \textit{J$_\varepsilon$-scheme} is defined as follows: Given $\phi^n \in \Phi_h=\mathbb{P}_1$, find $(\phi^{n+1},\mu^{n+1})\in \Phi_h\times M_h$ such that:
\beq\label{eq:SchemeJ}
\left\{\ba{rcl}\dis
\frac1{\Delta t}(\phi^{n+1} - \phi^n, \bar\mu)_h 
+(M^J_\varepsilon(\phi ^{n+1})\nabla\mu^{n+1},\nabla\bar\mu) 
&=&0\,,
\\ \hueco
(\nabla\phi^{n+1},\nabla\bar\phi)
+ \big(F'_c(\phi^{n+1}) + F'_e(\phi^{n}),\bar\phi\big)
&=&
(\mu^{n+1},\bar\phi)_h\,,
\ea\right.
\eeq
for all $(\bar\phi,\bar\mu)\in \Phi_h\times M_h$ with $M^J_\varepsilon(\phi^{n+1})$ being an approximation of $\left(1/J''_\varepsilon(\phi^{n+1})\right)^2$. 
In particular, for any $\phi\in \Phi_h=\mathbb{P}_1$,  $M^J_\varepsilon(\phi)$ is a diagonal matrix with elements $M^J_{kk}$ for $k=1,\dots,d$. That is, for each element $I_i$ of the triangulation $\{\mathcal{T}_h\}$ (with nodes $\x_0$, $\x_k$ in the $k$-th axis) we compute the $\mathbb{P}_0$ function
\beq\label{eq:defdiscreteMobilityJ}
M^J_{kk}\Big|_{I_i} 
\,:=\,
\left\{\ba{cc}\dis
\left(\frac{ \phi(\x_k) - \phi(\x_0) }
{ J'_\varepsilon(\phi(\x_k))  -  J'_\varepsilon(\phi(\x_0))}
\right)^2
& 
 \mbox{ if } \phi(\x_0) \neq \phi(\x_k)\,,
\\ \hueco
\dis
\left(\frac1{J''_{\varepsilon}(\phi(\x_0))}\right)^2 
& 
\mbox{ if } \phi(\x_0) = \phi(\x_k)\,.
\ea\right.
\eeq
In fact, using that $\Phi_h=\mathbb{P}_1$, the following $ \mathbb{P}_0$ relation holds:
\beq\label{eq:MJ}
 M^J_\varepsilon(\phi)\,\nabla I_h \Big(J'_\varepsilon(\phi)\Big) = \sqrt{ M^J_\varepsilon(\phi)}\nabla \phi\, ,
 \quad \forall\, \phi \in\Phi_h \,.
\eeq

\begin{obs}

$M^J_\varepsilon(\phi)$ can be understood as a non-centered modification of the mobility  $M_\varepsilon(\phi)$ in such a way that both
practically coincide when $\phi$ is located around the center $\phi=1/2$ of the interval $(0,1)$ but they differ when the values of $\phi$ are close to $0$ and $1$ (see  Figure~\ref{fig:ComparionsMgMjM}).  We observe how this  non-centered mobility  $M^J_\varepsilon(\phi)$ is going to zero at the endpoints $\phi=0$ and $\phi=1$ faster than $M_\varepsilon(\phi)$, as $\varepsilon $ goes to zero, which will help to control the boundedness of the unknown $\phi$ between $0$ and $1$ in the numerical scheme.

\end{obs}

\begin{obs}
In order to be able to construct the matrix $M^J_\varepsilon(\phi)$ we need to use the requirement of considering structured triangulations of $\Omega$.
\end{obs}

\subsubsection{Conservation of volume and energy-stability}

\begin{lem} Any solution of the $J_\varepsilon$-scheme \eqref{eq:SchemeJ} satisfies
the conservation property
\beq\label{eq:conservationJ}
\int_\Omega \phi^{n+1} d\x
\,=\,
\int_\Omega \phi^{n} d\x
\eeq
and
 the following discrete version of the energy law \eqref{eq:Energylawmod}:
\beq\label{eq:EnergylawJscheme}
\delta_tE(\phi^{n+1})
\,+\,\dis\int_\Omega  M^J_\varepsilon(\phi^{n+1})|\nabla\mu^{n+1}|^2 d\x
\,\leq\,0\,,
\eeq
where $E(\phi)$ is defined in \eqref{eq:CHenergy}.
\end{lem}
\begin{proof}
We obtain \eqref{eq:conservationJ} by taking  $\bar\mu=1$ in \eqref{eq:SchemeJ}$_1$.
Moreover, testing equation \eqref{eq:SchemeJ}$_1$ by $\bar\mu=\mu^{n+1}$,  equation \eqref{eq:SchemeJ}$_2$ by $\bar\phi=\frac1{\Delta t}(\phi^{n+1} - \phi^n)$ and using \eqref{eq:eyre} we arrive at \eqref{eq:EnergylawJscheme}.
\end{proof}

\subsubsection{Approximated maximum principle}

\begin{lem} If $M_h=\mathbb{P}_1$ (hence $\Phi_h=M_h$), then 
any solution of the $J_\varepsilon$-scheme \eqref{eq:SchemeJ} satisfies 
 the following discrete version of estimate \eqref{eq:Jestimatemod}:
\beq\label{eq:Jschemeestimate}
\delta_t\left(\int_\Omega I_h (J_\varepsilon(\phi^{n+1}))\right)
\,\leq\,
\int_\Omega  M^J_\varepsilon(\phi ^{n+1})|\nabla\mu^{n+1}|^2 d\x
+ \dis\int_\Omega |\nabla\phi^{n+1}|^2 d\x\,.
\eeq
Moreover, using \eqref{eq:EnergylawJscheme} we can deduce 
\beq\label{eq:IhJbound}
\int_\Omega I_h(J_\varepsilon(\phi^{n+1})) d\x
\leq
C_1 + C_2\,T\,,
\eeq
where $C_1$ and $C_2$ depend on the initial energy $E(\phi^0)$ and of $\int_\Omega I_h (J_\varepsilon(\phi^0))$.
\end{lem}
\begin{proof}
Testing equation \eqref{eq:SchemeJ}$_1$ by $\bar\mu=I_h(J'(\phi^{n+1}))$ and using relation \eqref{eq:MJ} we obtain:
$$
\ba{rcl}\dis
\int_\Omega I_h\Big(\delta_t \phi^{n+1} \, J'(\phi^{n+1})\Big)d\x 
&=&
\dis-\int_\Omega  \sqrt{ M^J_\varepsilon(\phi ^{n+1})}\nabla\mu^{n+1}\cdot\nabla\phi^{n+1} d\x
\\ \hueco
&\leq&
\dis\int_\Omega   \left| \sqrt{M^J_\varepsilon(\phi ^{n+1})} \nabla\mu^{n+1} \right|^2 d\x
+ \dis\int_\Omega |\nabla\phi^{n+1}|^2 d\x\,.
\ea
$$
By using Taylor expansion and the convexity of $J_\eps(\phi)$ (as for $G_\eps(\phi)$ in \eqref{eq:proofGeps2}), we deduce \eqref{eq:Jschemeestimate}.
\end{proof}
\begin{coro}\label{lem:Japprxbounds}
Any  solution $(\phi^{n+1},\mu^{n+1})$ of J$_\varepsilon$-scheme \eqref{eq:SchemeJ} satisfies the  following estimates:
\beq\label{eq:Jlowerbound}
\int_\Omega \big(I_h(\phi_-^{n+1})\big)^2 d\x 
\leq
C \, \sqrt{\varepsilon(1 - \varepsilon)}
\leq
C \, \sqrt{\varepsilon}
\eeq
and
\beq\label{eq:Jupperbound}
\int_\Omega \Big(I_h\big((\phi^n - 1)_+\big)\Big)^2 d\x 
\leq
C \, \sqrt{\varepsilon(1 - \varepsilon)}
\leq
C \, \sqrt{\varepsilon}\,,
\eeq
where $C$ depends on the initial energy $E(\phi^0)$ and $\int_\Omega I_h (J_\varepsilon(\phi^0))$. 
\end{coro}
\begin{proof}
Following the same steps presented in Lemma~\ref{lem:Gapprxbounds}, but now using Remark~\ref{rk:boundsJ} and estimate \eqref{eq:IhJbound}.
\end{proof}

\begin{obs}
J$_\varepsilon$-scheme \eqref{eq:SchemeJ} satisfies an approximate minimum principle for $\phi^{n+1}$, because estimates \eqref{eq:Jlowerbound} and \eqref{eq:Jupperbound} imply in particular  that
$$
I_h(\phi^{n+1}_-) \rightarrow 0
\quad\mbox{ and }\quad
I_h \big((\phi^{n+1} - 1)_+\big) \rightarrow 0
\quad \mbox{ in }  L^2(\Omega)
\quad \mbox{ as } \quad \varepsilon\rightarrow0\,,
$$
with $\mathcal{O}\left(\sqrt[4]{\varepsilon}\right)$ accuracy rate.
\end{obs}

\section{Numerical simulations}\label{sec:simulations}
The implementation of the numerical schemes follows the same ideas in any spatial dimension, but for the sake of simplicity to illustrate the properties of the schemes, most of the numerical experiments reported in this section have been carried out in a one dimensional domain $\Omega=(0,1)$ using MATLAB \cite{matlab}. At the end of the section we present one numerical experiment in a $2D$-domain using Freefem++ \cite{FREEFEM} to evidence that the proposed ideas also work in higher dimensions. In fact, although the implementation of $M^G_\varepsilon(\phi)$ and $M^J_\varepsilon(\phi)$ 
might seem rather complicated due to the need of comparing values of the unknown $\phi$ in different nodes of the each element of the triangulation, it can be easily done by checking the values of the corresponding $\mathbb{P}_0$ functions that form $\nabla\phi$.

Moreover, we remind the reader that parameter $\eta$ is associated with the width of the interface, so the size of the spatial mesh should always be small enough to capture the interface, hence one must consider at least $h \le \eta$. On the other hand, parameter $\varepsilon$ determine the separation of the mobility from zero, which does not have any relation with the spatial mesh size. In fact, due to the advancement on computing resources, we can take very small values of the parameter $\varepsilon$ (the lowest in this work is $\varepsilon=10^{-20}$) without being close to the machine precision of the computer.

\subsection{M$_0$-scheme}
In order to better comprehend the properties of the proposed schemes, we have compared them with the corresponding  FE scheme with the truncated by zero mobility $M_0(\phi)$, which has been approximated implicitly in time. The scheme reads: Find $(\phi^{n+1},\mu^{n+1})\in \Phi_h\times M_h$ solving
\beq\label{eq:SchemeM0}
\left\{\ba{rcl}\dis
\frac1{\Delta t}(\phi^{n+1} - \phi^n, \bar\mu)
+(M_0(\phi ^{n+1})\nabla\mu^{n+1},\nabla\bar\mu) 
&=&0\,,
\\ \hueco
(\nabla\phi^{n+1},\nabla\bar\phi)
+ \big(F'_c(\phi^{n+1}) + F'_e(\phi^{n}),\bar\phi\big)
&=&
(\mu^{n+1},\bar\phi)\,,
\ea\right.
\eeq
for all $(\bar\phi,\bar\mu)\in \Phi_h\times M_h$.
This type of FE scheme is the standard by default used in this problem, in order to be sure that even if $\phi$ goes outside of the interval $[0,1]$ the mobility $M_0(\phi)$ will never become negative. Moreover, this scheme is energy stable since the following result holds.
\begin{lem} Any solution of the 
M$_0$-scheme \eqref{eq:SchemeM0} satisfies the following discrete version of the energy law \eqref{eq:Energylawmod}:
$$
\delta_tE(\phi^{n+1})
\,+\,\dis\int_\Omega  M_0(\phi^{n+1})|\nabla\mu^{n+1}|^2 d\x
\,\leq\,0\,.
$$
\end{lem}
\begin{proof}
Testing equation \eqref{eq:SchemeM0}$_1$ by $\bar\mu=\mu^{n+1}$,  equation \eqref{eq:SchemeM0}$_2$ by $\bar\phi=\frac1{\Delta t}(\phi^{n+1} - \phi^n)$ and using \eqref{eq:eyre}.
\end{proof}
\begin{obs}
No results about the boundedness of $\phi^n$ in $[0,1]$ are known for M$_0$-scheme \eqref{eq:SchemeM0}.
\end{obs}

\subsection{Picard iterative algorithms}

Now we describe the iterative algorithms considered to approximate the nonlinear G$_\varepsilon$-scheme \eqref{eq:SchemeG}, J$_\varepsilon$-scheme \eqref{eq:SchemeJ} and 
M$_0$-scheme \eqref{eq:SchemeM0}. 
Let $(\phi^{n},\mu^{n})$ and $(\phi^{\ell},\mu^{\ell})$ be known 
(we consider initially $(\phi^{\ell=0},\mu^{\ell=0})=(\phi^{n},\mu^{n})$), compute 
$(\phi^{\ell + 1},\mu^{\ell + 1})$ such that:
\subsubsection{Iterative algorithm for G$_\varepsilon$-scheme}
$$
\left\{\ba{rcl}
\dis
\frac1{\Delta t}(\phi^{\ell+1}, \bar\mu)_h
+\Big(M^G_\varepsilon(\phi^{l})\nabla\mu^{l+1},\nabla\bar\mu\Big)  
&=&
\dis\frac1{\Delta t}(\phi^{n}, \bar\mu)_h
\,,
\\ \hueco
(\mu^{\ell+1}, \bar\phi)_h 
-(\nabla\phi^{\ell+1},\nabla \bar\phi)
&=&
\Big(I_h(F_c'(\phi^{l})) + I_h(F_e'(\phi^{n})),\bar\phi\Big)_h \,.
\ea\right.
$$

\subsubsection{Iterative algorithm for J$_\varepsilon$-scheme}


$$
\left\{\ba{rcl}\dis
\frac1{\Delta t}(\phi^{\ell+1}, \bar\mu)_h 
+(M^J_\varepsilon(\phi ^{\ell})\nabla\mu^{\ell+1},\nabla\bar\mu) 
&=&
\dis\frac1{\Delta t}(\phi^{n}, \bar\mu)_h \,,
\\ \hueco
(\mu^{n+1},\bar\phi)_h
- (\nabla\phi^{\ell+1},\nabla\bar\phi)
&=&
 (F'_c(\phi^{\ell}) + F'_e(\phi^{n}),\bar\phi) \,.
\ea\right.
$$

\subsubsection{Iterative algorithm for M$_0$-scheme}


$$
\left\{\ba{rcl}\dis
\frac1{\Delta t}(\phi^{\ell+1}, \bar\mu)
+(M_0(\phi ^{\ell})\nabla\mu^{\ell+1},\nabla\bar\mu) 
&=&
\dis\frac1{\Delta t}(\phi^{n}, \bar\mu) \,,
\\ \hueco
(\mu^{n+1},\bar\phi)
-
(\nabla\phi^{\ell+1},\nabla\bar\phi)
&=&
 (F'_c(\phi^{\ell}) + F'_e(\phi^{n}),\bar\phi) \,.
\ea\right.
$$

\

In all cases, we iterate until 
$$
\frac{\|(\phi^{\ell +1} - \phi^{\ell},\mu^{\ell +1} - \mu^{\ell})\|_{L^2\times L^2}}{\|(\phi^{\ell +1},\mu^{\ell +1})\|_{L^2\times L^2}} \leq tol\,=\,10^{-8}\,.
$$

\subsection{Example I. Discrete boundedness}
In the first example we study how each of the schemes behave when the values of $\phi$ are close to the endpoints of the interval $[0,1]$. To this end we have considered as initial condition the configuration presented in Figure~\ref{fig:exI_initial_final}, that is, two ``balls'' that are defined as:
\beq\label{eq:initialexa1}
\phi_0(x)=
-\frac12\left[\tanh\left(\frac{(x-0.3)-0.08}{\sqrt{2}\eta}\right) 
+\tanh\left(\frac{(x-0.72)-0.15}{\sqrt{2}\eta}\right)\right]
+ 1\,.
\eeq

\begin{figure}
\begin{center}
\includegraphics[scale=0.110]{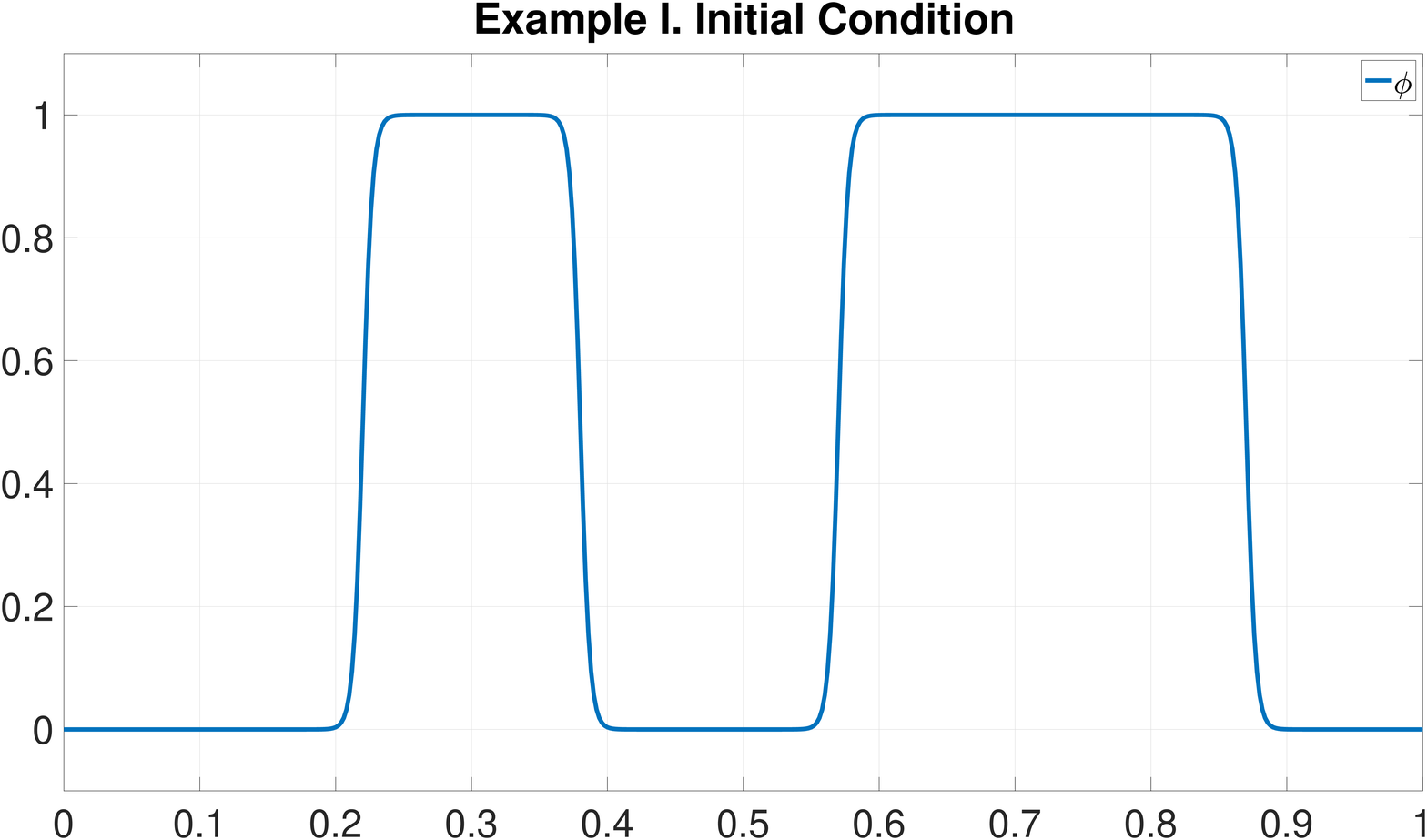}
\includegraphics[scale=0.110]{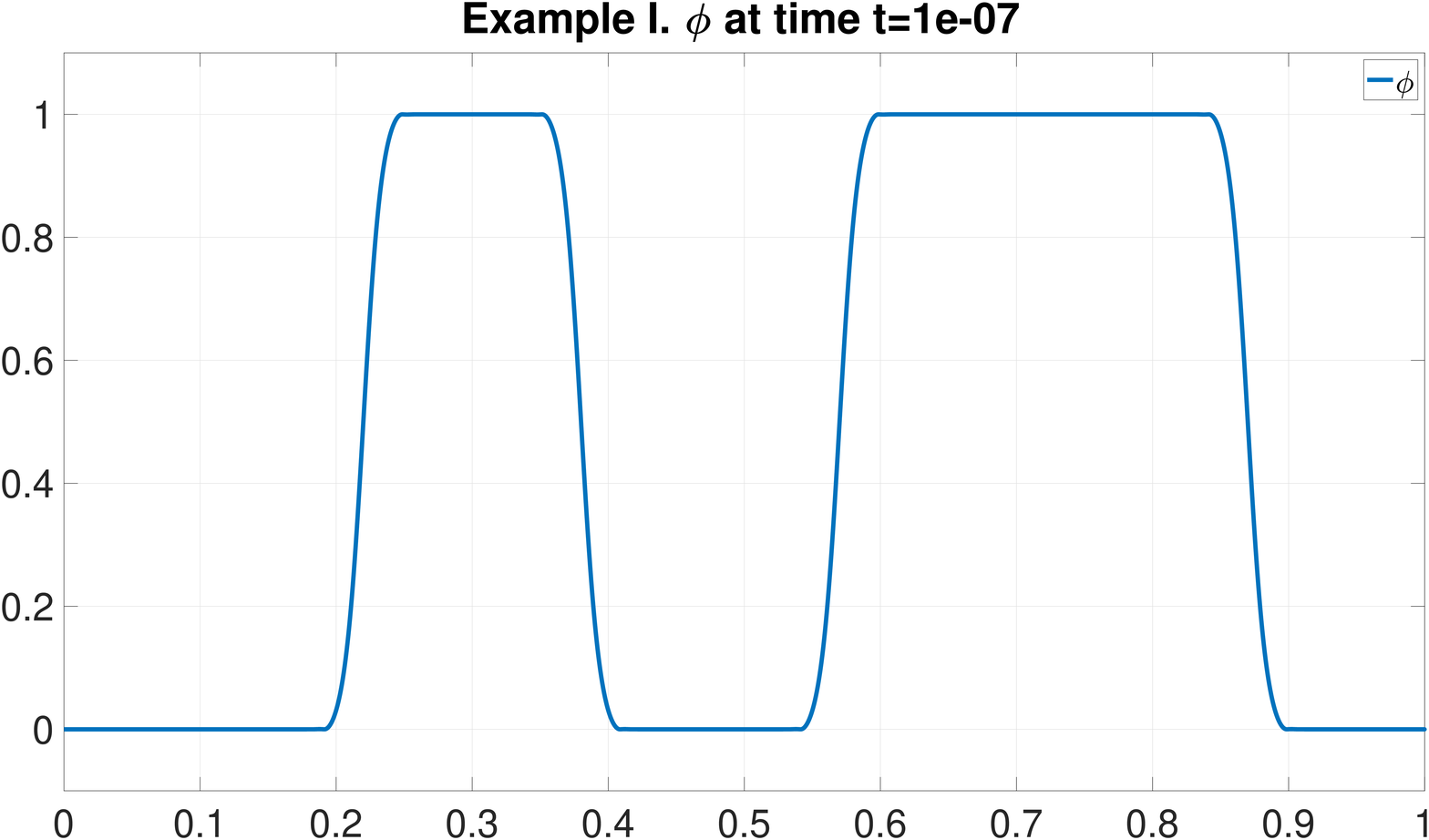}
\end{center}
\caption{Example I. Left: Initial condition $\phi_0(x)$ at $t=0$. Right: $\phi(x,t)$ at $t=10^{-7}$.}\label{fig:exI_initial_final}
\end{figure}

In all the simulations in this section we have considered the time interval $[0,T]=[0,10^{-7}]$ and time step $\Delta t=10^{-10}$. The expected dynamic of the system is that the pure phases $\phi=0$ and $\phi=1$  will not merge in a larger pure phase region because they are too far apart to interact between them, in fact they will just accommodate the width of the interface to reach an equilibrium. This dynamic is completely different from the one associated with constant mobility, where the two pure phases will end up forming larger regions due to the coarsening effects. 

We have separated the presentation of the results into two parts: first we present the comparison of the three schemes with fixed values of $\eta$ and $\varepsilon$ while varying the size of the spatial discretization, with $N (=h^{-1})$ denoting the number of points in the (equally distributed) spatial mesh. In the second step we compare the G$_\varepsilon$ and J$_\varepsilon$ schemes with fixed value of $N$ and varying the values of $\eta$ and $\varepsilon$ in order to see if the estimates derived in Lemmas~\ref{lem:Gapprxbounds} and \ref{lem:Japprxbounds} become apparent in our numerical experiments.

\subsubsection{Comparison of G$_\varepsilon$, J$_\varepsilon$ and M$_0$ schemes. Fixed $\eta=0.005$ and $\varepsilon=10^{-20}$}
In Figure~\ref{fig:ExI_part1} we present comparisons of the evolution in time of the maximum and minimum of $\phi$ for each scheme when different size of the discretization of the spatial mesh is considered (i.e. different values of the number of points $N$). In particular,
since $\eta=0.005$ we focus on cases such that $N\ge 200$, hence $h\le \eta$.

We can observe how M$_0$-scheme is always much less effective in maintaining the values of $\phi$ inside the interval $[0,1]$ than the G$_\varepsilon$ and J$_\varepsilon$ schemes. Moreover, 
the larger $N$
the better the behavior of M$_0$-scheme with respect to the boundedness of $\phi$, achieving equivalent performance to the G$_\varepsilon$ and J$_\varepsilon$ schemes only for the demanding requirement  $N\geq10^4$. This fact is expected because taking $h$ to zero we are getting closer to the only discrete in time version of scheme  \eqref{eq:SchemeM0}, which satisfies the boundedness estimates.

\begin{figure}
\begin{center}
\includegraphics[scale=0.1125]{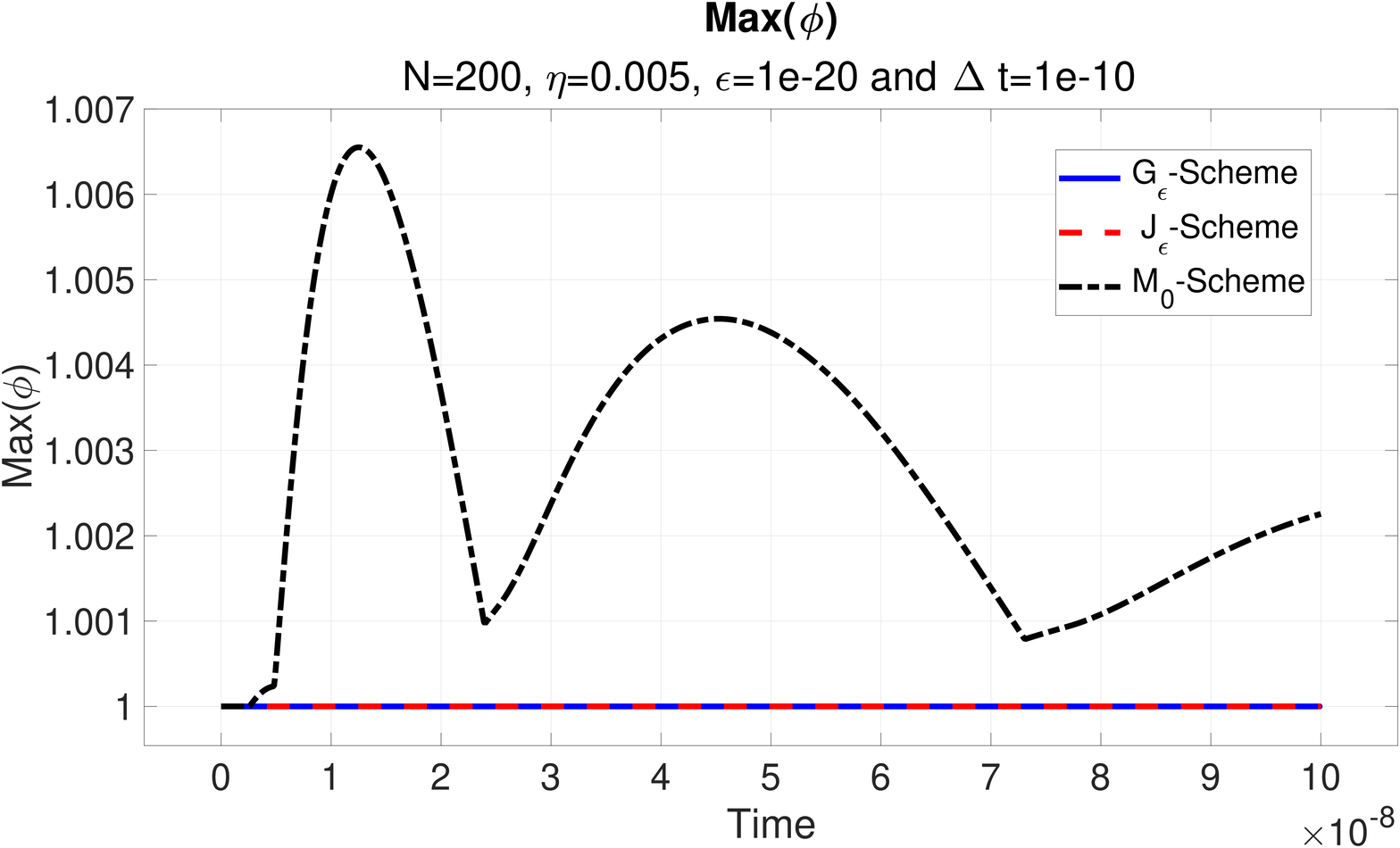}
\includegraphics[scale=0.1125]{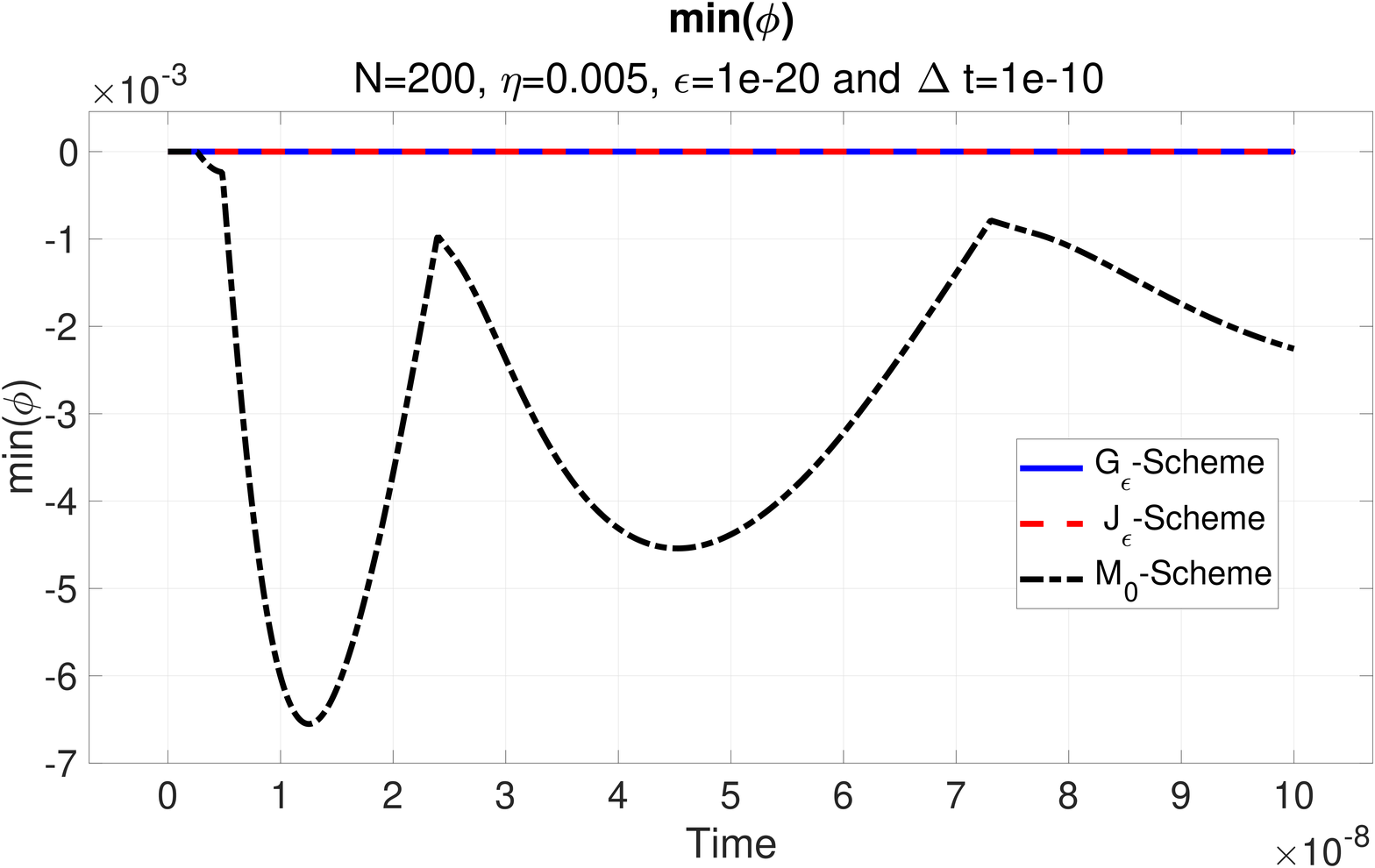}
\\
\includegraphics[scale=0.1125]{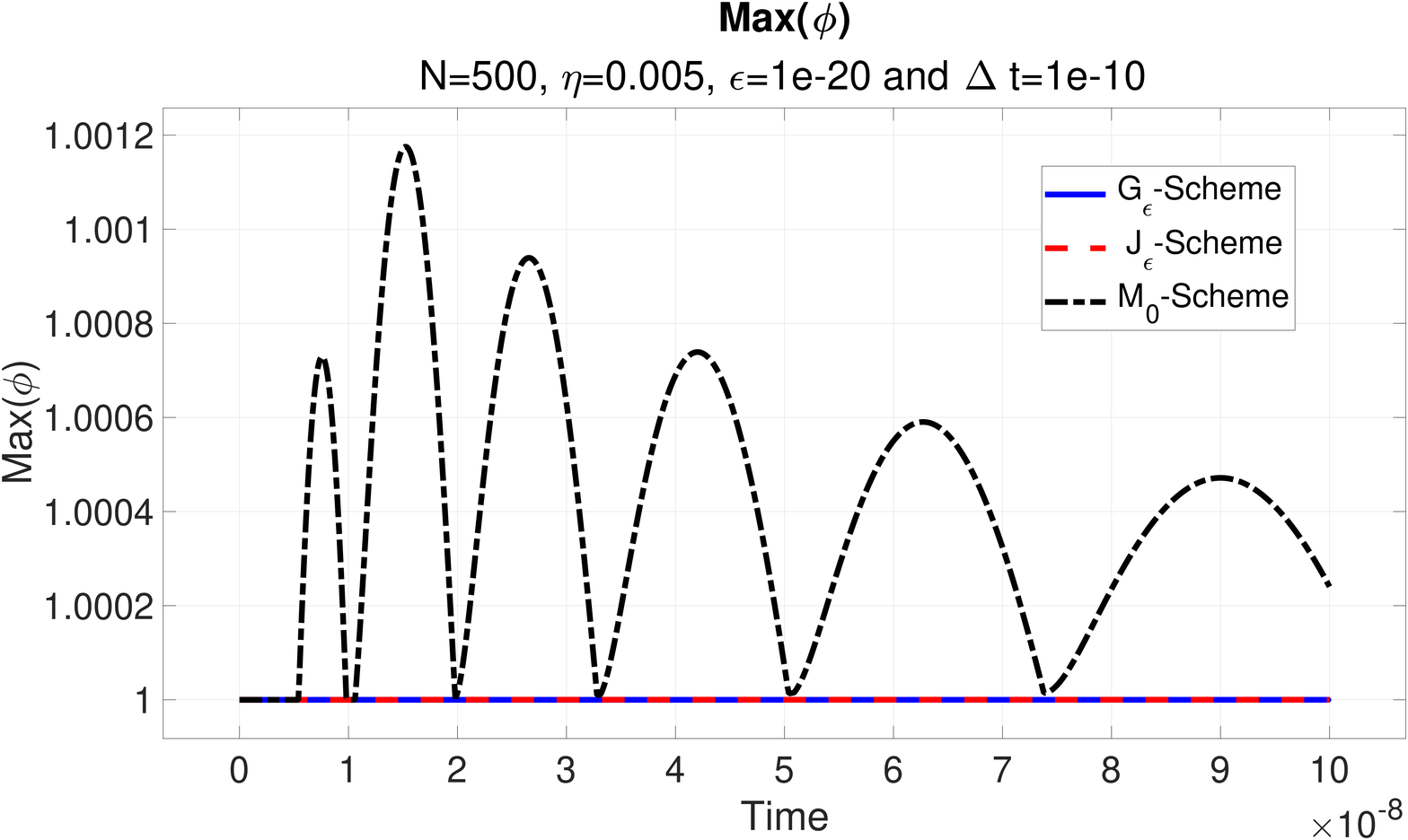}
\includegraphics[scale=0.1125]{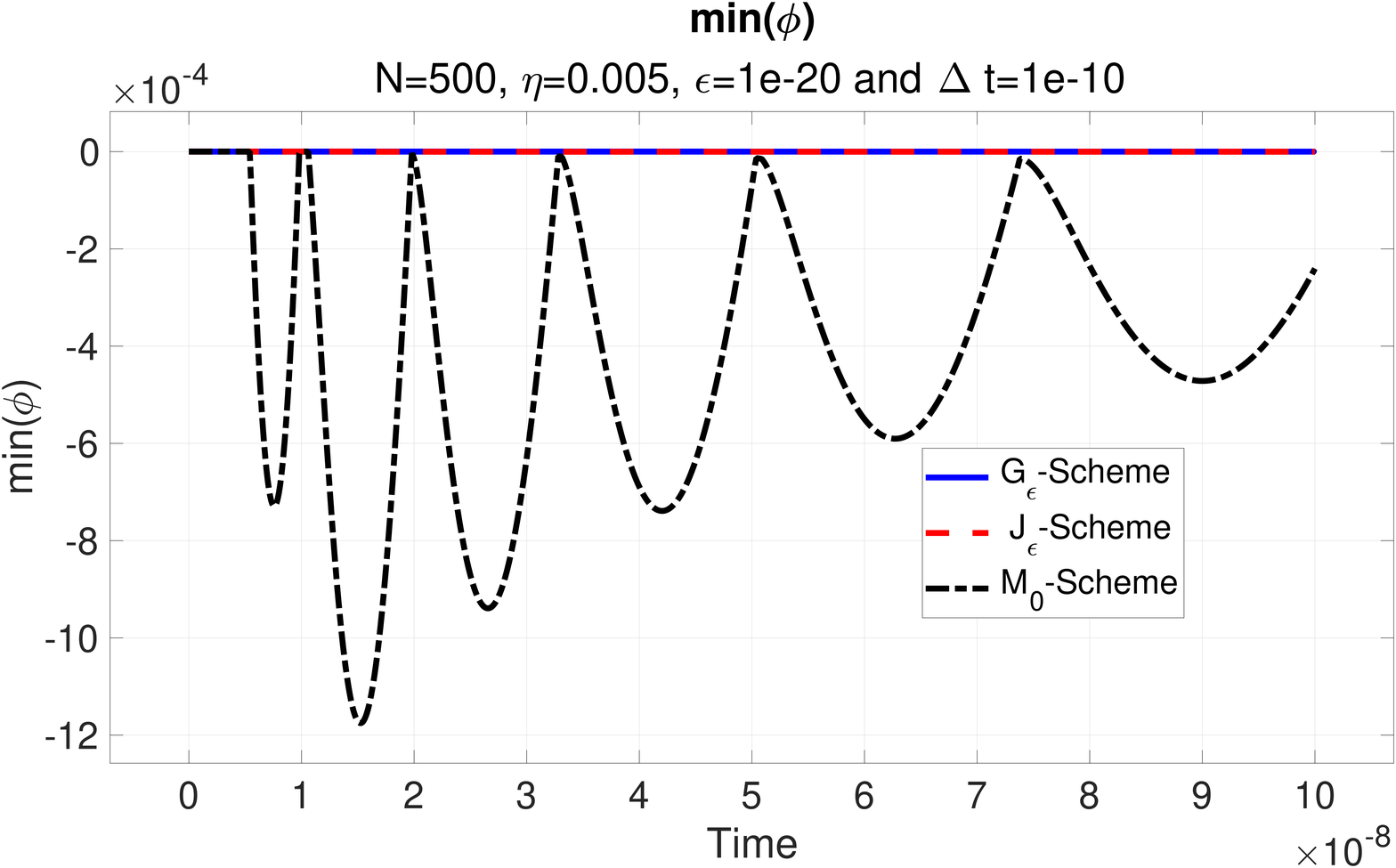}
\\
\includegraphics[scale=0.1125]{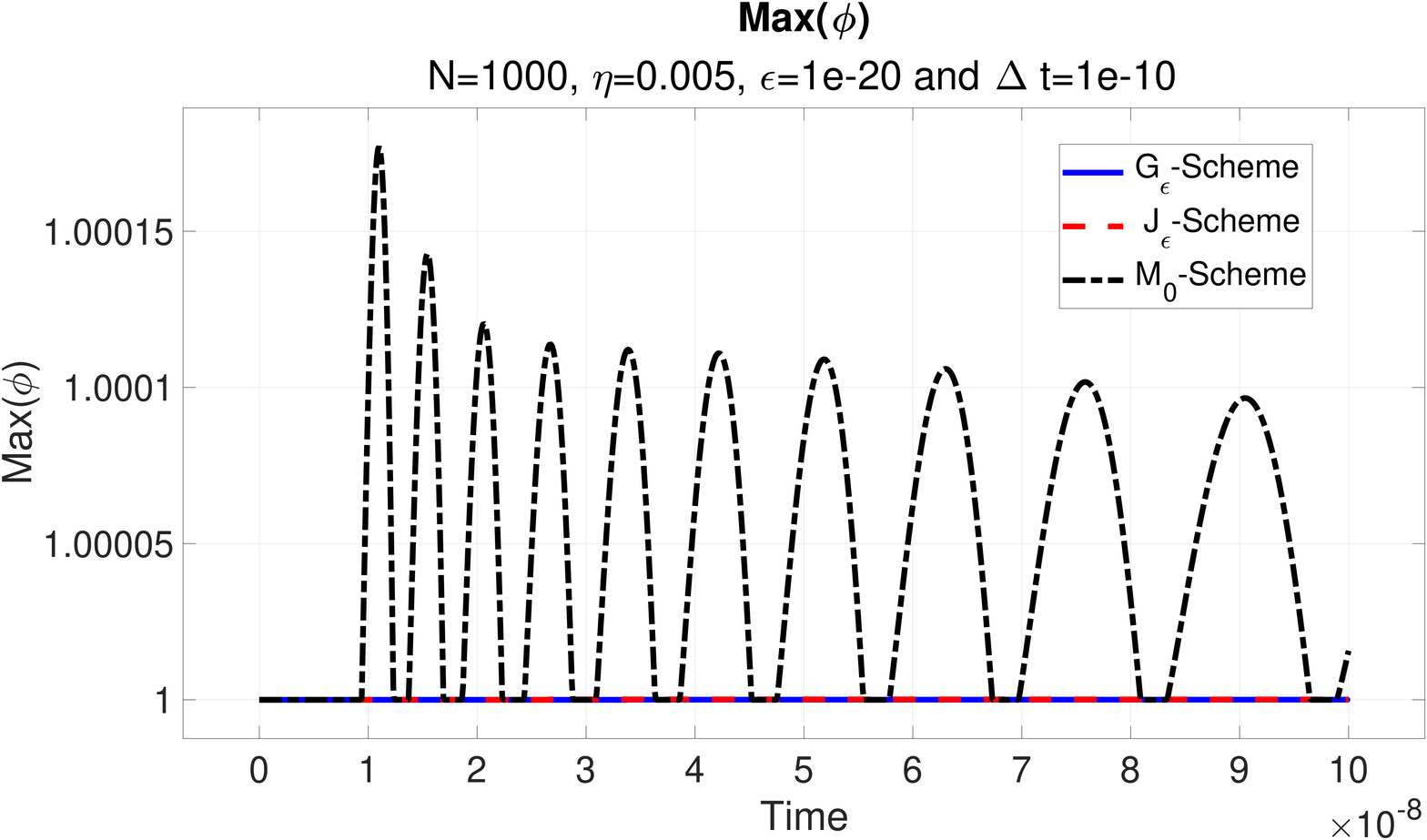}
\includegraphics[scale=0.1125]{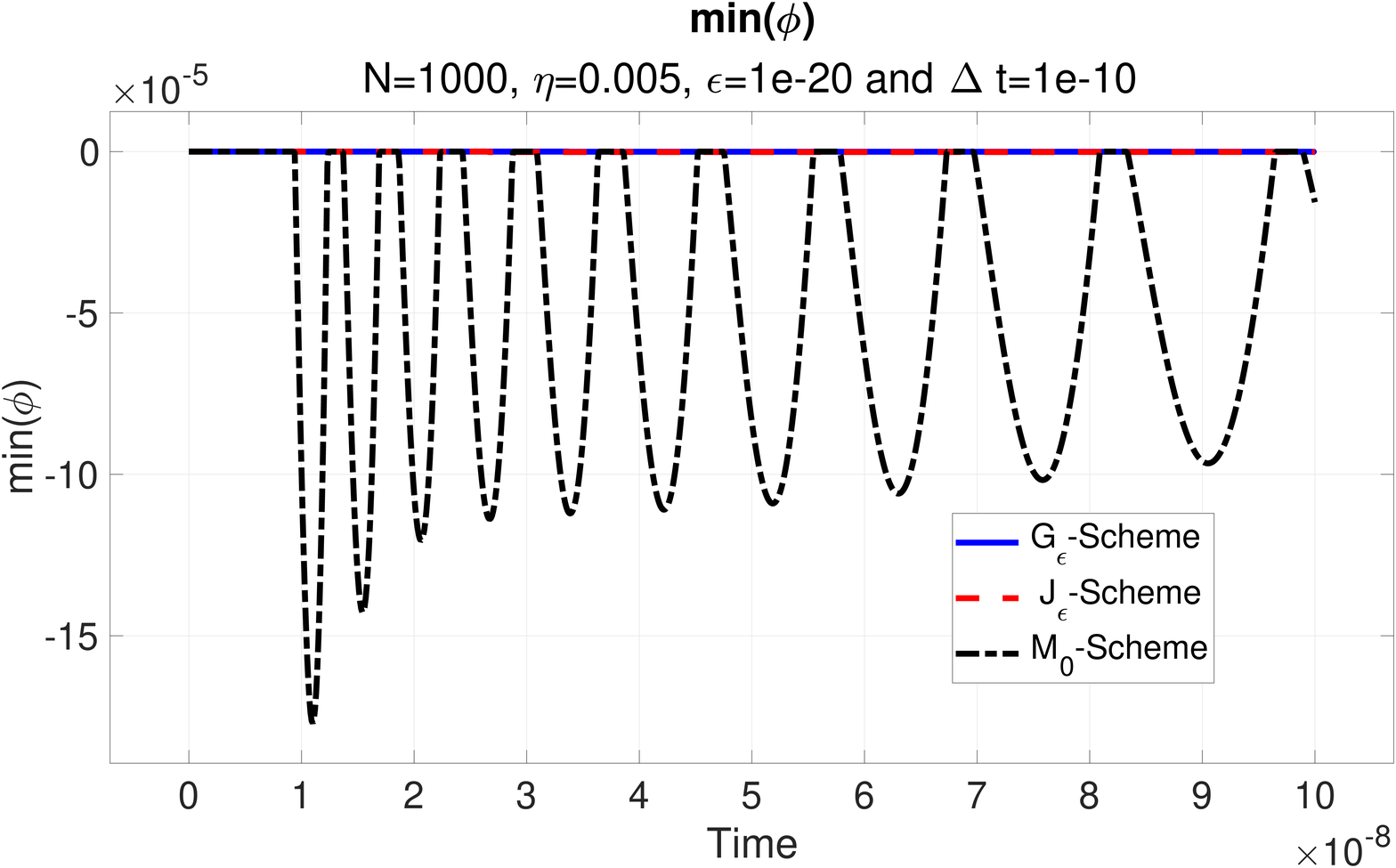}
\\
\includegraphics[scale=0.1125]{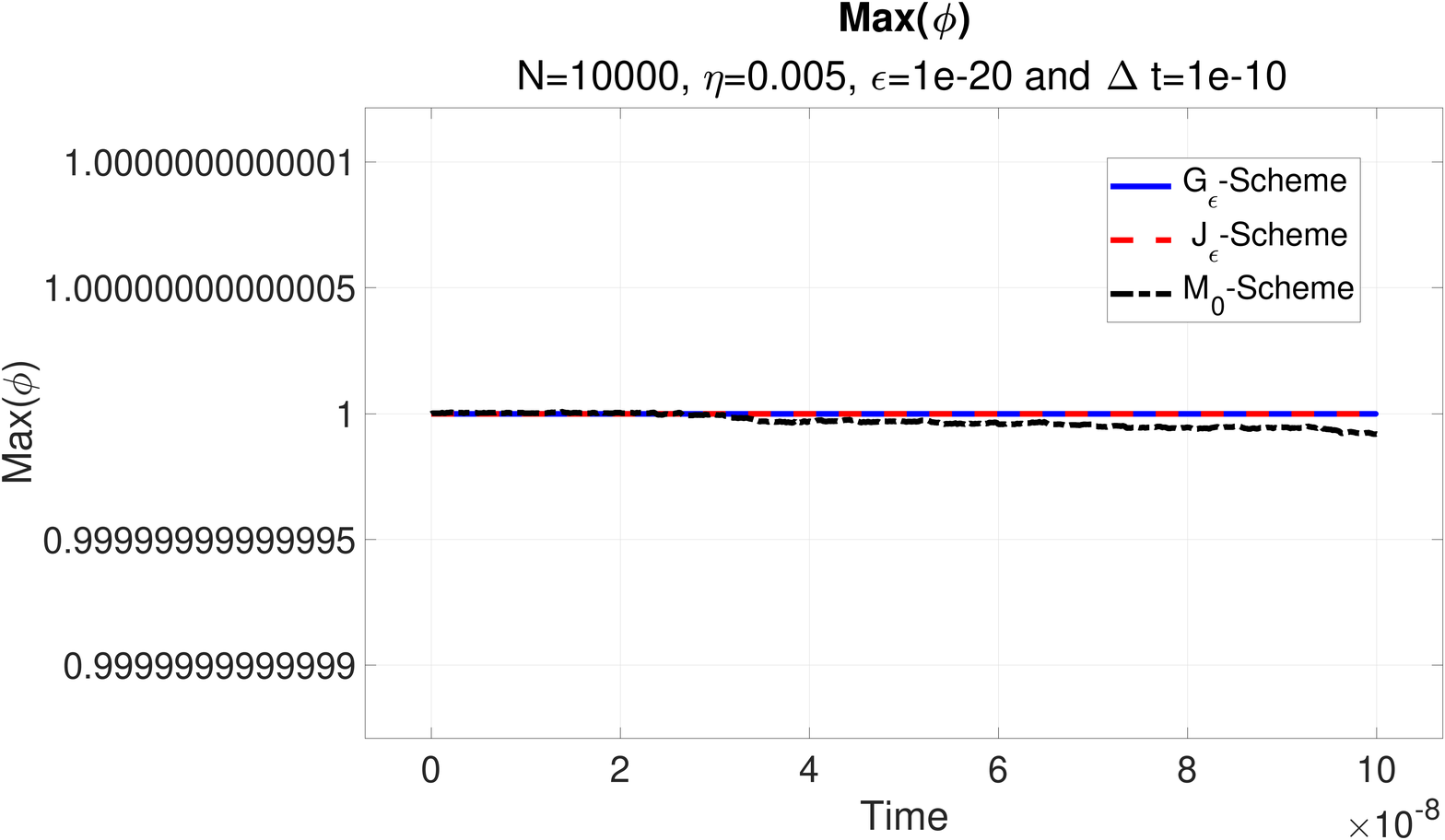}
\includegraphics[scale=0.1125]{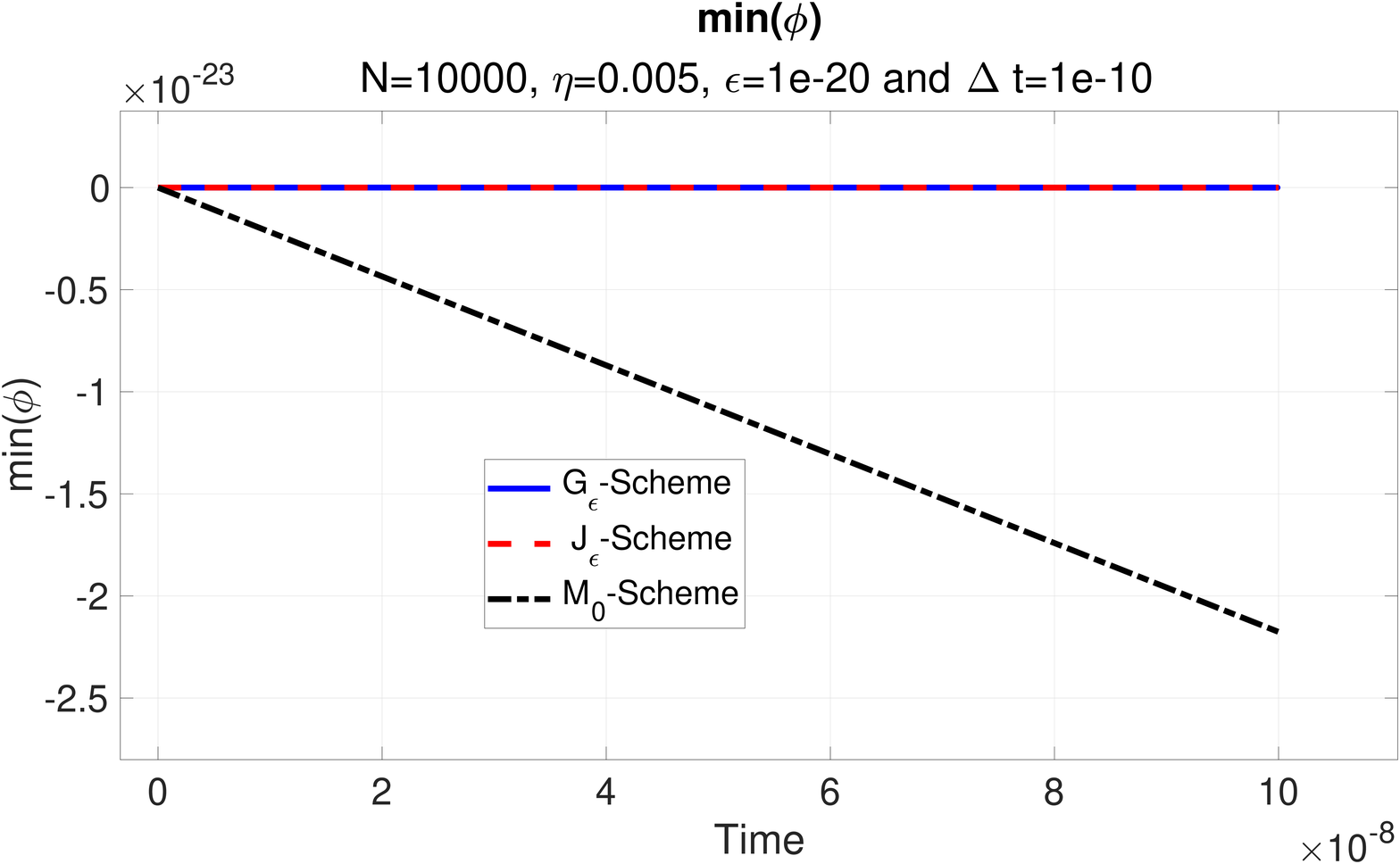}
\end{center}
\caption{Example I. Evolution in time of $\max(\phi)$ (left column) and $\min(\phi)$ (right column) for G$_\varepsilon$-scheme, J$_\varepsilon$-scheme and M$_0$-scheme considering $N=200, 500, 1000$ and $10000$ (from top to bottom).}\label{fig:ExI_part1}
\end{figure}

%

\subsubsection{Study of the influence of $\eta$ and $\varepsilon$ in the boundedness estimates (fixed $N=10^4$ and $\Delta t=10^{-10}$)}

We present in Tables~\ref{tab:Gscheme_boundedness} and \ref{tab:Jscheme_boundedness} the results obtained using G$_\varepsilon$ and J$_\varepsilon$ schemes with different values  of $\eta$ and $\varepsilon$. For $\eta=0.01$ (the largest value of $\eta$ considered) both schemes achieve $\phi<1$ and very small negative  values of $\phi$ when $\varepsilon\leq10^{-3}$. Decreasing the value of $\eta$ from $\eta=0.01$ to $\eta=0.005$ produces that the value of $\varepsilon$ needed to achieve equivalent sharpness of the bounds also decreases, in particular, it is needed to consider $\varepsilon\leq10^{-4}$ with $\eta=0.075$ and $\varepsilon\leq10^{-5}$ with $\eta=0.005$ for both schemes. The main difference here is that larger values of $\varepsilon$ using G$_\varepsilon$-scheme result in no reliable simulations (for some of the larger values of $\varepsilon$ the iterative algorithm does not even converge).
On the other hand, for J$_\varepsilon$-scheme the obtained results are not as sharp as for smaller values of $\varepsilon$, although convergence of the algorithm is achieved in all cases. The different convergence behavior for both schemes can be explained by the boundedness estimates in Corollaries \ref{lem:Gapprxbounds} and \ref{lem:Japprxbounds}, because the bounds for G$_\varepsilon$-scheme degenerate as $\eta$ decreases while the bounds for J$_\varepsilon$-scheme are independent of $\eta$ .

When the value of $\eta$ is reduced even further the situation become more challenging, because a smaller value of $\eta$ results in a narrower interface thickness while using the same discretization parameters $N=10^4$ and $\Delta t =10^{-10}$. For the cases $\eta=0.0025$ and $\eta=0.001$, G$_\varepsilon$-scheme needs to consider $\varepsilon\leq10^{-6}$ to be reliable (again for some of the larger values of $\varepsilon$ the iterative algorithm does not converge). By comparison, J$_\varepsilon$-scheme seems to be much more reliable in these situations for  $\varepsilon\leq10^{-2}$, although the obtained bounds in J$_\varepsilon$-scheme are not as sharp as the ones obtained with larger values $\eta$.

The results of these numerical experiments corroborate that the boundedness estimates for both schemes improve when considering lower values of $\varepsilon$ and in the case of G$_\varepsilon$-scheme the truncation parameter $\varepsilon$ should be smaller than $\eta^2$ to be effective.

\begin{table}
\begin{tabular}{|c|cc|cc|cc|}                      
\hline                                                                    
G$_\varepsilon(\phi)$&  \multicolumn{2}{c|}{$\eta=0.01$}  & \multicolumn{2}{c|}{$\eta=0.0075$} & \multicolumn{2}{c|}{$\eta=0.005$} 
\\ 
\hline                                                                    
$\varepsilon$& $\dis\max_{t\in[0,T]}\max_{x\in\Omega}\phi$  & $\dis\min_{t\in[0,T]}\min_{x\in\Omega}\phi$  & $\dis\max_{t\in[0,T]}\max_{x\in\Omega}\phi$  & $\dis\min_{t\in[0,T]}\min_{x\in\Omega}\phi$ & $\dis\max_{t\in[0,T]}\max_{x\in\Omega}\phi$  & $\dis\min_{t\in[0,T]}\min_{x\in\Omega}\phi$ 
\\ \hline                                                                                                                                                               
$10^{-2}$ & $1.00001$ & $  -1.9\times10^{-3}  $ & $1.0004 $ & $ -6.0\times10^{-3} $ & $1.0033 $ & $ -7.05\times10^{-3} $ 
\\ \hline
$10^{-3}$ & $0.99999$ & $ -3.07\times10^{-14} $ & $\times $ & $ \times$ & $\times $ & $ \times $  
\\ \hline
$10^{-4}$ & $0.99999$ & $  -3.07\times10^{-14} $ & $0.99999 $ & $ -7.33\times10^{-18} $ & $\times $ & $ \times $ 
\\ \hline
$10^{-5}$ & $0.99999$ & $ -3.07\times10^{-14}$ & $ 0.99999$ & $ -5.79\times10^{-18} $ & $1.0 + 10^{-14} $ & $ -6.73\times10^{-18} $ 
\\ \hline
$10^{-6}$ & $0.99999$ & $ -3.07\times10^{-14} $ & $0.99999 $ & $ -5.85\times10^{-18} $ & $1.000000 $ & $ -5.81\times10^{-18} $  
\\ \hline
$10^{-7}$ & $0.99999$ & $ -3.07\times10^{-14} $ & $0.99999 $ & $ -5.76\times10^{-18} $ & $1.000000 $ & $ -5.84\times10^{-18} $ 
\\ \hline
$10^{-8}$ & $0.99999$ & $ -3.07\times10^{-14} $ & $0.99999 $ & $ -5.72\times10^{-18} $ & $1.000000 $ & $ -5.72\times10^{-18} $ 
\\ \hline
\end{tabular}
\begin{tabular}{|c|cc|cc|}       
\hline               
 G$_\varepsilon(\phi)$& \multicolumn{2}{c|}{$\eta=0.0025$} & \multicolumn{2}{c|}{$\eta=0.001$} 
 \\ \hline    
 $\varepsilon$&$\dis\max_{t\in[0,T]}\max_{x\in\Omega}\phi$  & $\dis\min_{t\in[0,T]}\min_{x\in\Omega}\phi$  & $\dis\max_{t\in[0,T]}\max_{x\in\Omega}\phi$  & $\dis\min_{t\in[0,T]}\min_{x\in\Omega}\phi$ 
\\ \hline                                                                                                                                                               
$10^{-2}$ & $1.024$ & $  - 6.9\times10^{-3}  $ & $\times$ & $ \times $
\\ \hline
$10^{-3}$ & $1.0008$ & $ - 2.0\times10^{-3} $ & $\times $ & $ \times$ 
\\ \hline
$10^{-4}$ & $\times$ & $  \times $ & $\times $ & $  \times  $\\ \hline
$10^{-5}$ & $1.0943$ & $ - 5.6\times10^{-5}$ & $1.0244 $ & $  -6.38\times10^{-5} $ 
\\ \hline
$10^{-6}$ & $1.000009$ & $ -9.72 \times10^{-6} $ & $1.0007 $ & $ -1.08\times10^{-5}  $
\\ \hline
$10^{-7}$ & $1.0000014$ & $ - 1.43\times10^{-6} $ & $1.000024 $ & $ -6.07\times10^{-6}  $
\\ \hline
$10^{-8}$ & $1.0000012$ & $ -1.26 \times10^{-6} $ & $1.000006 $ & $ -6.73\times10^{-6}  $
\\ \hline
\end{tabular}
\caption{Example I. G$_\varepsilon$-scheme. Evolution in time of the minimum and maximum values achieved by $\phi$ in the domain $\Omega$ in the whole interval $[0,T]$ for different values of $\eta$ and $\varepsilon$ with fixed values of the discretization parameters $N=10000$ and $\Delta t=10^{-10}$. Symbol $\times$ denotes that at some point the iterative method did not converge.}\label{tab:Gscheme_boundedness}
\end{table}
\begin{table}
\begin{tabular}{|c|cc|cc|cc|}                      
\hline                                                                    
J$_\varepsilon(\phi)$&  \multicolumn{2}{c|}{$\eta=0.01$}  & \multicolumn{2}{c|}{$\eta=0.0075$} & \multicolumn{2}{c|}{$\eta=0.005$} 
\\ 
\hline                                                                    
$\varepsilon$& $\dis\max_{t\in[0,T]}\max_{x\in\Omega}\phi$  & $\dis\min_{t\in[0,T]}\min_{x\in\Omega}\phi$  &$\dis\max_{t\in[0,T]}\max_{x\in\Omega}\phi$  & $\dis\min_{t\in[0,T]}\min_{x\in\Omega}\phi$ & $\dis\max_{t\in[0,T]}\max_{x\in\Omega}\phi$  & $\dis\min_{t\in[0,T]}\min_{x\in\Omega}\phi$
\\ \hline                                                                                                                                                               
$10^{-2}$ & $1.00195$ & $  -1.9\times10^{-3}  $ & $1.00594 $ & $ -5.94\times10^{-3} $ & $1.00693 $ & $ -6.93\times10^{-3} $ 
\\ \hline
$10^{-3}$ & $0.999999$ & $  -3.7\times10^{-14}  $ & $1.00112 $ & $ -1.12\times10^{-3} $ & $1.00201 $ & $ -2.01\times10^{-3} $  
\\ \hline
$10^{-4}$ & $0.999999$ & $  -3.7\times10^{-14}  $ & $0.999999 $ & $ -7.31\times10^{-18} $ & $1.00022 $ & $ -2.26\times10^{-4} $ 
\\ \hline
$10^{-5}$ & $0.999999$ & $  -3.7\times10^{-14}  $ & $0.999999 $ & $ -5.79\times10^{-18} $ & $1.0 + 7\times10^{-15} $ & $ -6.73\times10^{-18} $ 
\\ \hline
$10^{-6}$ & $0.999999$ & $  -3.7\times10^{-14}  $ & $0.999999 $ & $ -5.85\times10^{-18} $ & $1.00000 $ & $ -5.81\times10^{-18} $   
\\ \hline
$10^{-7}$ & $0.999999$ & $  -3.7\times10^{-14}  $ & $0.999999 $ & $ -5.76\times10^{-18} $ & $1.00000 $ & $ -5.84\times10^{-18} $  
\\ \hline
$10^{-8}$ & $0.999999$ & $  -3.7\times10^{-14}  $ & $0.999999 $ & $ -5.72\times10^{-18} $ & $1.00000 $ & $ -5.72\times10^{-18} $  
\\ \hline
\end{tabular}
\begin{tabular}{|c|cc|cc|}       
\hline               
 J$_\varepsilon(\phi)$& \multicolumn{2}{c|}{$\eta=0.0025$} & \multicolumn{2}{c|}{$\eta=0.001$} 
 \\ \hline    
 $\varepsilon$& $\dis\max_{t\in[0,T]}\max_{x\in\Omega}\phi$  & $\dis\min_{t\in[0,T]}\min_{x\in\Omega}\phi$  &$\dis\max_{t\in[0,T]}\max_{x\in\Omega}\phi$  & $\dis\min_{t\in[0,T]}\min_{x\in\Omega}\phi$ 
\\ \hline                                                                                                                                                               
$10^{-2}$ & $1.00679$ & $  -6.79\times10^{-3}  $ & $1.0044 $ & $ -0.0044 $ 
\\ \hline
$10^{-3}$ & $ 1.00203$ & $  -2.03\times10^{-3}  $ & $1.0017 $ & $ -0.0017 $ 
\\ \hline
$10^{-4}$ & $1.00033$ & $  -3.39\times10^{-4}  $ & $1.0004 $ & $ -3.66\times10^{-4}  $ 
\\ \hline
$10^{-5}$ & $1.000054$ & $  -5.45\times10^{-5}  $ & $1.0001 $ & $ -6.47\times10^{-5} $
\\ \hline
$10^{-6}$ & $ 1.000009$ & $  -9.7\times10^{-6}  $ &$1.000012 $ & $ -1.207\times10^{-5} $
\\ \hline
$10^{-7}$ & $1.000001$ & $  -1.43\times10^{-6}  $ & $1.000015 $ & $ -1.58\times10^{-5} $ 
\\ \hline
$10^{-8}$ & $1.000001$ & $  -1.25\times10^{-6}  $ & $1.0000088 $ & $ -8.81\times10^{-6} $
\\ \hline
\end{tabular}
\caption{Example I. J$_\varepsilon$-scheme. Evolution in time of the minimum and maximum values achieved by $\phi$ in the domain $\Omega$ in the whole interval $[0,T]$ for different values of $\eta$ and $\varepsilon$ with fixed values of the discretization parameters $N=10000$ and $\Delta t=10^{-10}$.}\label{tab:Jscheme_boundedness}\end{table}

\subsection{Example II. Accuracy study}\label{sec:sim_orderconver}
In this second example we estimate numerically the order of convergence in space of the presented numerical schemes. We
consider the parameters values $\eta=0.005$ and $\varepsilon=10^{-20}$. We now introduce some additional notation. The individual errors using discrete norms and the convergence rate between two 
spatial meshes of size $h$ and $\widetilde{h}$ are defined as
$$
e_2(\phi):=\|\phi_{exact} - \phi_h\|_{L^2(\Omega)}\,,
\quad
\mbox{ and }
\quad
r_2(\phi):=\left[\log\left(\frac{e_2(\phi)}{\tilde{e}_2(\phi)}\right)\right]\Big/ \left[\log\left(\frac{h}{\widetilde{h}}\right)\right]\,.
$$

We consider the time step set to $\Delta t=10^{-10}$ and the time interval $[0,T]=[0,10^{-7}]$. The initial configuration  $\phi_0$ is determined by \eqref{eq:initialexa1}, the same configuration that was used in Example I (initial and final configuration of $\phi$ are presented in Figure~\ref{fig:exI_initial_final}).
We have considered this challenging configuration where there are regions with $\phi=0$ and $\phi=1$.
We compute the EOC using as reference (or exact) solution the one obtained by solving the system using the spatial mesh size $h=10^{-5}$ using G$_\varepsilon$-scheme. The errors and orders of convergence are presented in Table~\ref{tab:Exa_II_Exp_II_order}. The results from the two experiments show that  G$_\varepsilon$-scheme and J$_\varepsilon$-scheme achieve the optimal order of $\mathcal{O}(h^2)$ as the classical $M_0(\phi)$-scheme, that is, the construction of the proposed boundedness preserving numerical schemes do not cause to lose order of convergence.

\begin{table}
\begin{tabular}{|c|cc|cc|cc|}                      
\hline                                                                    
&  \multicolumn{2}{c|}{G$_\varepsilon$-Scheme}  & \multicolumn{2}{c|}{J$_\varepsilon$-Scheme}  & \multicolumn{2}{c|}{M$_0$-Scheme} 
\\ 
\hline                                                                    
$N (=h^{-1})$& $e_2(\phi)$  & $r_2(\phi)$  & $e_2(\phi)$ & $r_2(\phi)$ & $e_2(\phi)$ & $r_2(\phi)$ 
\\ \hline                                                                                                                                                               
$2000$ 
& $0.000013314$ & $ - $ 
& $0.000008683$ & $ - $ 
& $0.000009283$ & $ - $
\\ \hline
$3000$ 
& $0.000005907$ & $2.004047$ 
& $0.000003863$ & $1.997343$ 
& $0.000004086$ & $2.024024$
\\ \hline
$4000$ 
& $0.000003330$ & $1.992257$ 
& $0.000002192$ & $1.970207$ 
& $0.000002309$ & $1.982729$
\\ \hline
$5000$ 
& $0.000002142$ & $1.976693$ 
& $0.000001416$ & $1.957809$ 
& $0.000001494$ & $1.950956$
\\ \hline
$6000$ 
& $0.000001500 $ & $1.953188$ 
& $0.000000990 $ & $1.963118$ 
& $0.000001053 $ & $1.917623$ 
\\ \hline
\end{tabular}
\caption{Example II.  Experimental absolute errors and accuracy order.
}\label{tab:Exa_II_Exp_II_order}
\end{table}

\subsection{Example III. Spinodal decomposition}

In this example we illustrate the validity of the schemes to capture realistic dynamics. To this end we have only used G$_\varepsilon$-scheme with parameter $\varepsilon=10^{-20}$, because we have seen in Example I that with this choice of the truncation parameter both schemes G$_\varepsilon$ and J$_\varepsilon$ behaves equivalently. The time interval is $[0,10^{-3}]$ and the rest of the considered parameters are $\eta=0.005$, $N=10^3$ and $\Delta t=10^{-8}$
We have simulated a spinodal decomposition dynamics, that is, we consider initially that the two components of our system are very mixed by taking as initial condition a random perturbation of amplitude $0.01$ of the constant value $\phi=0.5$. In order to decrease the energy of the system, the dynamics are expected to pull the values of $\phi$ close to the end points of the interval $[0,1]$ and at the same time try to reduce the amount of interface in the system. We can observe in Figure~\ref{fig:ExIII_dynamic} how the results of our simulations correspond with the expected dynamics. Moreover we can observe in Figure~\ref{fig:ExIII_plots} the evolution in time of: the energy (always decreasing), the volume (always constant),  $\dis\max_{x\in\Omega}\phi$ (always $\leq 1$) and $\dis\min_{x\in\Omega}\phi$ (always $\geq 0$), in fact in this example
$$
\max_{t\in[0,10^{-3}]}\max_{x\in\Omega}\phi= 0.999624
\quad\quad\quad
\mbox{ and }
\quad\quad\quad
\min_{t\in[0,10^{-3}]}\min_{x\in\Omega}\phi=1.58\times10^{-4}\,.
$$

\begin{figure}
\begin{center}
\includegraphics[scale=0.114]{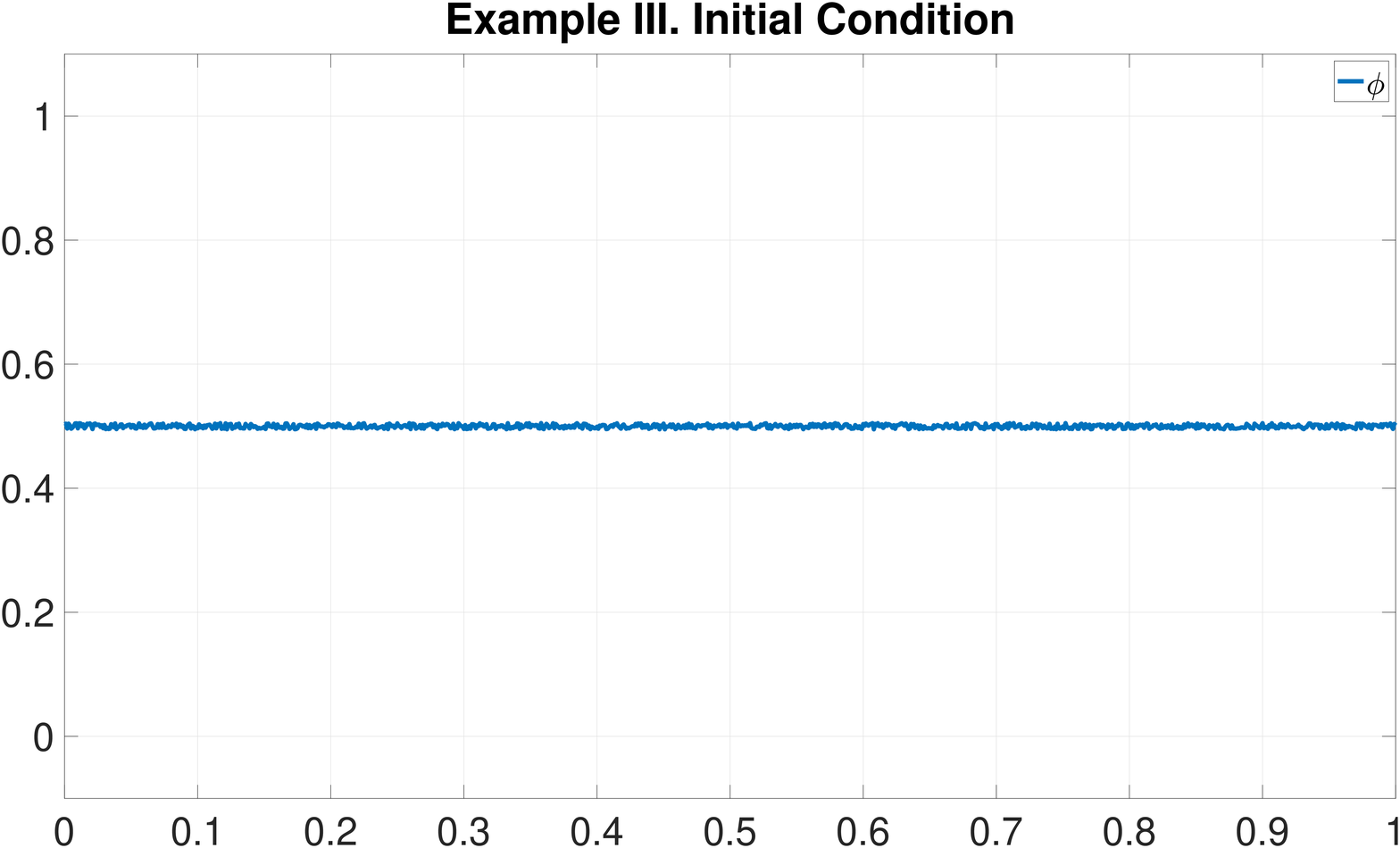}
\includegraphics[scale=0.114]{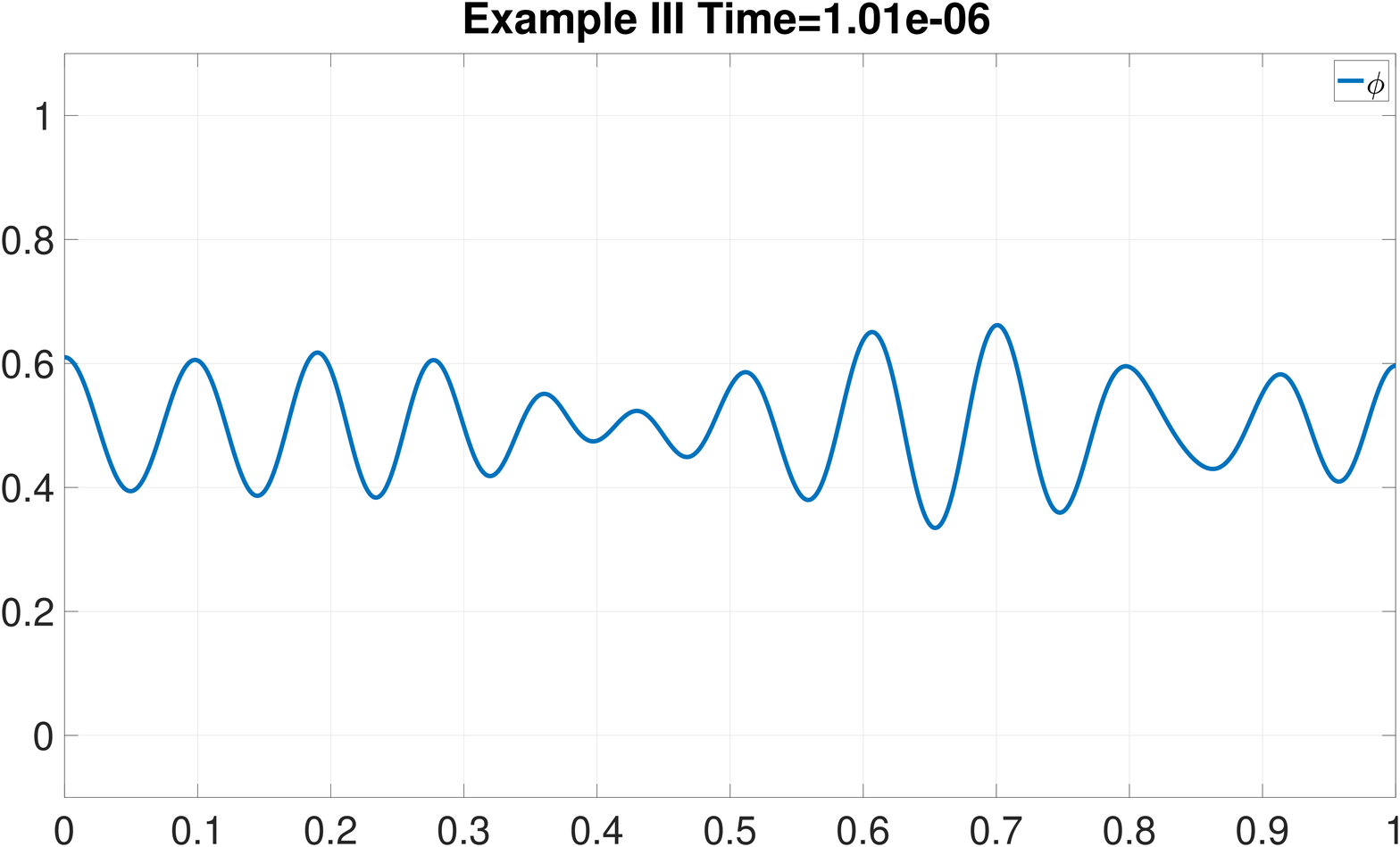}
\\
\includegraphics[scale=0.114]{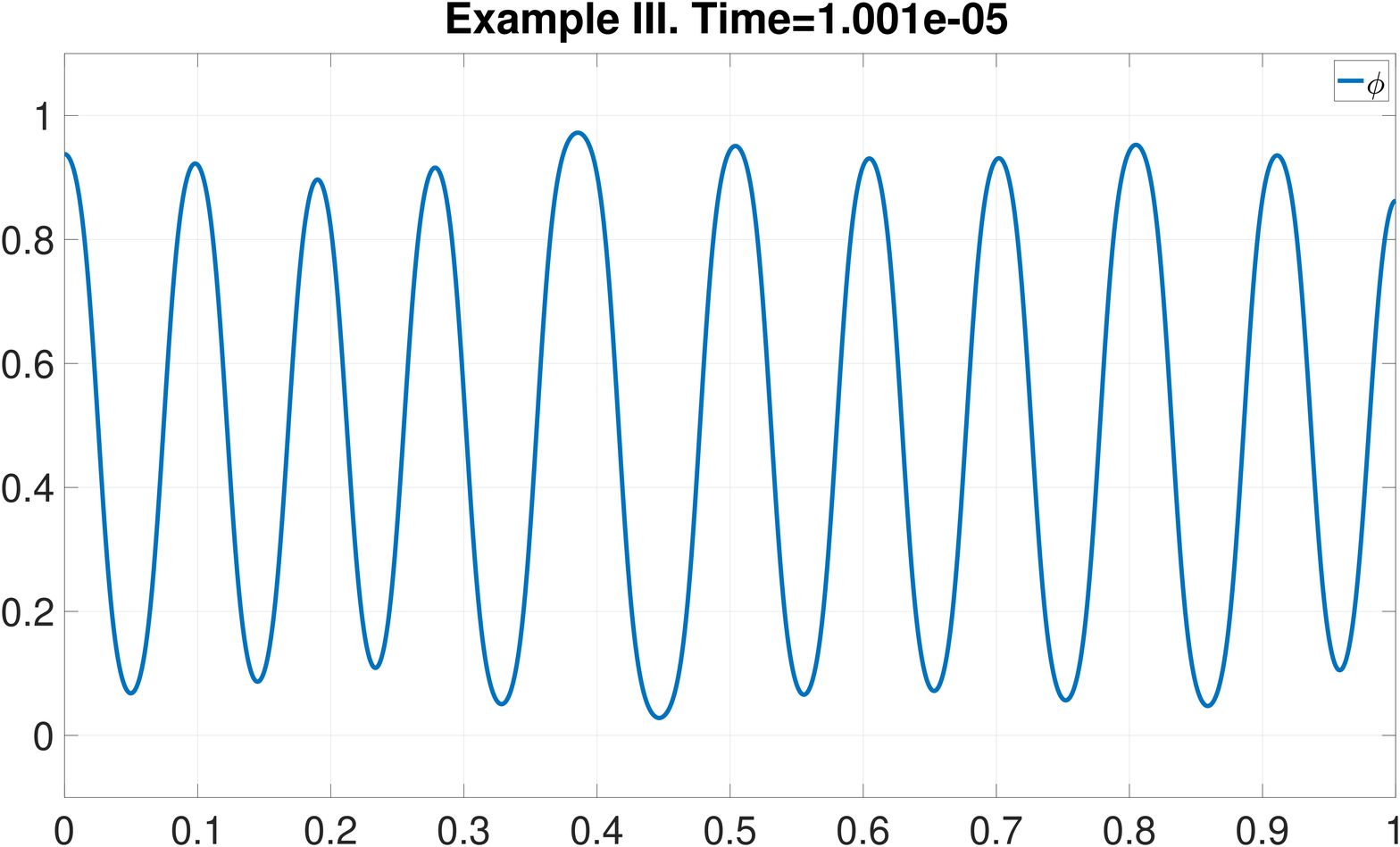}
\includegraphics[scale=0.114]{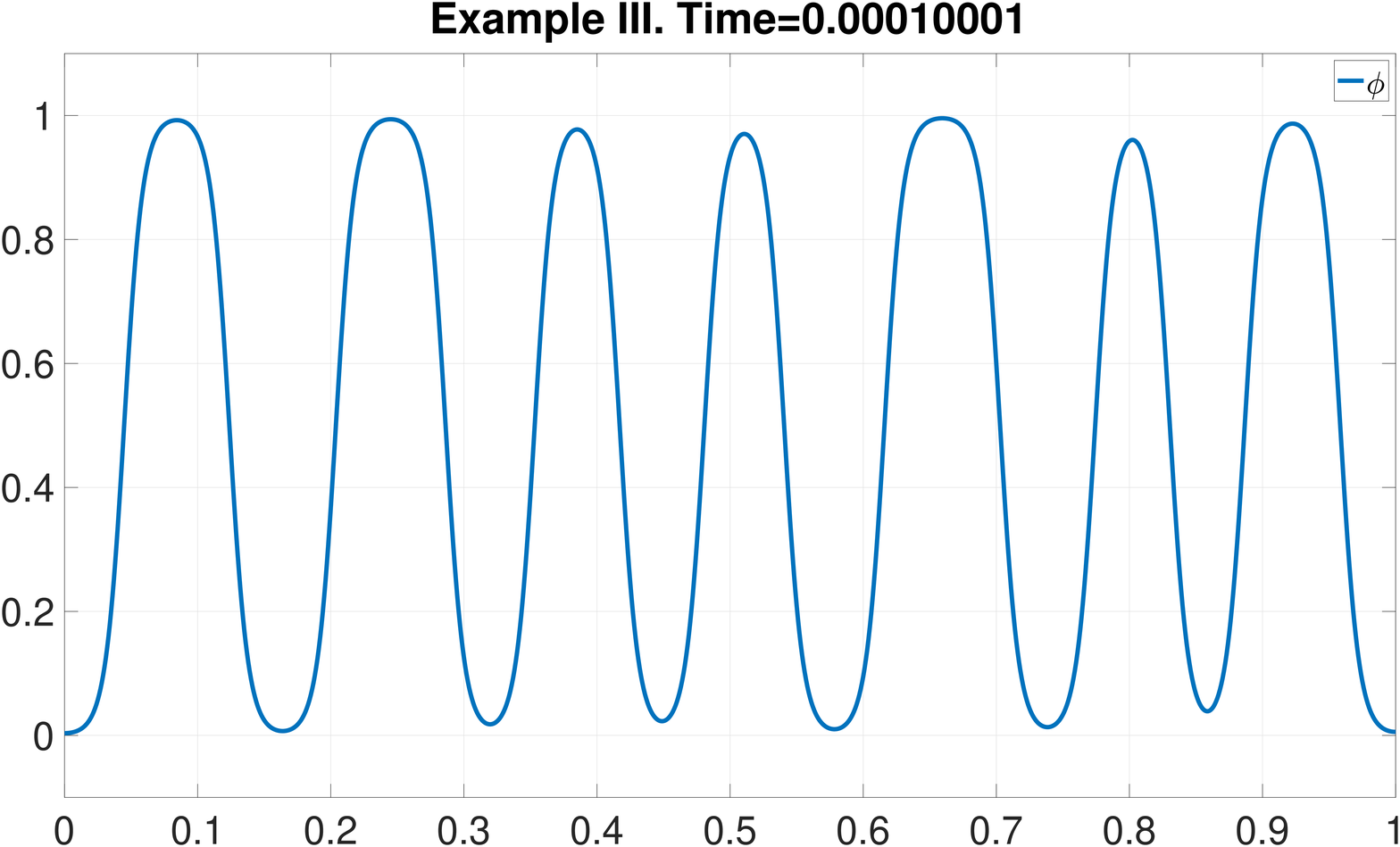}
\\
\includegraphics[scale=0.114]{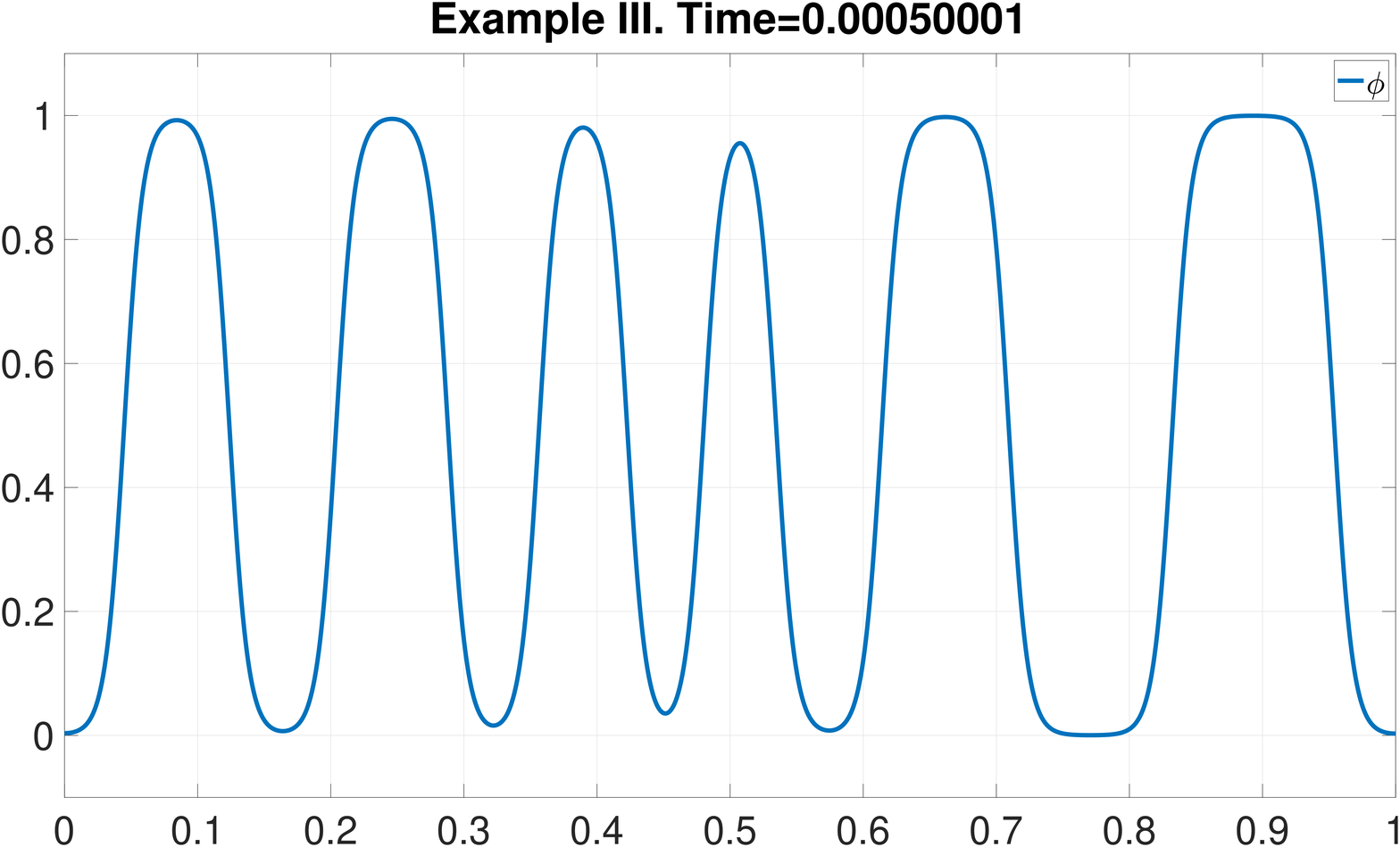}
\includegraphics[scale=0.114]{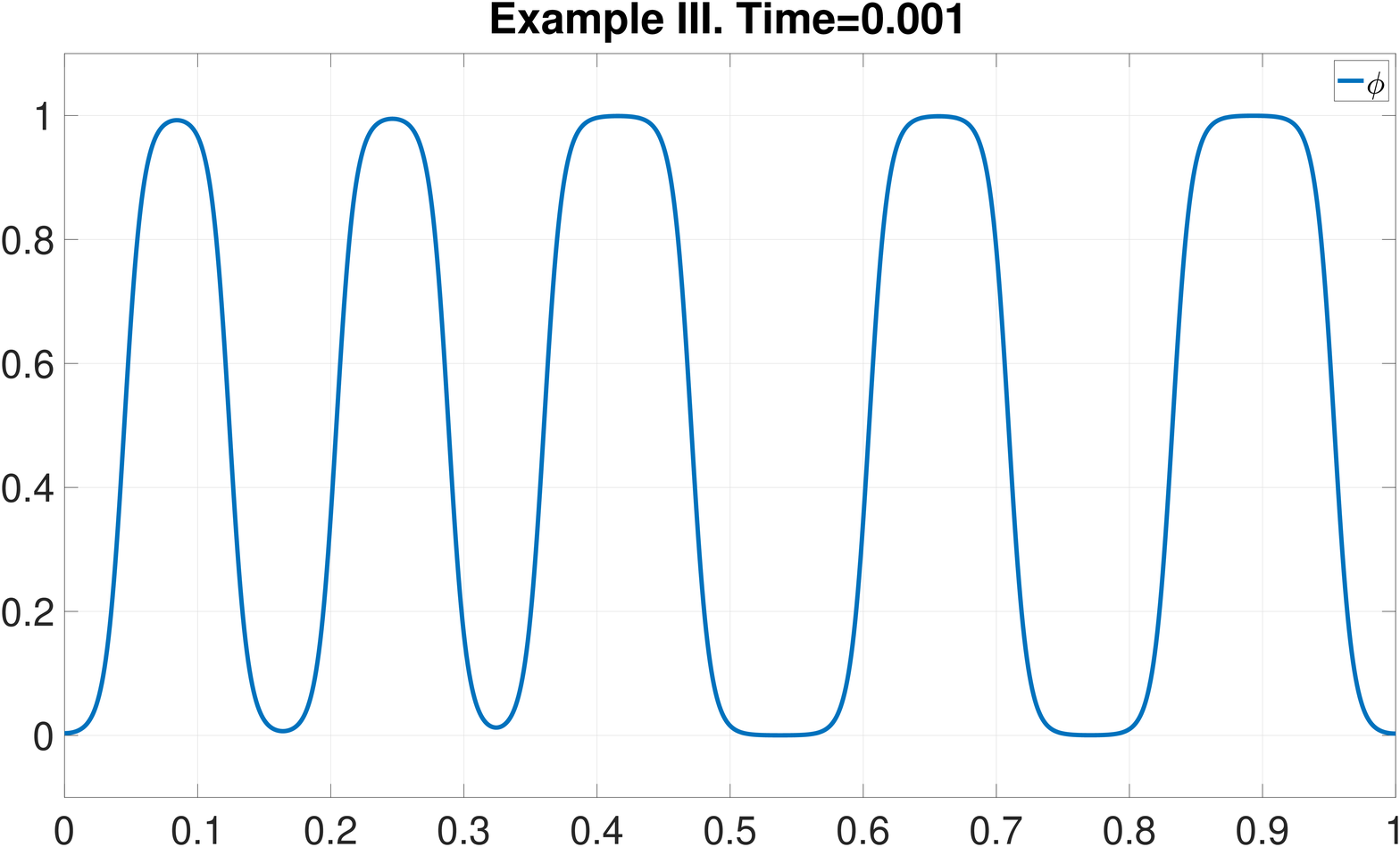}
\end{center}
\caption{Example III. Evolution in time of $\phi$ for G$_\varepsilon$-scheme}\label{fig:ExIII_dynamic}
\end{figure}

\begin{figure}
\begin{center}
\includegraphics[scale=0.114]{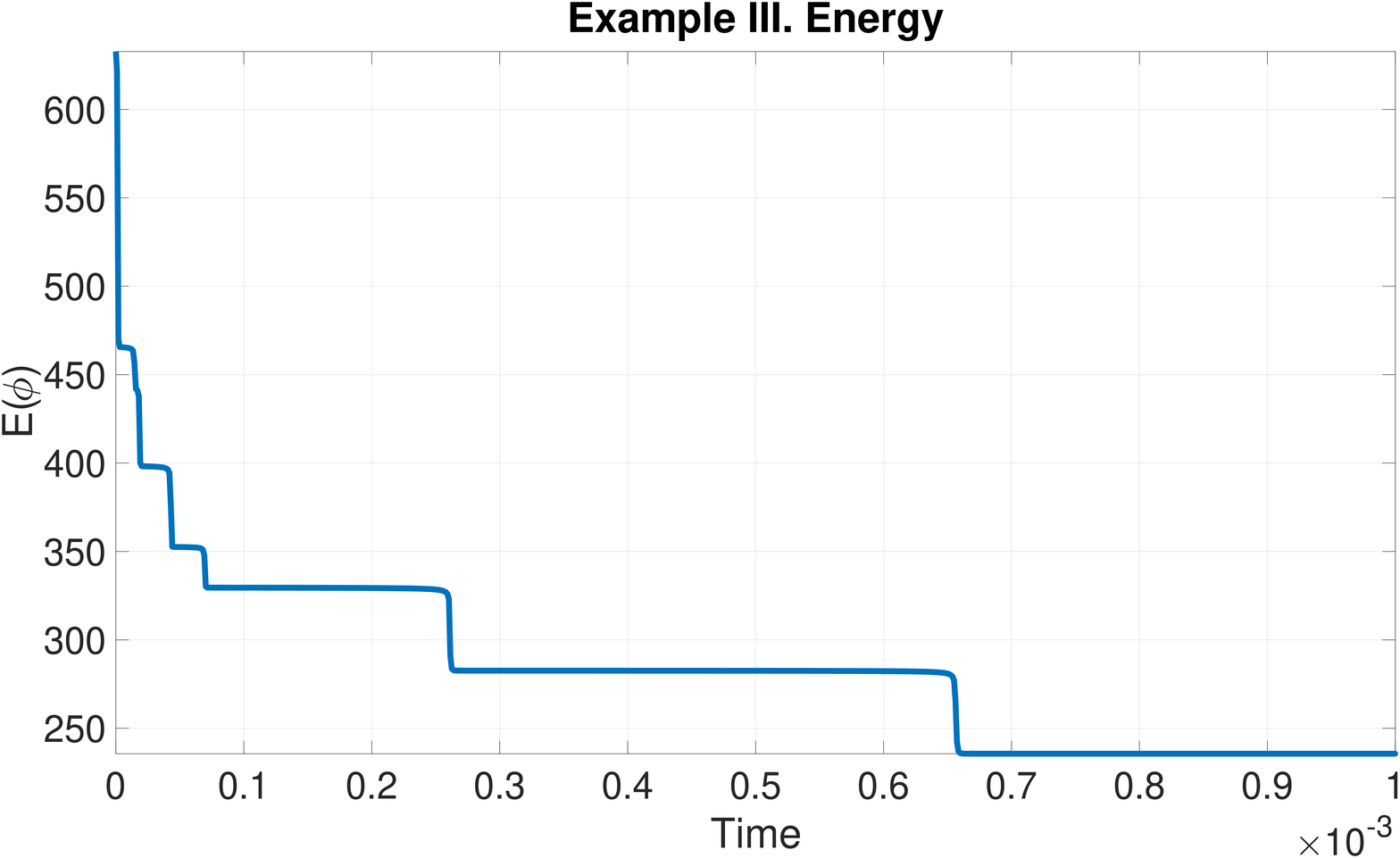}
\includegraphics[scale=0.114]{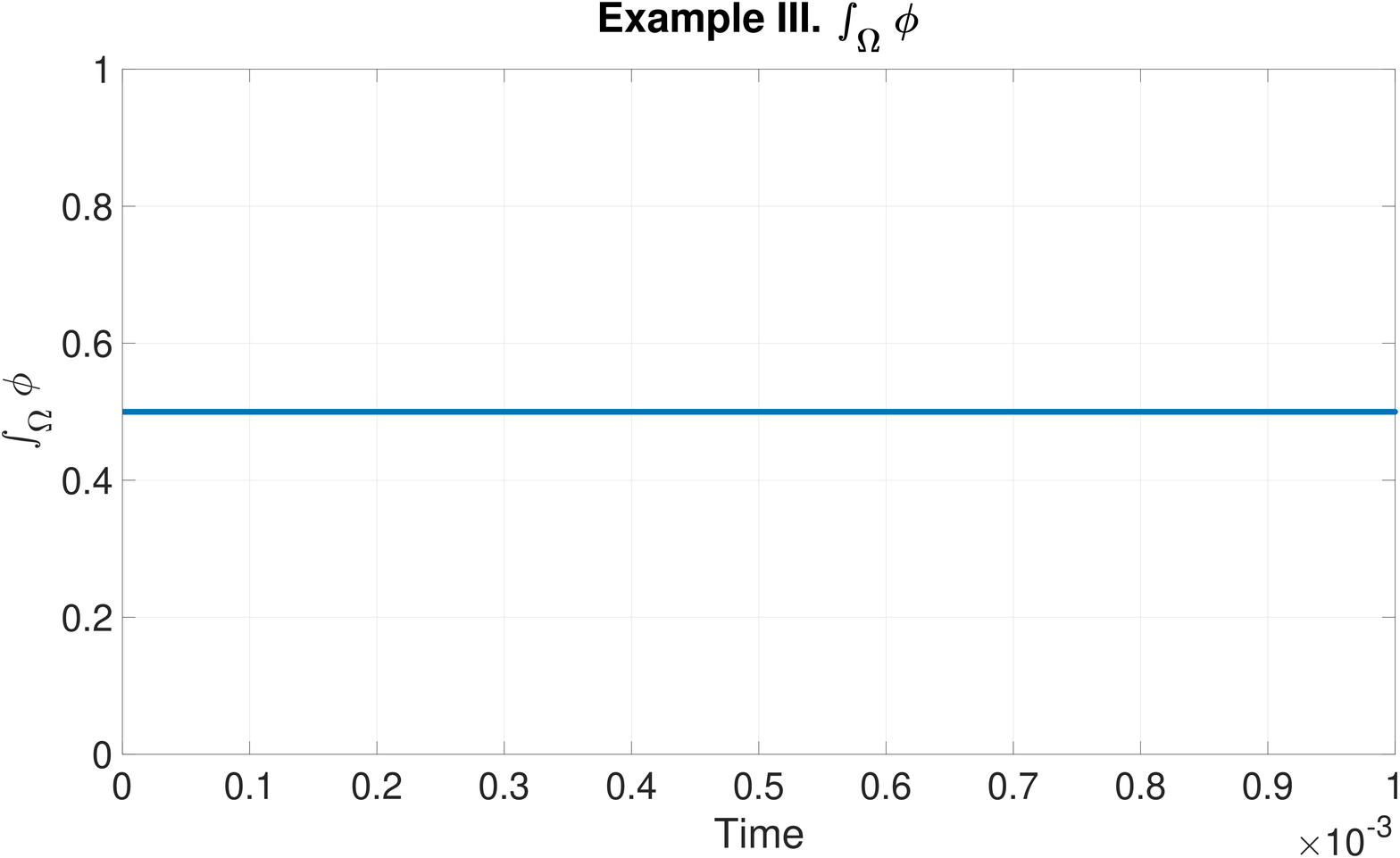}
\\
\includegraphics[scale=0.114]{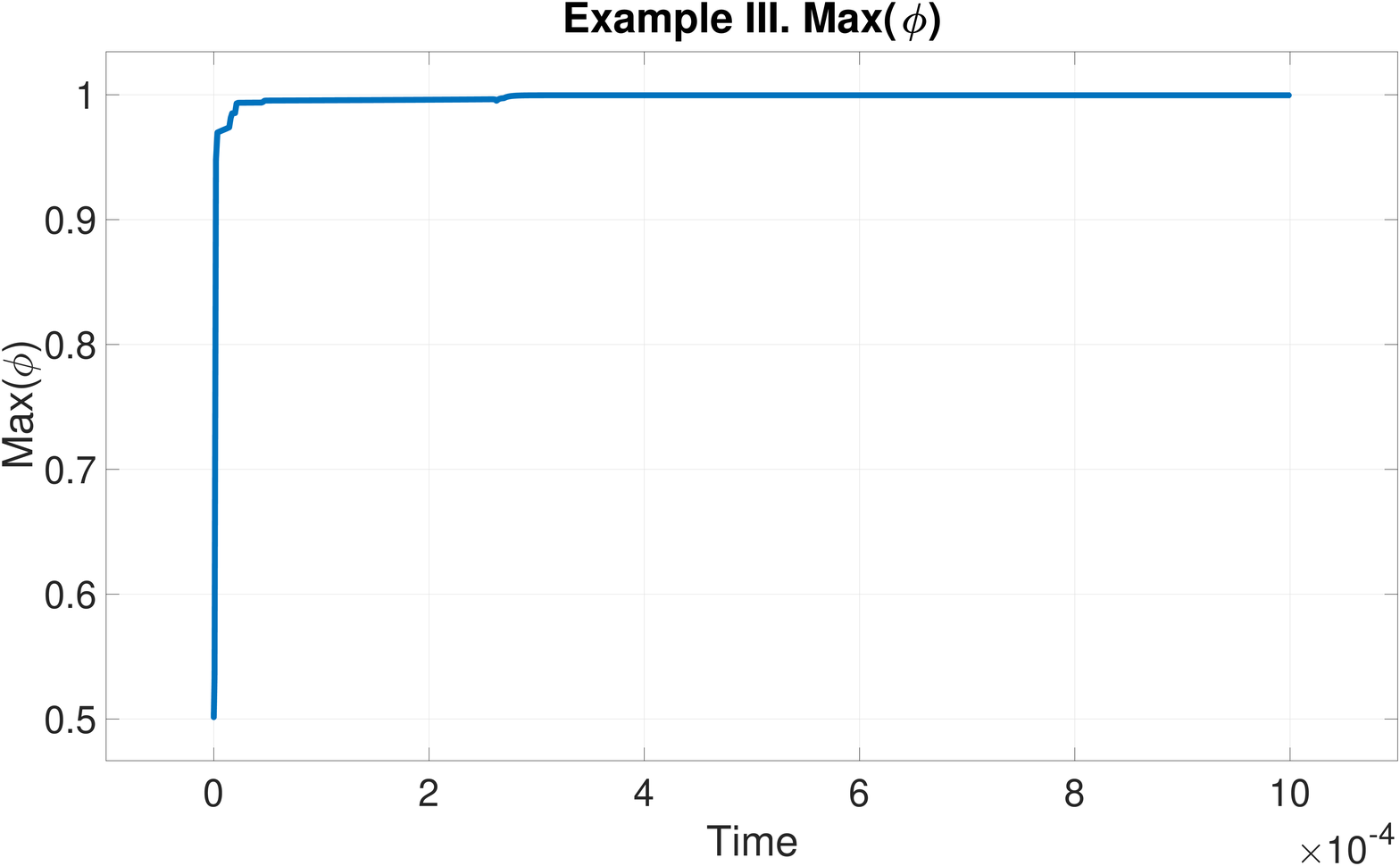}
\includegraphics[scale=0.114]{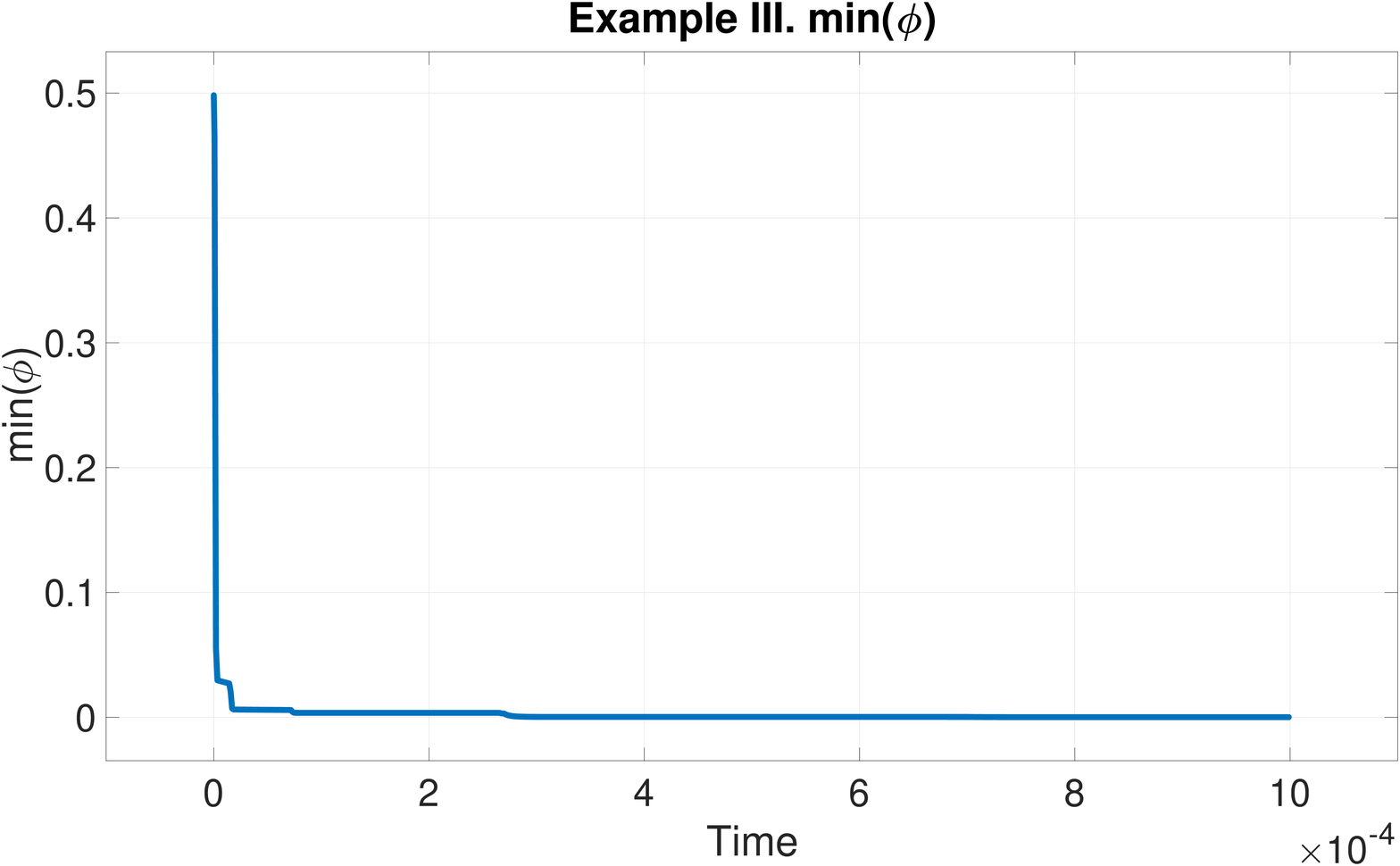}
\end{center}
\caption{Example III. Evolution in time of: $E(\phi)$ (top left), $\int_\Omega\phi dx$ (top right), $\max(\phi)$ (bottom left) and $\min(\phi)$ (bottom right).}\label{fig:ExIII_plots}
\end{figure}

\subsection{Example IV. Spinodal decomposition in $2D$}
In this section we perform simulations in the two dimensional domain $\Omega=[0,1]\times[0,1]$ in order to evidence that the effective implementation of the code in higher dimension can be performed in structured meshes and to compare the different dynamics obtained using constant and non-constant mobilities. 
To this end we have compared G$_\varepsilon$-scheme and J$_\varepsilon$-scheme (both with parameter $\varepsilon=10^{-8}$ and non-constant mobility) with M$_0$-scheme \eqref{eq:SchemeM0}  (with constant mobility $M_0(\phi)=1$).
 The time interval is $[0,5\times10^{-5}]$ and the rest of the considered parameters are $\eta=0.01$, $N=10^2$ and $\Delta t=10^{-9}$. 
As in the previous example, we have considered initially that the two components of our system are very mixed by taking as initial condition a random perturbation of amplitude $0.01$ of the constant value $\phi=0.5$ to simulate a spinodal decomposition dynamic. The dynamics of the three simulations are presented in Figure~\ref{fig:ExIV_dynamics}, and we can observe (as expected) that the separation process occurs faster in the case of considering a constant mobility (this fact is also illustrated in the results presented in \cite{DaiDu16}). 
\begin{figure}
\begin{center}
\includegraphics[scale=0.074]{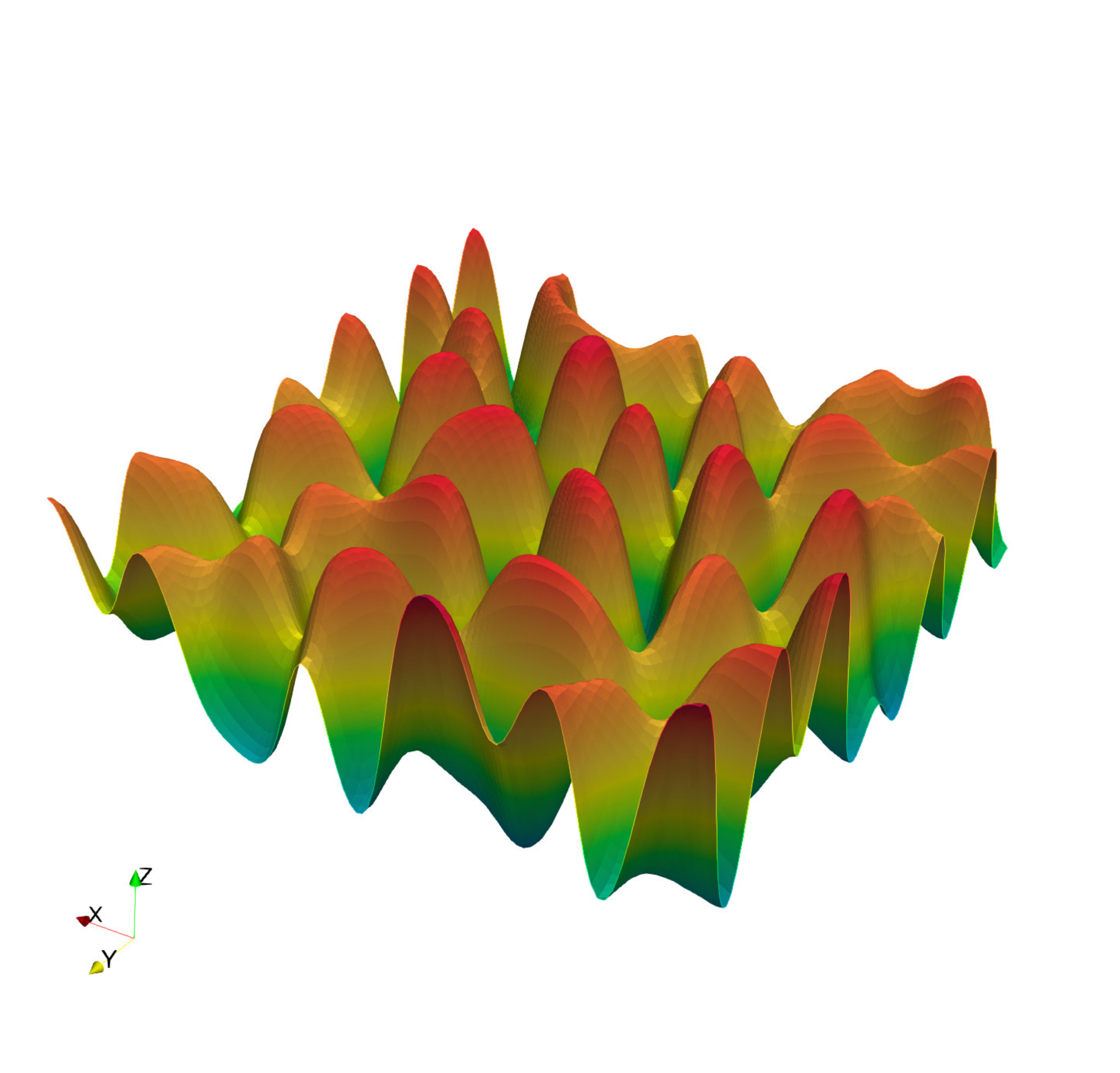}
\includegraphics[scale=0.074]{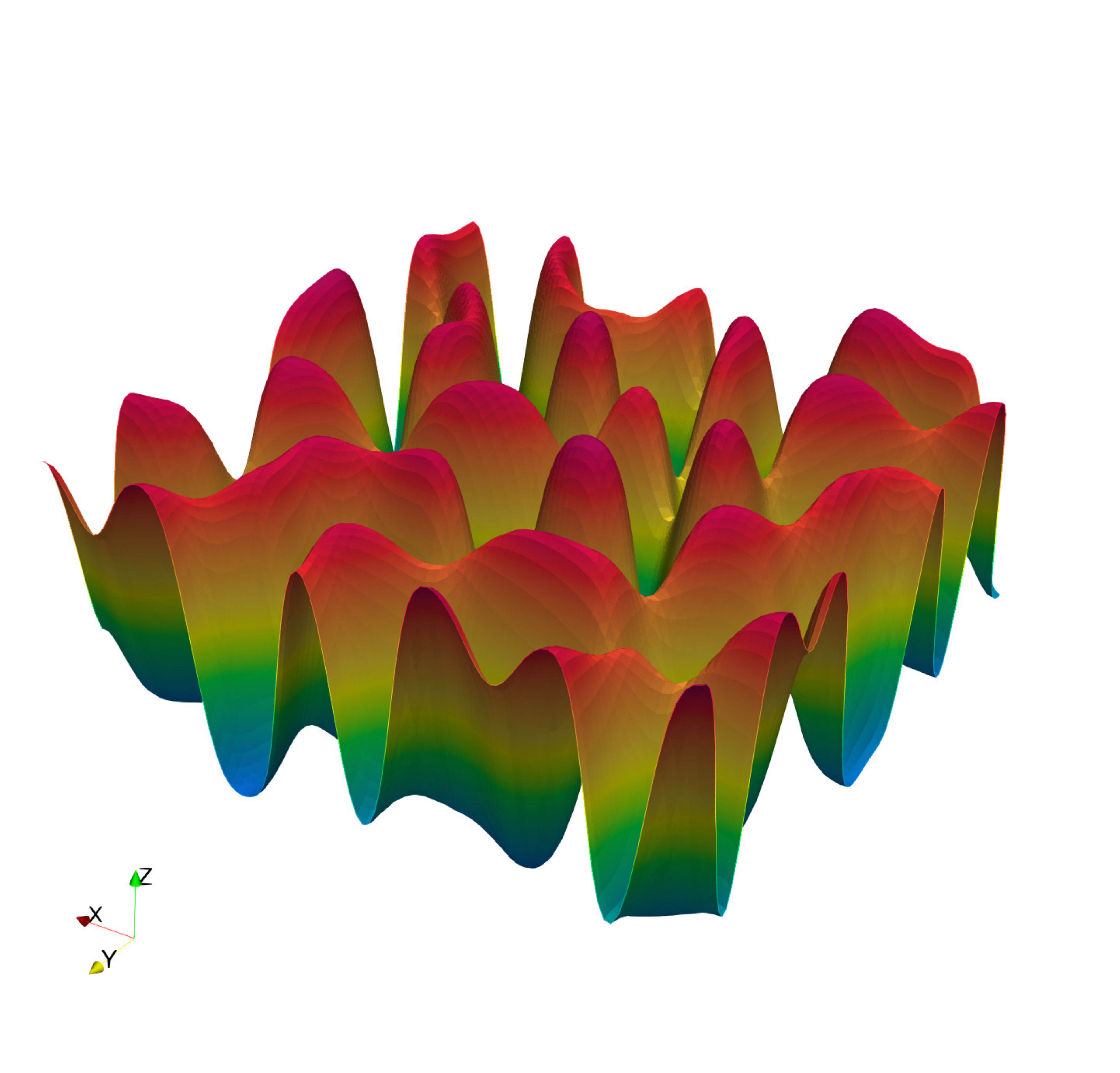}
\includegraphics[scale=0.074]{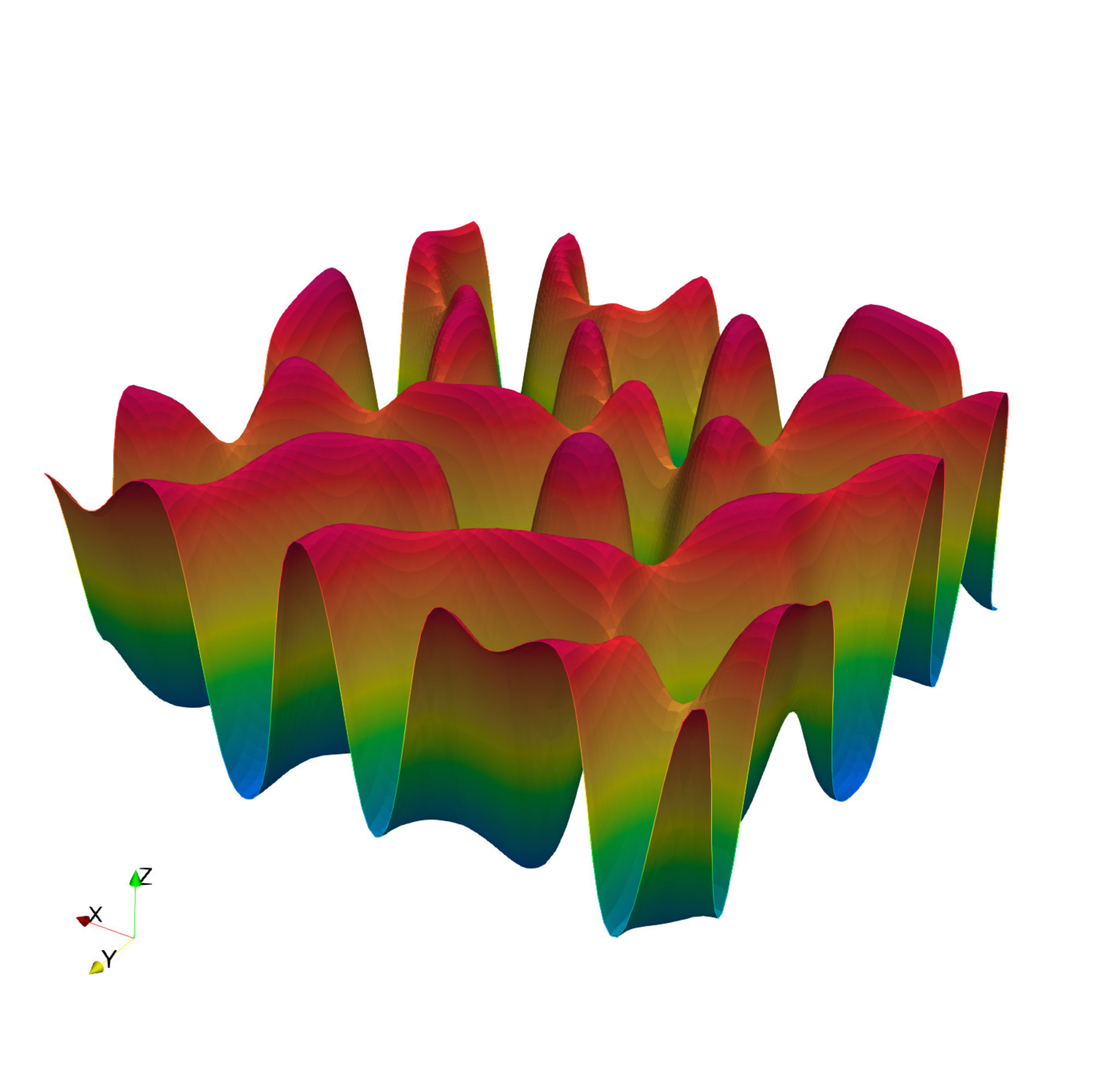}
\includegraphics[scale=0.074]{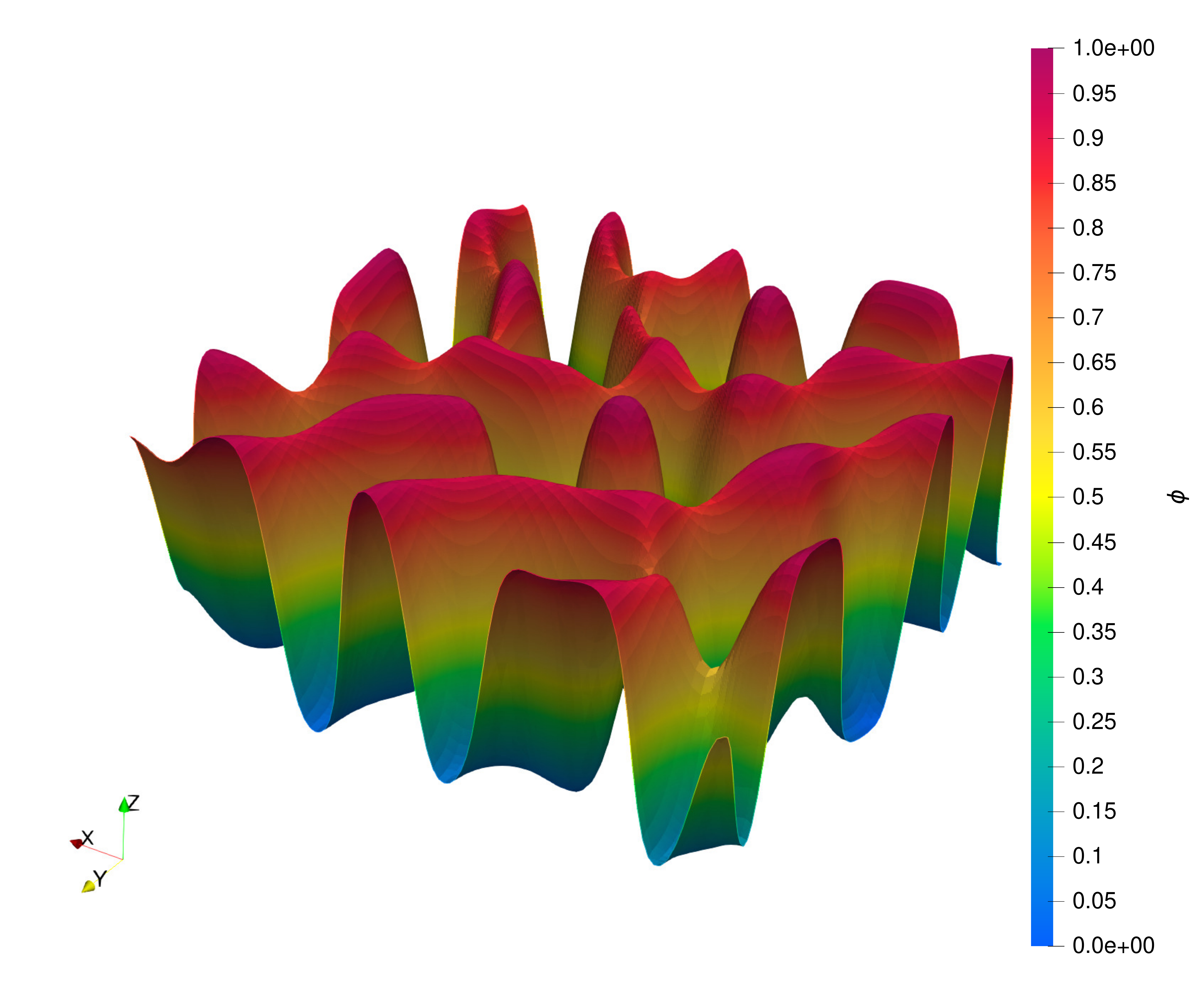}
\\
\includegraphics[scale=0.074]{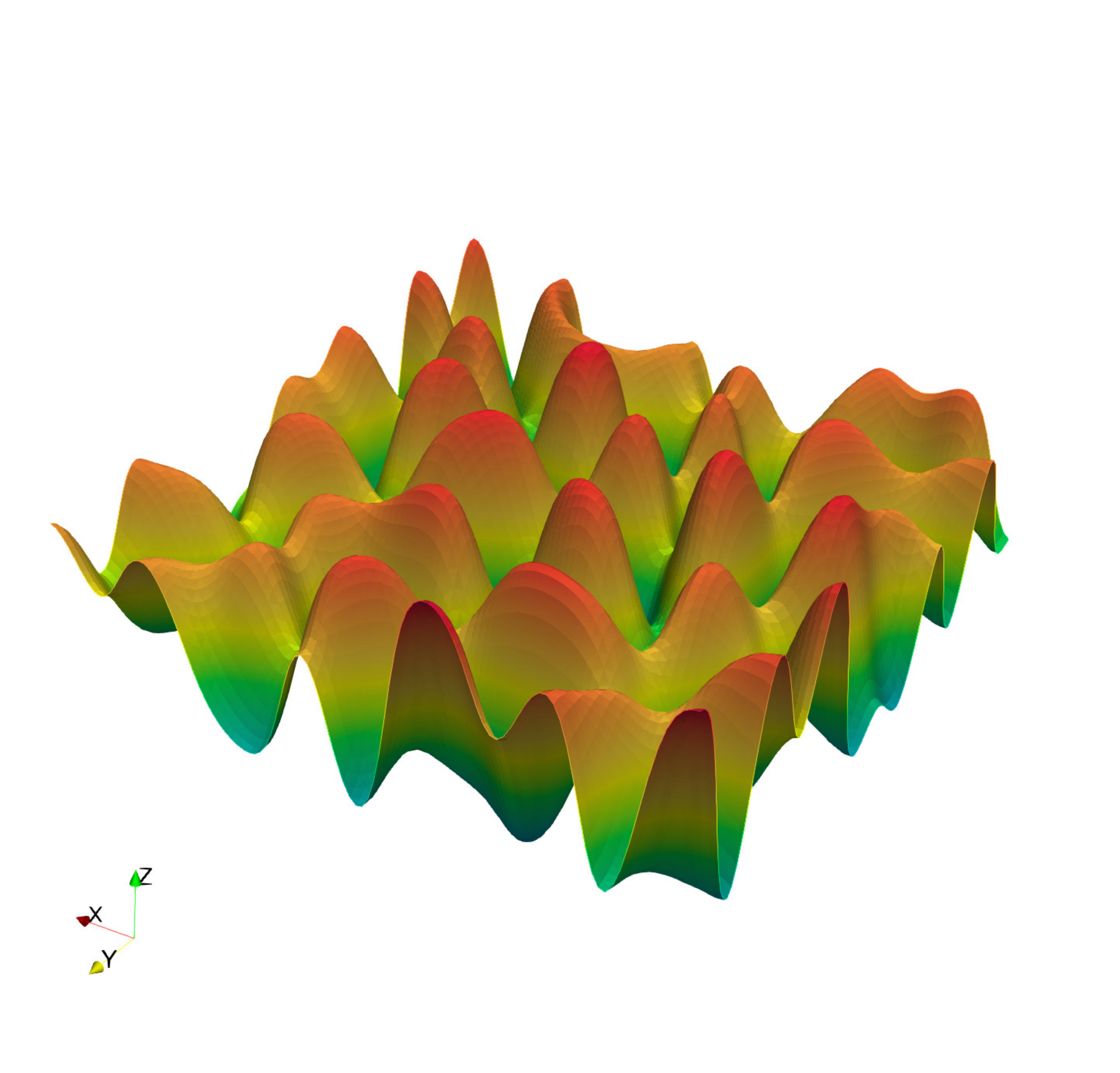}
\includegraphics[scale=0.074]{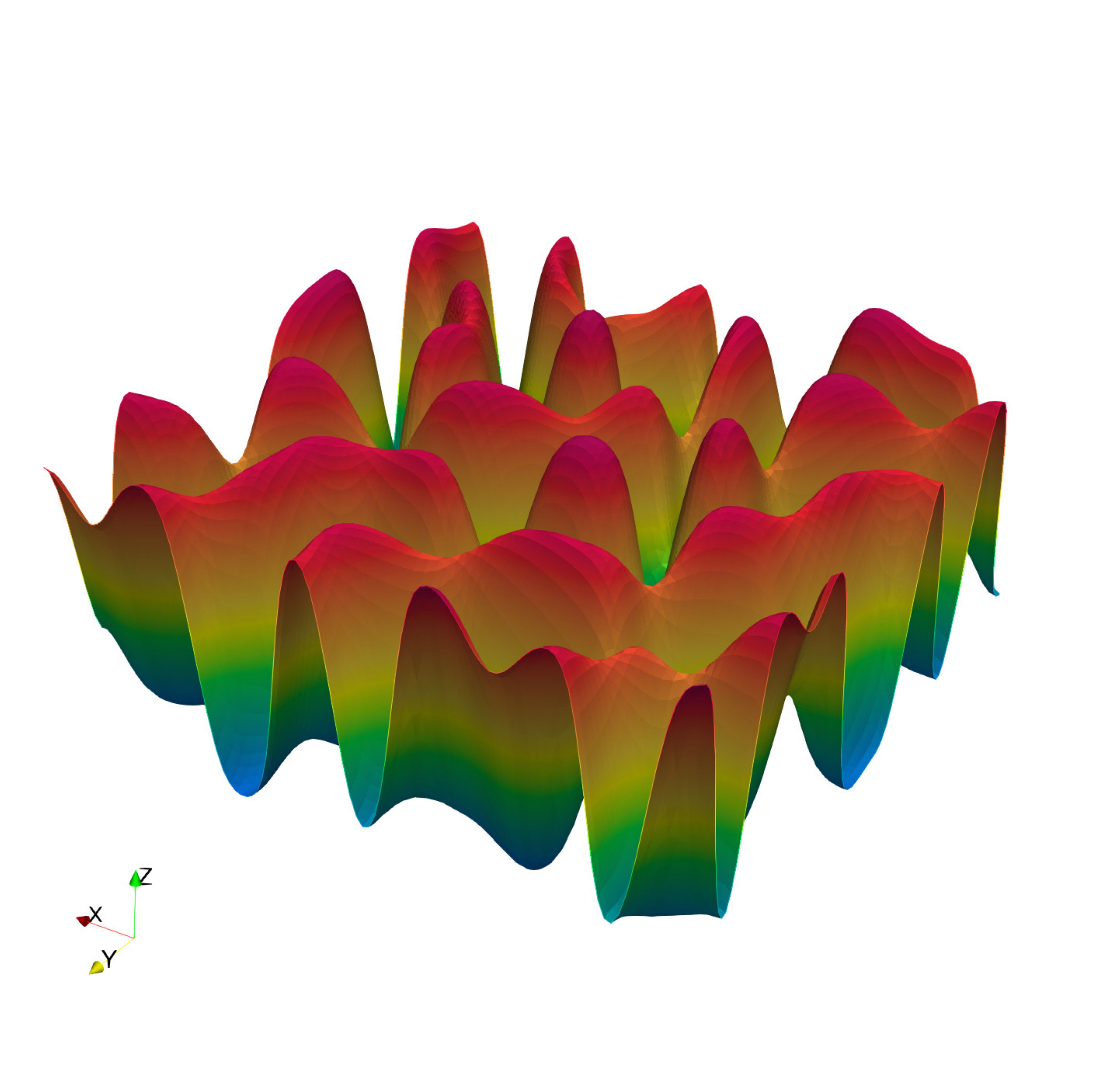}
\includegraphics[scale=0.074]{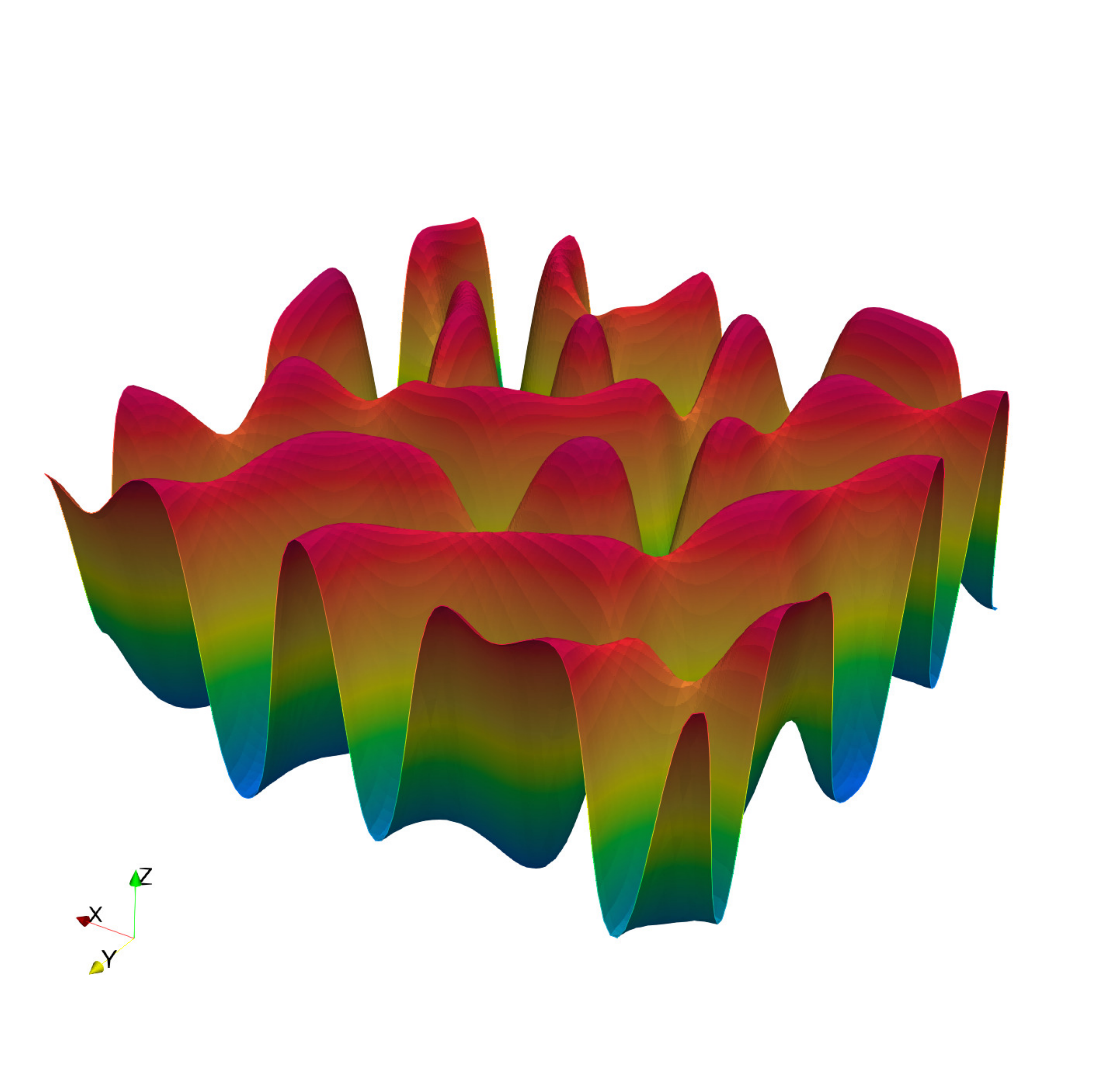}
\includegraphics[scale=0.074]{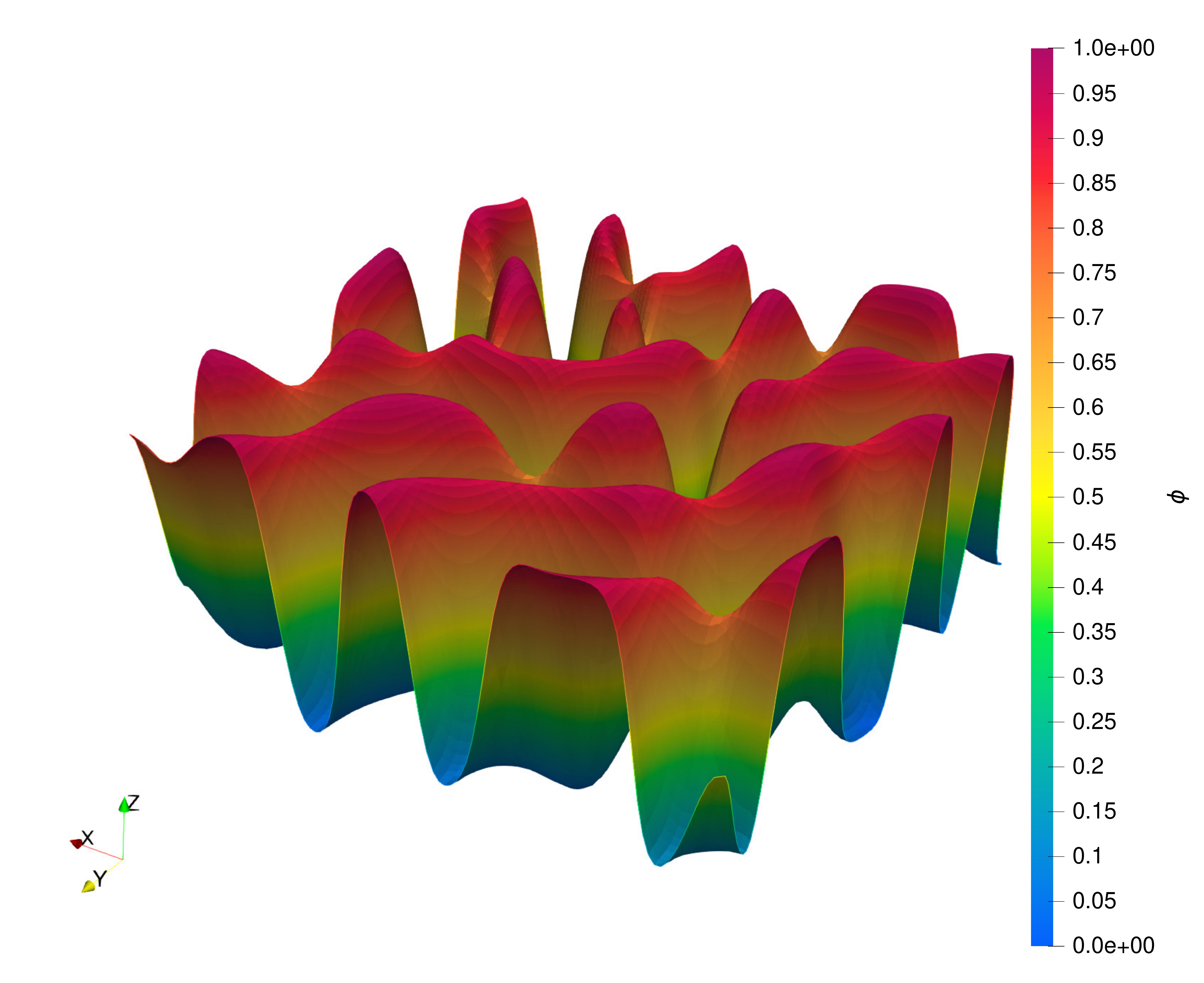}
\\
\includegraphics[scale=0.074]{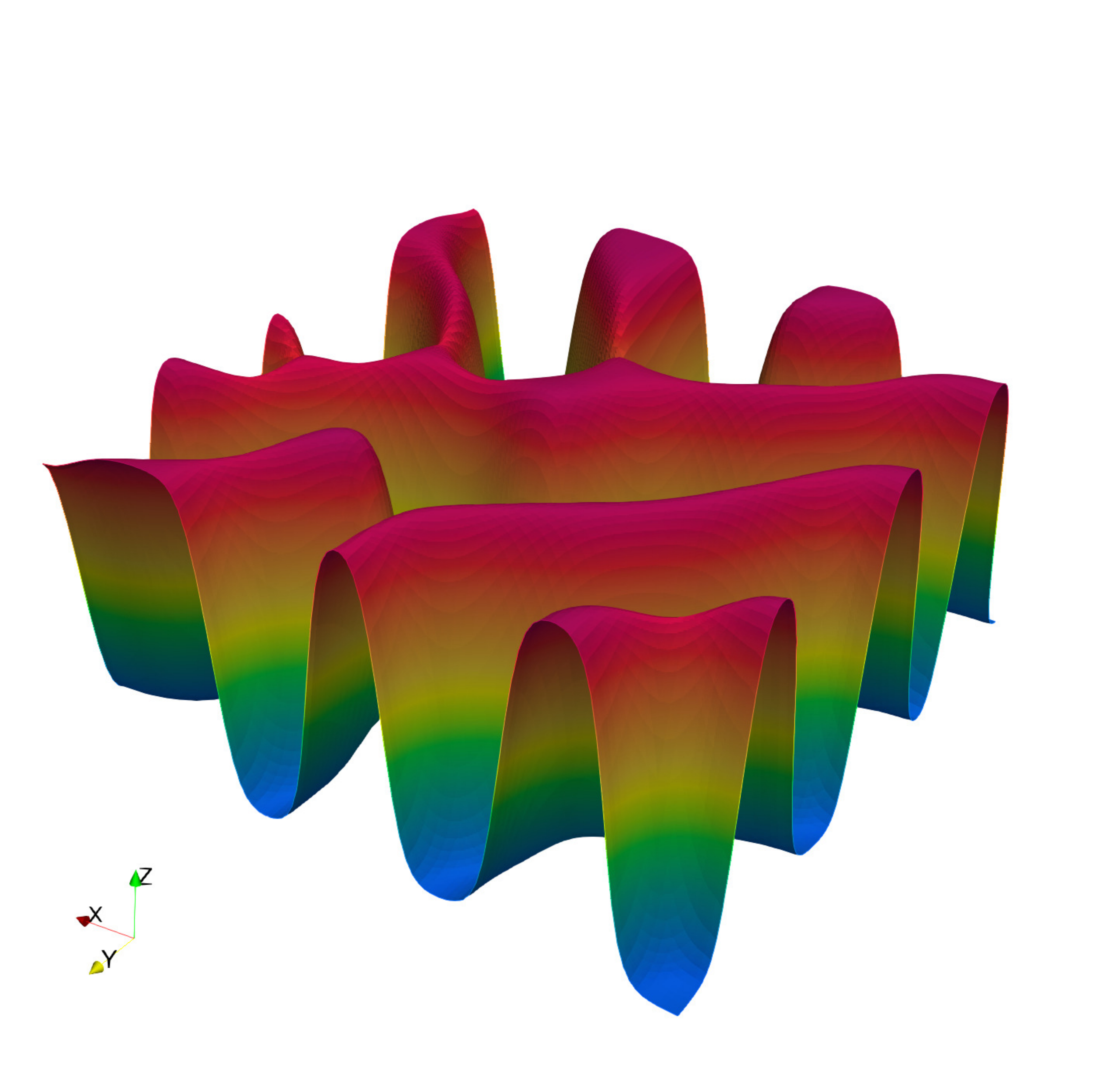}
\includegraphics[scale=0.074]{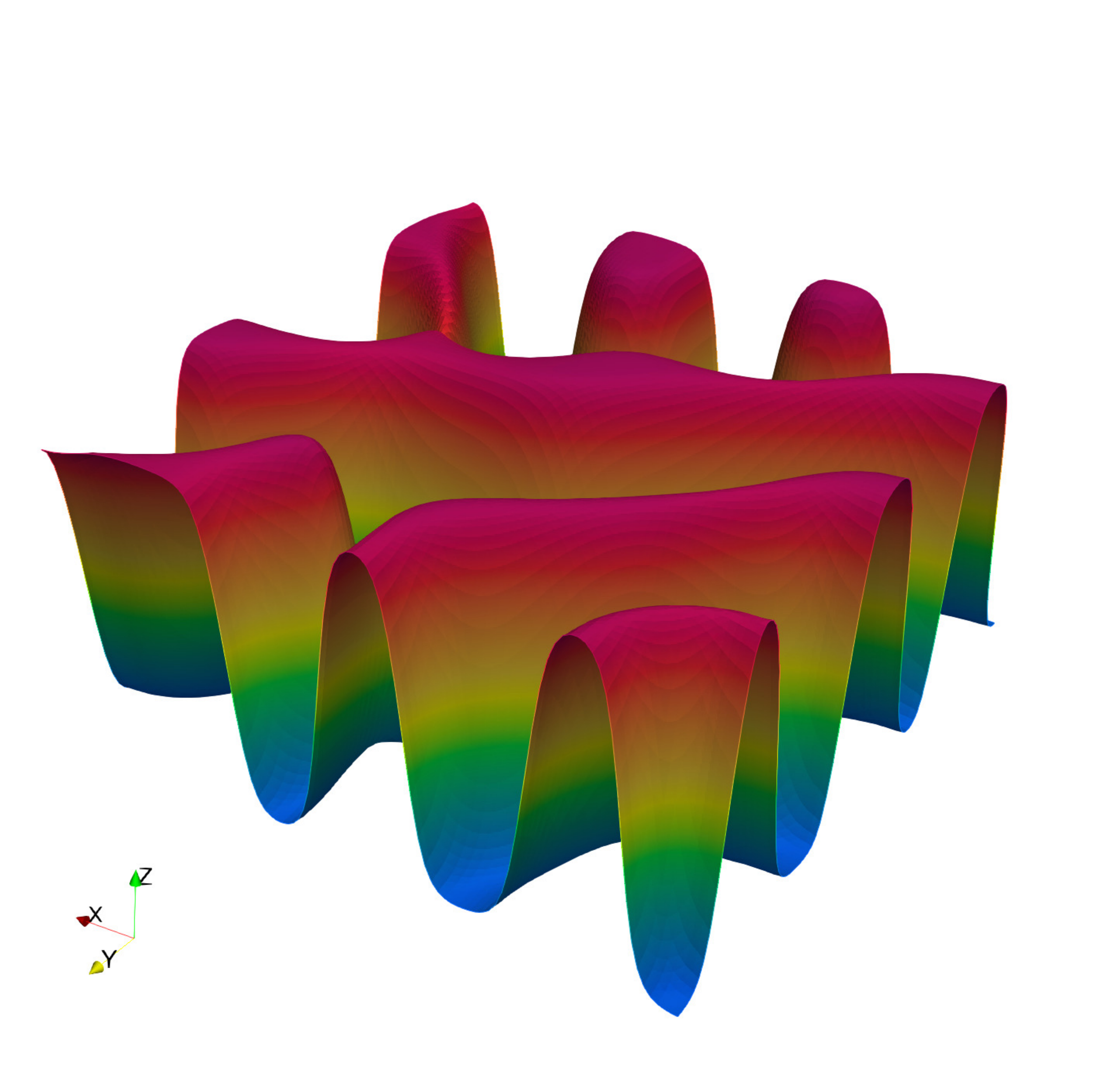}
\includegraphics[scale=0.074]{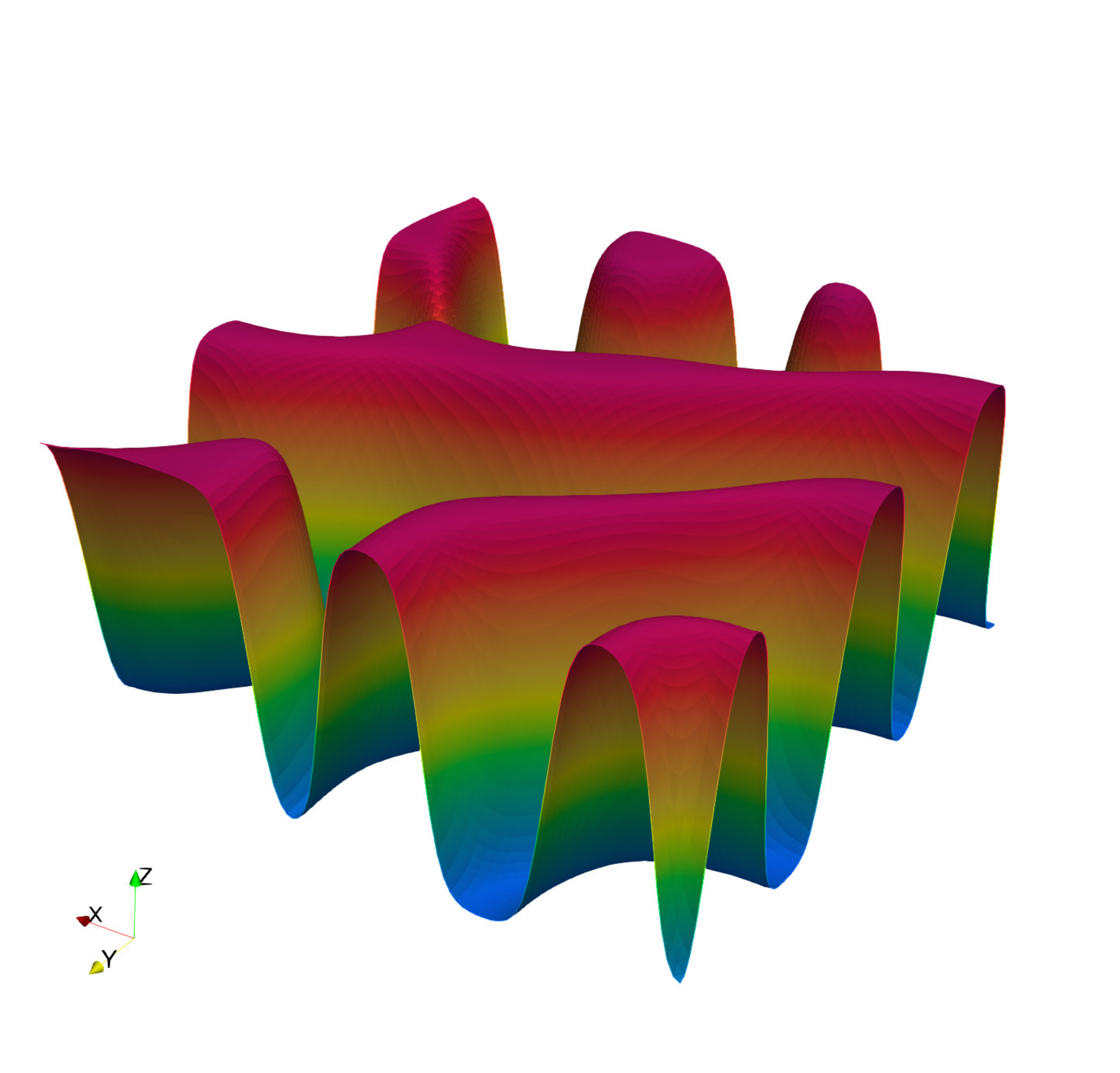}
\includegraphics[scale=0.074]{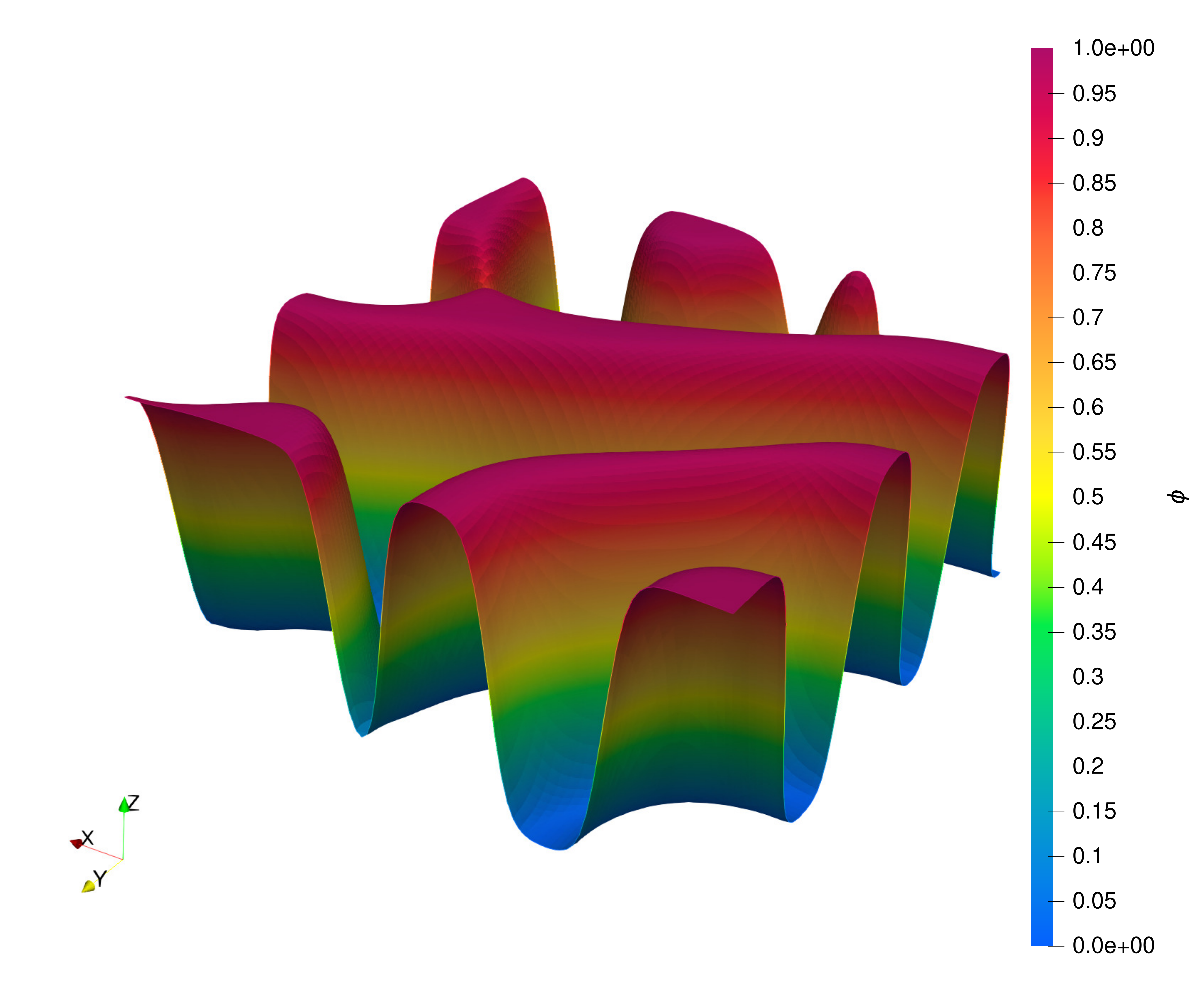}
\end{center}
\caption{Example IV. Evolution in time (from left to right) of $\phi$ at times $t= 2\times10^{-5}, 3\times10^{-5}, 4\times10^{-5}$ and $5\times10^{-5}$. Top row: G$_\varepsilon$-scheme. Middle row: J$_\varepsilon$-scheme. Bottom row: Scheme with $M(\phi)=1$.}\label{fig:ExIV_dynamics}
\end{figure}

\section{Conclusions}\label{sec:conclusions}

In this work we have derived two new numerical schemes to approximate the Cahn-Hilliard equation with degenerate mobility. The main ideas to derive the schemes are to first truncate the mobility term away from the values $\varepsilon $ and $1-\varepsilon $ and then realize that apart of the classical energy law, the system also satisfy estimates for singular potentials which can be used to derive estimates about the boundedness of the variable $\phi$ in terms of the truncation parameter $\varepsilon$. In particular, the derived estimates for each of the numerical schemes are associated to different singular potentials.

The resulting numerical schemes have been implemented and compared with the standard approach of truncating the mobility by zero (to avoid negative flux even if $\phi$ goes outside of the interval $[0,1]$) and it has been shown that the new schemes behave much better in terms of achieving the desired bounds for $\phi$. Moreover, the new schemes have been shown to obtain optimal order of convergence and to be able to capture the most challenging of the benchmarks for the Cahn-Hilliard equation, that is, to simulate properly the spinodal decomposition in one and two dimensions, evidencing that these ideas are  valid independently of the spatial dimension.

\section*{Acknowledgements}
This work has been partially supported by Grant PGC2018-098308-B-I00 (MCI/AEI/FEDER, UE, Spain). FGG has also been financed in part by the Grant US-1381261 (US/JUNTA/FEDER, UE, Spain) and Grant P20-01120 (PAIDI/JUNTA/FEDER, UE, Spain)

\end{document}